\documentclass[10pt]{article}
\usepackage{hyperref}
\usepackage{amsmath}
\usepackage{amsfonts}
\usepackage{amssymb}
\usepackage{amsthm}

\usepackage[shortlabels]{enumitem}
\usepackage{verbatim}
\usepackage{graphicx}
\usepackage{microtype}
\usepackage{fullpage}
\usepackage{xparse}
\usepackage{mathrsfs}
\usepackage[capitalise]{cleveref}
\usepackage{mleftright}

\crefname{thm}{Theorem}{Theorems}
\Crefname{thm}{Theorem}{Theorems}
\crefname{conj}{Conjecture}{Conjectures}
\Crefname{conj}{Conjecture}{Conjectures}
\crefname{prop}{Proposition}{Propositions}
\Crefname{prop}{Proposition}{Propositions}
\crefname{cor}{Corollary}{Corollaries}
\Crefname{cor}{Corollary}{Corollaries}
\crefname{defn}{Definition}{Definitions}
\Crefname{defn}{Definition}{Definitions}
\crefname{rmk}{Remark}{Remarks}
\Crefname{rmk}{Remark}{Remarks}
\crefname{prob}{Problem}{Problems}
\Crefname{prob}{Problem}{Problems}

\crefname{enumi}{}{}
\Crefname{enumi}{}{}

\crefname{figure}{Figure}{Figures}
\Crefname{figure}{Figure}{Figures}

\crefformat{equation}{(#2#1#3)}
\crefrangeformat{equation}{(#3#1#4) to~(#5#2#6)}
\crefmultiformat{equation}{(#2#1#3)}%
{ and~(#2#1#3)}{, (#2#1#3)}{, and~(#2#1#3)}

\Crefformat{equation}{(#2#1#3)}
\Crefrangeformat{equation}{(#3#1#4) to~(#5#2#6)}
\Crefmultiformat{equation}{(#2#1#3)}%
{ and~(#2#1#3)}{, (#2#1#3)}{, and~(#2#1#3)}

\begin{document}

\NewDocumentCommand{\C}{}{{\mathbb{C}}}
\NewDocumentCommand{\R}{}{{\mathbb{R}}}
\NewDocumentCommand{\Q}{}{{\mathbb{Q}}}
\NewDocumentCommand{\Z}{}{{\mathbb{Z}}}
\NewDocumentCommand{\N}{}{{\mathbb{N}}}
\NewDocumentCommand{\M}{}{{\mathbb{M}}}
\NewDocumentCommand{\grad}{}{\nabla}
\NewDocumentCommand{\sA}{}{\mathcal{A}}
\NewDocumentCommand{\sF}{}{\mathcal{F}}
\NewDocumentCommand{\sH}{}{\mathcal{H}}
\NewDocumentCommand{\sD}{}{\mathcal{D}}
\NewDocumentCommand{\sB}{}{\mathcal{B}}
\NewDocumentCommand{\sC}{}{\mathcal{C}}
\NewDocumentCommand{\sE}{}{\mathcal{E}}
\NewDocumentCommand{\sL}{}{\mathcal{L}}
\NewDocumentCommand{\sT}{}{\mathcal{T}}
\NewDocumentCommand{\sO}{}{\mathcal{O}}
\NewDocumentCommand{\sP}{}{\mathcal{P}}
\NewDocumentCommand{\sQ}{}{\mathcal{Q}}
\NewDocumentCommand{\sR}{}{\mathcal{R}}
\NewDocumentCommand{\sM}{}{\mathcal{M}}
\NewDocumentCommand{\sI}{}{\mathcal{I}}
\NewDocumentCommand{\sK}{}{\mathcal{K}}
\NewDocumentCommand{\Span}{}{\mathrm{span}}
\NewDocumentCommand{\fM}{}{\mathfrak{M}}
\NewDocumentCommand{\fN}{}{\mathfrak{N}}
\NewDocumentCommand{\fX}{}{\mathfrak{X}}
\NewDocumentCommand{\fY}{}{\mathfrak{Y}}
\NewDocumentCommand{\gammat}{}{\tilde{\gamma}}
\NewDocumentCommand{\ct}{}{\tilde{c}}
\NewDocumentCommand{\bt}{}{\tilde{b}}
\NewDocumentCommand{\ch}{}{\hat{c}}
\NewDocumentCommand{\Ut}{}{\tilde{U}}
\NewDocumentCommand{\Gt}{}{\widetilde{G}}
\NewDocumentCommand{\Vt}{}{\tilde{V}}
\NewDocumentCommand{\ah}{}{\hat{a}}
\NewDocumentCommand{\at}{}{\tilde{a}}
\NewDocumentCommand{\Yh}{}{\widehat{Y}}
\NewDocumentCommand{\Yt}{}{\widetilde{Y}}
\NewDocumentCommand{\Ah}{}{\widehat{A}}
\NewDocumentCommand{\Ch}{}{\widehat{C}}
\NewDocumentCommand{\At}{}{\widetilde{A}}
\NewDocumentCommand{\Vol}{m}{\mathrm{Vol}(#1)}
\NewDocumentCommand{\BVol}{m}{\mathrm{Vol}\left(#1\right)}
\NewDocumentCommand{\fg}{}{\mathfrak{g}}
\NewDocumentCommand{\Div}{}{\mathrm{div}}
\NewDocumentCommand{\Phih}{}{\widehat{\Phi}}
\NewDocumentCommand{\Phit}{}{\widetilde{\Phi}}

\NewDocumentCommand{\Deriv}{}{\mathscr{D}}
\NewDocumentCommand{\BofA}{}{\mathscr{B}}
\NewDocumentCommand{\ADeriv}{}{\mathscr{A}}

\NewDocumentCommand{\Xa}{}{X^{(\alpha)}}
\NewDocumentCommand{\Va}{}{V^{(\alpha)}}

\NewDocumentCommand{\transpose}{}{\top}

\NewDocumentCommand{\ICond}{}{\sC}

\NewDocumentCommand{\LebDensity}{}{\sigma_{\mathrm{Leb}}}


\NewDocumentCommand{\Lie}{m}{\sL_{#1}}

\NewDocumentCommand{\ZygSymb}{}{\mathscr{C}}

\NewDocumentCommand{\Zyg}{m o}{\IfNoValueTF{#2}{\ZygSymb^{#1}}{\ZygSymb^{#1}(#2) }}
\NewDocumentCommand{\ZygX}{m m o}{\IfNoValueTF{#3}{\ZygSymb^{#2}_{#1}}{\ZygSymb^{#2}_{#1}(#3) }}

\NewDocumentCommand{\CSpace}{m o}{\IfNoValueTF{#2}{C(#1)}{C(#1;#2)}}

\NewDocumentCommand{\CjSpace}{m o o}{\IfNoValueTF{#2}{C^{#1}}{ \IfNoValueTF{#3}{ C^{#1}(#2)}{C^{#1}(#2;#3) } }  }

\NewDocumentCommand{\CXjSpace}{m m o}{\IfNoValueTF{#3}{C^{#2}_{#1}}{ C^{#2}_{#1}(#3) } }

\NewDocumentCommand{\HSpace}{m m o o}{\IfNoValueTF{#3}{C^{#1,#2}}{ \IfNoValueTF{#4} {C^{#1,#2}(#3)} {C^{#1,#2}(#3;#4)} }}

\NewDocumentCommand{\HXSpace}{m m m o}{\IfNoValueTF{#4}{C_{#1}^{#2,#3}}{  {C_{#1}^{#2,#3}(#4)}  }}

\NewDocumentCommand{\ZygSpace}{m o o}{\IfNoValueTF{#2}{\ZygSymb^{#1}}{ \IfNoValueTF{#3} { \ZygSymb^{#1}(#2) }{\ZygSymb^{#1}(#2;#3) } } }

\NewDocumentCommand{\ZygXSpace}{m m o}{\IfNoValueTF{#3}{\ZygSymb^{#2}_{#1}}{\ZygSymb^{#2}_{#1}(#3) }}

\NewDocumentCommand{\ZygSpaceloc}{m o o}{\IfNoValueTF{#2}{\ZygSymb_{\mathrm{loc}}^{#1}}{ \IfNoValueTF{#3} { \ZygSymb^{#1}_{\mathrm{loc}}(#2) }{\ZygSymb^{#1}_{\mathrm{loc}}(#2;#3) } } }

\NewDocumentCommand{\Norm}{m o}{\IfNoValueTF{#2}{\| #1\|}{\|#1\|_{#2} }}
\NewDocumentCommand{\BNorm}{m o}{\IfNoValueTF{#2}{\left\| #1\right\|}{\left\|#1\right\|_{#2} }}

\NewDocumentCommand{\CjNorm}{m m o o}{ \IfNoValueTF{#3}{ \Norm{#1}[\CjSpace{#2}]} { \IfNoValueTF{#4}{\Norm{#1}[\CjSpace{#2}[#3]]} {\Norm{#1}[\CjSpace{#2}[#3][#4]]}  }  }

\NewDocumentCommand{\CNorm}{m m}{\Norm{#1}[\CSpace{#2}]}

\NewDocumentCommand{\BCNorm}{m m}{\BNorm{#1}[\CSpace{#2}]}

\NewDocumentCommand{\CXjNorm}{m m m o}{\Norm{#1}[
\IfNoValueTF{#4}
{\CXjSpace{#2}{#3}}
{\CXjSpace{#2}{#3}[#4]}
]}

\NewDocumentCommand{\BCXjNorm}{m m m o}{\BNorm{#1}[
\IfNoValueTF{#4}
{\CXjSpace{#2}{#3}}
{\CXjSpace{#2}{#3}[#4]}
]}

\NewDocumentCommand{\LpNorm}{m m o o}{
\Norm{#2}[L^{#1}
\IfNoValueTF{#3}{}{
(#3
\IfNoValueTF{#4}{}{;#4}
)
}
]
}


\NewDocumentCommand{\BCjNorm}{m m o}{ \IfNoValueTF{#3}{ \BNorm{#1}[C^{#2}]} { \BNorm{#1}[C^{#2}(#3)]  }  }

\NewDocumentCommand{\HNorm}{m m m o o}{ \IfNoValueTF{#4}{ \Norm{#1}[\HSpace{#2}{#3}]} {
\IfNoValueTF{#5}
{\Norm{#1}[\HSpace{#2}{#3}[#4]]}
{\Norm{#1}[\HSpace{#2}{#3}[#4][#5]] }
}  }

\NewDocumentCommand{\HXNorm}{m m m m o}{ \IfNoValueTF{#5}{ \Norm{#1}[\HXSpace{#2}{#3}{#4}]} {
{\Norm{#1}[\HXSpace{#2}{#3}{#4}[#5]]}
}  }

\NewDocumentCommand{\BHXNorm}{m m m m o}{ \IfNoValueTF{#5}{ \BNorm{#1}[\HXSpace{#2}{#3}{#4}]} {
{\BNorm{#1}[\HXSpace{#2}{#3}{#4}[#5]]}
}  }

\NewDocumentCommand{\ZygNorm}{m m o o}{ \IfNoValueTF{#3}{ \Norm{#1}[\ZygSpace{#2}]} {
\IfNoValueTF{#4}
{\Norm{#1}[\ZygSpace{#2}[#3]]}
{\Norm{#1}[\ZygSpace{#2}[#3][#4]]}
}  }

\NewDocumentCommand{\BZygNorm}{m m o o}{ \IfNoValueTF{#3}{ \BNorm{#1}[\ZygSpace{#2}]} {
\IfNoValueTF{#4}
{\BNorm{#1}[\ZygSpace{#2}[#3]]}
{\BNorm{#1}[\ZygSpace{#2}[#3][#4]]}
}  }

\NewDocumentCommand{\ZygXNorm}{m m m o}{\Norm{#1}[
\IfNoValueTF{#4}
{\ZygXSpace{#2}{#3}}
{\ZygXSpace{#2}{#3}[#4]}
]}

\NewDocumentCommand{\BZygXNorm}{m m m o}{\BNorm{#1}[
\IfNoValueTF{#4}
{\ZygXSpace{#2}{#3}}
{\ZygXSpace{#2}{#3}[#4]}
]}


\NewDocumentCommand{\diff}{o m}{\IfNoValueTF{#1}{\frac{\partial}{\partial #2}}{\frac{\partial^{#1}}{\partial #2^{#1}} }}

\NewDocumentCommand{\dt}{o}{\IfNoValueTF{#1}{\diff{t}}{\diff[#1]{t} }}

\NewDocumentCommand{\Zygad}{m}{\{ #1\}}

\NewDocumentCommand{\Zygsonu}{}{[s_0;\nu]}

\NewDocumentCommand{\Had}{m}{\langle #1\rangle}

\NewDocumentCommand{\DiffOp}{m}{\Delta_{#1}}

\NewDocumentCommand{\Matrix}{m}{\mathscr{#1}}


\NewDocumentCommand{\SSFunctionSpacesSection}{}{Section 2}
\NewDocumentCommand{\SSStrangeZygSpace}{}{Remark 2.1}
\NewDocumentCommand{\SSBeyondManifold}{}{Section 2.2.1}
\NewDocumentCommand{\SSNormsAreInv}{}{Proposition 2.3}
\NewDocumentCommand{\SSDefineVectDeriv}{}{Remark 2.4}

\NewDocumentCommand{\SSSectionMoreOnAssumptions}{}{Section 4.1}
\NewDocumentCommand{\SSMainResult}{}{Theorem 4.7}
\NewDocumentCommand{\SSLemmaMoreOnAssump}{}{Proposition 4.14}

\NewDocumentCommand{\SSDivideWedge}{}{Section 5}
\NewDocumentCommand{\SSDerivWedge}{}{Lemma 5.1}

\NewDocumentCommand{\SSDensities}{}{Section 6}
\NewDocumentCommand{\SSDensitiesTheorem}{}{Theorem 6.5}
\NewDocumentCommand{\SSDensityCor}{}{Corollary 6.6}

\NewDocumentCommand{\SSScaling}{}{Section 7}
\NewDocumentCommand{\SSNSW}{}{Section 7.1}
\NewDocumentCommand{\SSHormandersCondition}{}{Section 7.1.1}
\NewDocumentCommand{\SSGenSubR}{}{Section 7.3}
\NewDocumentCommand{\SSGenSubResult}{}{Theorem 7.6}

\NewDocumentCommand{\SSCompareFunctionSpaces}{}{Lemma 8.1}
\NewDocumentCommand{\SSZygIsAlgebra}{}{Proposition 8.3}
\NewDocumentCommand{\SSBiggerNormMap}{}{Proposition 8.6}
\NewDocumentCommand{\SSCompareEuclidNorms}{}{Proposition 8.12}

\NewDocumentCommand{\SSDeriveODE}{}{Proposition 9.1}
\NewDocumentCommand{\SSExistODE}{}{Proposition 9.4}
\NewDocumentCommand{\SSExistXiOne}{}{Lemma 9.23}
\NewDocumentCommand{\SSDifferentOneAdmis}{}{Proposition 9.26}
\NewDocumentCommand{\SSCXjNormWedgeQuotient}{}{Lemma 9.32}
\NewDocumentCommand{\SSExistXiTwo}{}{Lemma 9.35}
\NewDocumentCommand{\SSComputefjzero}{}{Lemma 9.38}
\NewDocumentCommand{\SSSectionDensities}{}{Section 9.4}

\NewDocumentCommand{\SSProofInjectiveImmersion}{}{Appendix A}
\NewDocumentCommand{\SSFinerTopology}{}{Lemma A.1}

\newtheorem{thm}{Theorem}[section]
\newtheorem{cor}[thm]{Corollary}
\newtheorem{prop}[thm]{Proposition}
\newtheorem{lemma}[thm]{Lemma}
\newtheorem{conj}[thm]{Conjecture}
\newtheorem{prob}[thm]{Problem}

\theoremstyle{remark}
\newtheorem{rmk}[thm]{Remark}

\theoremstyle{definition}
\newtheorem{defn}[thm]{Definition}

\theoremstyle{definition}
\newtheorem{assumption}[thm]{Assumption}

\theoremstyle{remark}
\newtheorem{example}[thm]{Example}

\numberwithin{equation}{section}

\title{Coordinates Adapted to Vector Fields II: Sharp Results}
\author{Brian Street\footnote{This material is partially based upon work supported by the National Science Foundation under Grant No.\ 1440140, while the author was in residence at the Mathematical Sciences Research Institute in Berkeley, California, during the spring semester of 2017.  The author was also partially supported by National Science Foundation Grant Nos.\ 1401671 and 1764265.}}
\date{}

\maketitle

\begin{abstract}
Given a finite collection of $C^1$ vector fields on a $C^2$ manifold which span the tangent space at every point, we consider the question of when there is locally a coordinate system
in which these vector fields are $\mathscr{C}^{s+1}$ for $s\in (1,\infty]$, where $\mathscr{C}^s$ denotes the Zygmund space of order $s$.
We give necessary and sufficient, coordinate-free conditions for the existence of such a coordinate system.
Moreover, we present a quantitative study of these coordinate charts.  This is the second part in a three part series of papers.
The first part, joint with Stovall, addressed the same question, though the results were not sharp, and showed how such
coordinate charts can be viewed as scaling maps in sub-Riemannian geometry.  When viewed in this light,
these results can be seen as strengthening and generalizing previous works on the quantitative theory
of sub-Riemannian geometry, initiated by Nagel, Stein, and Wainger, and furthered by Tao and Wright,
the author, and others. In the third part, we prove similar results concerning real analyticity.
\end{abstract}


\section{Introduction}
Let $X_1,\ldots, X_q$ be $C^1$ vector fields on a $C^2$ manifold $M$, which span the tangent space at every point of $M$.
For $s>0$, let $\ZygSpace{s}$ denote the Zygmund space
of order $s$, and let $\ZygSpace{\infty}$ denote $C^\infty$ (for noninteger $s$, the Zygmund space coincides
with the classical H\"older space--see \cref{Section::Zygmund} for more details on Zygmund spaces).
In this paper, we investigate the following closely related questions for $s\in (1,\infty]$:
\begin{enumerate}[(i)]
\item\label{Item::Intro::Qual} When is there a coordinate system near a fixed point $x_0\in M$ such that the vector fields $X_1,\ldots, X_q$
are $\ZygSpace{s+1}$ in this coordinate system?
\item\label{Item::Intro::Quant} When is there a $\ZygSpace{s+2}$ manifold structure on $M$, compatible with its $C^2$ structure, such that $X_1,\ldots, X_q$
are $\ZygSpace{s+1}$ with respect to this structure?  When such a structure exists, we will see it is unique.
\item\label{Item::Intro::Coord} When there is a a coordinate system as in \cref{Item::Intro::Qual}, how can we pick it so that $X_1,\ldots, X_q$
are ``normalized'' in this coordinate system in a quantitative way which is useful for applying techniques from analysis?
\end{enumerate}
We present necessary and sufficient conditions for \cref{Item::Intro::Qual,Item::Intro::Quant}, and under these conditions give a quantitative
answer to \cref{Item::Intro::Coord}.

The heart of this paper is \cref{Item::Intro::Coord}; \cref{Item::Intro::Qual,Item::Intro::Quant} are simple consequences of our answer to \cref{Item::Intro::Coord}.
The first paper in this series, joint with Stovall, \cite{StovallStreet} focused on a solution to \cref{Item::Intro::Coord} which ``lost one derivative''.  In this paper,
we take the coordinate chart developed in \cite{StovallStreet} as a black box, and show how to improve it to give the sharp result.
The methods in \cite{StovallStreet} are based on ODEs, while the methods in this paper are based on elliptic PDEs.
These PDE methods were inspired by, and are closely related to, Malgrange's celebrated proof of the Newlander-Nirenberg theorem \cite{MalgrangeSurLIntegbrabilite}.
In the third paper in this series, \cite{StreetIII}, we return to ODE methods to prove analogous results concerning real analyticity.

The coordinate charts developed in \cref{Item::Intro::Coord} can be viewed as scaling maps in sub-Riemannian geometry.
When viewed in this light, these coordinate charts can be seen as the latest
results on
the quantitative theory of sub-Riemannian geometry which was initiated by Nagel, Stein, and Wainger
 \cite{NagelSteinWaingerBallsAndMetrics} and C.\ Fefferman and S\'anchez-Calle \cite{FeffermanSanchezCalleFundamentalSoltuions},
and furthered by many others, including Tao and Wright \cite{TaoWrightLpImproving} and the author \cite{S}.  
We refer the reader to \cite{StovallStreet} for how these charts can be viewed as scaling maps, as well as a more leisurely introduction
to the questions investigated in this paper.

This paper is a continuation of the results in \cite{StovallStreet}.
That paper gives several applications and motivations for the results described here (see, also, \cref{Rmk::QuantRes::SharpApps,Rmk::QuantRes::SharpApps2}), and a more leisurely description of some of the main definitions (though we include all necessary definitions in this paper,
so that the statement of the results is self-contained).

The results in this paper are a key tool in a companion paper where we study analogous questions regarding complex vector fields \cite{StreetNN}.  When viewed from the perspective of sub-Riemannian geometry,
this companion paper allows us to create a quantitative theory of sub-Riemannian geometry  which is adapted to the complex structure of a complex manifold.  We call this sub-Hermitian geometry;
see \cite{StreetNN} for more details. 

\begin{rmk}
The results in this paper may be reminiscent of the celebrated results of DeTurck and Kazdan
\cite{DeTurckKazdanSomeRegularityTheoremsInRiemannianGeometry}
regarding a coordinate system in which a Riemnnian metric tensor has optimal regularity--which also
used the methods introduced by Malgrange \cite{MalgrangeSurLIntegbrabilite}.  However,
there does not seem to be a direct relationship between our results and theirs.
\end{rmk}

\section{Results}
In this section, we present the main results of this paper.  In  \cref{Section::FuncSpace} (also in \cite[\SSFunctionSpacesSection]{StovallStreet}), Zygmund spaces are defined, where a distinction is made
between Zygmund spaces on a subset of $\R^n$, and Zygmund spaces on a $C^2$ manifold $M$.
If $\Omega\subset \R^n$ is a bounded, connected, open set and $s>0$, we write $\ZygSpace{s}[\Omega]$
for the classical  Zygmund space of order $s$ on $\Omega$; and for a Banach space $V$, we write
$\ZygSpace{s}[\Omega][V]$ for the Zygmund space of order $s$ of functions taking values in $V$.
For a vector field $Y=\sum_{j=1}^na_j(t) \diff{t_j}$ on $\Omega$, we identify $Y$ with the function $(a_1,\ldots, a_n):\Omega\rightarrow \R^n$,
so that it makes sense to consider $\ZygNorm{Y}{s}[\Omega][\R^n]$.
We write $\ZygSpace{\infty}[\Omega]:=\bigcap_{s>0} \ZygSpace{s}[\Omega]$, which coincides with the space of smooth functions on $\Omega$,
all of whose derivatives are bounded on $\Omega$.  For complete definitions and more details on $\ZygSpace{s}[\Omega]$, see \cref{Section::Zygmund}.

Fix $M$ a $C^2$ manifold with $C^1$ vector fields $X_1,\ldots, X_q$ on $M$.
On $M$, we have the following:
\begin{itemize}
\item $B_X(x,\delta)$: the sub-Riemannian ball of radius $\delta>0$ centered at $x\in M$, induced by $X_1,\ldots, X_q$.
This is defined by
\begin{equation}\label{Eqn::Res::DefnCCBall}
    \begin{split}
        B_X(x,\delta):=\Bigg\{
        y\in M\: \bigg|\: &\exists \gamma:[0,1]\rightarrow M, \gamma(0)=x, \gamma(1)=y,
        \gamma\text{ is absolutely continuous},
        \\&\gamma'(t)=\sum_{j=1}^q a_j(t) \delta X_j(\gamma(t)),
        a_j\in L^\infty([0,1]), \BNorm{\sum_{j=1}^q|a_j|^2 }[L^\infty]<1
        \Bigg\}.
    \end{split}
\end{equation}
\item $\rho(x,y)$: the sub-Riemannian distance
on $M$ induced by $X_1,\ldots, X_q$--this is the distance associated to the balls $B_X(x,\delta)$.
\begin{equation}\label{Eqn::Res::rho}
    \rho(x,y):=\inf\{\delta>0 : y\in B_X(x,\delta)\}.
\end{equation}
In general, $\rho$ is merely an extended metric ($\rho$ may take the value $\infty$).  However, if $X_1,\ldots, X_q$ span the tangent space at every point and $M$ is connected,
then $\rho$ is a metric--this is the setting we are most interested in.

\item $\HXSpace{X}{m}{s}[M]$:  the scale of H\"older spaces on $M$, for $m\in \N$, $s\in [0,1]$, with respect to $X_1,\ldots, X_q$.  Here, and in the rest of the paper,
we use the convention $0\in \N$.
\item $\ZygXSpace{X}{s}[M]$: the Zygmund space of order $s\in (0,\infty]$ on $M$, with respect to $X_1,\ldots, X_q$.
\end{itemize}
Definitions of $\HXSpace{X}{m}{s}[M]$ and $\ZygXSpace{X}{s}[M]$ are given in \cref{Section::FuncSpace::Manifold}, and we refer the reader to \cite{StovallStreet} for more leisurely discussion of these spaces.  We remark that the Banach
spaces $\HXSpace{X}{m}{s}[M]$ and $\ZygXSpace{X}{s}[M]$ are defined in such a way that their norms are invariant under $C^2$ diffeomorphisms.
More precisely, if $\Psi:N\rightarrow M$ is a $C^2$ diffeomorphism, then
\begin{equation}\label{Eqn::Results::NormsDiffoInv}
\HXNorm{f}{X}{m}{s}[M]=\HXNorm{\Psi^{*}f}{\Psi^{*}X}{m}{s}[N], \quad
\ZygXNorm{f}{X}{s}[M]=\ZygXNorm{\Psi^{*}f}{\Psi^{*}X}{s}[N].
\end{equation}

\begin{rmk}\label{Rmk::Results::NormsDiffInv}
\Cref{Eqn::Results::NormsDiffoInv} can be interpreted as saying the norms $\HXNorm{f}{X}{m}{s}[M]$ and $\ZygXNorm{f}{X}{s}[M]$ are ``coordinate-free.''  In practice,
this means that these norms can be computed in any $C^2$ coordinate system, and the answer is independent of the chosen coordinate system.
Moreover, it makes sense to talk about, for example, $\CXjSpace{X}{\infty}[M]=\bigcap_{m} \HXSpace{X}{m}{0}[M]$, even if $M$ is merely a $C^2$ manifold, and $X_1,\ldots, X_q$ are $C^1$ vector fields on $M$.
\end{rmk}

Throughout the paper, if we say $\ZygXNorm{f}{X}{s}[M]<\infty$ we mean $f\in \ZygXSpace{X}{s}[M]$ and the norm is finite, and similarly for any other function spaces. 

	\subsection{Qualitative Results}
Let $X_1,\ldots, X_q$ be $C^1$ vector fields on a $C^2$ manifold $\fM$.  For $x,y\in \fM$, let $\rho(x,y)$ denote the sub-Riemannian distance
associated to $X_1,\ldots, X_q$ on $\fM$ defined in \cref{Eqn::Res::rho}.  Fix $x_0\in \fM$ and let $Z:=\{y\in \fM : \rho(x_0,y)<\infty\}$.
$\rho$ is a metric on $Z$, and we give $Z$ the topology induced by $\rho$ (this is finer
than the topology as a subspace of $\fM$,
and may be strictly finer--see \cite[\SSFinerTopology]{StovallStreet} for details).  Let $M\subseteq Z$ be a connected open subset of $Z$ containing $x_0$.
We give $M$ the topology of a subspace of $Z$.  We begin with a classical result to set the stage.

\begin{prop}\label{Prop::QualRes::InjectiveImmers}
Suppose $[X_i,X_j]=\sum_{k=1}^q c_{i,j}^k X_k$, where $c_{i,j}^k:M\rightarrow \R$ are locally bounded.
Then, there is a $C^2$ manifold structure on $M$ (compatible with its topology) such that:
\begin{itemize}
\item The inclusion $M\hookrightarrow \fM$ is a $C^2$ injective immersion.
\item $X_1,\ldots, X_q$ are $C^1$ vector fields tangent to $M$.
\item $X_1,\ldots, X_q$ span the tangent space at every point of $M$.
\end{itemize}
Furthermore, this $C^2$ structure is unique in the sense that if $M$ is given another $C^2$ structure (compatible with its topology) such that the inclusion
map $M\hookrightarrow \fM$ is a $C^2$ injective immersion, then the identity map $M\rightarrow M$ is a $C^2$ diffeomorphism between these
two structures.
\end{prop}
For a proof of \cref{Prop::QualRes::InjectiveImmers}, see \cite[\SSProofInjectiveImmersion]{StovallStreet}.
Henceforth, we assume the conditions of \cref{Prop::QualRes::InjectiveImmers} so that $M$ is a $C^2$ manifold and $X_1,\ldots, X_q$ are $C^1$ vector fields
on $M$ which span the tangent space at every point.
We write $n:=\dim \Span\{X_1(x_0),\ldots, X_q(x_0)\}$ so that $\dim M=n$.

\begin{rmk}
If $X_1(x_0),\ldots, X_q(x_0)$ span $T_{x_0} \fM$, then $M$ is an open submanifold of $\fM$.  If $X_1,\ldots, X_q$ span the tangent space
at every point of $\fM$ and $\fM$ is connected, one may take $M=\fM$.
\end{rmk}

\begin{thm}[The Local Theorem]\label{Thm::QualRes::LocalQual}
For $s\in (1,\infty]$, the following three conditions are equivalent:
\begin{enumerate}[(i)]
\item\label{Item:QualRes::Local::Diffeo} There is an open neighborhood $V\subseteq M$ of $x_0$ and a $C^2$ diffeomorphism $\Phi:U\rightarrow V$ where
$U\subseteq \R^n$ is open, such that $\Phi^{*}X_1,\ldots, \Phi^{*}X_q\in \ZygSpace{s+1}[U][\R^n]$.

\item\label{Item::QualRes::LocalQual::Basis}
Re-order the vector fields so that $X_1(x_0),\ldots, X_n(x_0)$ are linearly
independent.
There is an open neighborhood $V\subseteq M$ of $x_0$ such that:
    \begin{itemize}
    \item $[X_i,X_j]=\sum_{k=1}^n \ch_{i,j}^k X_k$, $1\leq i,j\leq n$, where $\ch_{i,j}^k\in \ZygXSpace{X}{s}[V]$.
    \item For $n+1\leq j\leq q$, $X_j=\sum_{k=1}^n b_j^k X_k$, where $b_j^k\in \ZygXSpace{X}{s+1}[V]$.
    \end{itemize}

\item\label{Item::QualRes::LocalQual::Spanning} There exists an open neighborhood $V\subseteq M$ of $x_0$ such that
$[X_i,X_j]=\sum_{k=1}^q c_{i,j}^k X_k$, $1\leq i,j\leq q$, where $c_{i,j}^k\in \ZygXSpace{X}{s}[V]$.
\end{enumerate}
\end{thm}

\begin{rmk}
\Cref{Item::QualRes::LocalQual::Basis} and \cref{Item::QualRes::LocalQual::Spanning} of \cref{Thm::QualRes::LocalQual} are similar but have slightly different advantages.
In \cref{Item::QualRes::LocalQual::Basis}, because $X_1,\ldots, X_n$ form a basis for the tangent space of $M$ near $x_0$, the functions
$\ch_{i,j}^k$ and $b_j^k$ are uniquely determined (so long as $V$ is chosen sufficiently small).
Moreover, one can directly check to see if \cref{Item::QualRes::LocalQual::Basis} holds
by computing these functions.  In light of \cref{Rmk::Results::NormsDiffInv}, this computation can be done in any $C^2$ coordinate system.
If $q>n$, $X_1,\ldots, X_q$ are linearly dependent, and the $c_{i,j}^k$ in \cref{Item::QualRes::LocalQual::Spanning} are not uniquely determined; \cref{Item::QualRes::LocalQual::Spanning}
only asks that there exist a choice of $c_{i,j}^k$ satisfying the conditions in \cref{Item::QualRes::LocalQual::Spanning}.
Despite this lack of uniqueness, in many applications it is more convenient to use the setting in \cref{Item::QualRes::LocalQual::Spanning}
(see, for example, the application of the quantitative results in \cite[\SSHormandersCondition]{StovallStreet}).
\end{rmk}

\begin{rmk}
\Cref{Thm::QualRes::LocalQual} is stated for $s\in (1,\infty]$.  It would be nice to obtain the same result for $s\in (0,\infty]$,
however to do this with the methods of this paper, if it is even possible, would require a more technical analysis of the PDEs which arise.  See \cref{Rmk::PfPhi1::WhyGreater1} for more details.
Similar remarks hold for the other main results of this paper.
\end{rmk}


\begin{thm}[The Global Theorem]\label{Thm::QualRes::GlobalQual}
For $s\in (1,\infty]$, the following three conditions are equivalent:
\begin{enumerate}[(i)]
\item\label{Item::QualRes::Global::Atlas} There exists a $\ZygSpace{s+2}$ atlas on $M$, compatible with its $C^2$ structure,
such that $X_1,\ldots, X_q$ are $\ZygSpace{s+1}$ with respect to this atlas.
\item\label{Item::QualRes::Global::Charts} For each $x_0\in M$, any of the three equivalent conditions from \cref{Thm::QualRes::LocalQual} holds for this choice of $x_0$.
\item\label{Item::QualRes::Global::Commutator} $[X_i,X_j]=\sum_{k=1}^q c_{i,j}^k X_k$, $1\leq i,j\leq q$, where $\forall x_0\in M$, $\exists V\subseteq M$ open with $x_0\in V$ such that $c_{i,j}^k\big|_V\in \ZygXSpace{X}{s}[V]$, $1\leq i,j,k\leq q$.
\end{enumerate}
Furthermore, under these conditions, the $\ZygSpace{s+2}$ manifold structure on $M$ induced by the atlas from \cref{Item::QualRes::Global::Atlas}
is unique, in the sense that if there is another $\ZygSpace{s+2}$ atlas on $M$, compatible with its $C^2$ structure, and such that
$X_1,\ldots, X_q$ are $\ZygSpace{s+1}$ with respect to this second atlas, then the identity
map $M\rightarrow M$ is a $\ZygSpace{s+2}$ diffeomorphism between these two $\ZygSpace{s+2}$ manifold structures
on $M$.
See \cref{Section::FuncSpace::ZygManifolds} for formal definitions regarding $\ZygSpace{s+2}$ manifolds.
\end{thm}

\begin{rmk}
As a corollary, we obtain results similar to \cref{Thm::QualRes::LocalQual,Thm::QualRes::GlobalQual} with the Zygmund spaces $\ZygSpace{m+s}$ replaced
by the easier to understand H\"older spaces $\HSpace{m}{s}$, with the restriction that $s\in (0,1)$.  For details, see
\cref{Section::Holder}.
\end{rmk}

\begin{rmk}
The reader only wishing to understand proof of the above qualitative results, and not the more technical quantitative results,
may wish to skip to the proof outline presented in \cref{Section::Outline}.
\end{rmk}


	\subsection{Quantitative Results}
\Cref{Thm::QualRes::LocalQual} gives necessary and sufficient conditions for a certain type of coordinate chart to exist.  For applications in
analysis, it is essential to have quantitative
 control of this coordinate chart and the quantitative control we obtain will be invariant under arbitrary $C^2$ diffeomorphisms; see \cref{Rmk::QuantRes::DiffeoInv}.  By using this quantitative control, these charts
can be seen as generalized scaling maps in sub-Riemannian geometry--see \cite[\SSScaling]{StovallStreet} and \cref{Rmk::QuantRes::SharpApps,Rmk::QuantRes::SharpApps2} for more details on this and other applications.
We now turn to these quantitative results, which are the heart of this paper.
Because the goal is to keep track of what each constant depends on, this is somewhat technical.
To ease notation, we introduce various notions of ``admissible constants''; these are constants
which depend only on certain parameters.  While these definitions are somewhat unwieldy, they greatly simplify
the statement of results and proofs throughout the paper.

Let $X_1,\ldots, X_q$ be $C^1$ vector fields on a $C^2$ manifold $\fM$.
Throughout the paper, $B^n(\eta)$ denotes the Euclidean ball of radius $\eta>0$ centered at $0\in \R^n$.

\begin{defn}
For $x\in \fM$, $\eta>0$, and $U\subseteq \fM$, we say the list $X=X_1,\ldots, X_q$ satisfies $\sC(x_0,\eta,U)$ if
for every $a\in B^q(\eta)$ the expression
\begin{equation*}
e^{a_1 X_1+\cdots+a_q X_q}x_0
\end{equation*}
exists in $U$.  More precisely, consider the differential equation
\begin{equation*}
\diff{r} E(r) = a_1 X_1(E(r))+\cdots+a_qX_q(E(r)), \quad E(0)=x_0.
\end{equation*}
We assume that a solution to this differential equation exists up to $r=1$, $E:[0,1]\rightarrow U$.  We have
$E(r)=e^{ra_1 X_1+\cdots+ra_q X_q}x_0$.
\end{defn}

For $1\leq n\leq q$, we let
\begin{equation*}
\sI(n,q):=\{(i_1,i_2,\ldots, i_n) : i_j\in \{1,\ldots, q\}\} = \{1,\ldots, q\}^n.
\end{equation*}
For $J=(j_1,\ldots, j_n)\in \sI(n,q)$ we write $X_J$ for the list of vector fields $X_{j_1},\ldots, X_{j_n}$.  We write
$\bigwedge X_J:=X_{j_1}\wedge X_{j_2}\wedge \cdots \wedge X_{j_n}$.

Fix $x_0\in \fM$ and let $n:=\dim \Span\{X_1(x_0),\ldots, X_q(x_0)\}$.  Fix $\xi,\zeta\in (0,1]$.  We assume
that on $B_X(x_0,\xi)$, the $X_j$'s satisfy
\begin{equation*}
[X_j,X_k]=\sum_{l=1}^q c_{j,k}^l X_l, \quad c_{j,k}^l\in C(B_X(x_0,\xi)),
\end{equation*}
where $B_X(x_0,\xi)$ is given the metric topology induced by $\rho$ from \cref{Eqn::Res::rho}.
\Cref{Prop::QualRes::InjectiveImmers} applies to show that $B_X(x_0,\xi)$ is an $n$-dimensional, $C^2$,
injectively immersed submanifold of $\fM$.  $X_1,\ldots, X_q$ are $C^1$ vector fields on $B_X(x_0,\xi)$
and span the tangent space at every point.  Henceforth, we treat $X_1,\ldots, X_q$ as vector fields on $B_X(x_0,\xi)$.

Let $J_0\in \sI(n,q)$ be such that $\bigwedge X_{J_0}(x_0)\ne 0$ and moreover
\begin{equation}\label{Eqn::QuantRes::PickJ0}
\max_{J\in \sI(n,q)} \left| \frac{\bigwedge X_J(x_0)}{\bigwedge X_{J_0}(x_0)} \right|\leq \zeta^{-1},
\end{equation}
where $\frac{\bigwedge X_J(x_0)}{\bigwedge X_{J_0}(x_0)}$ is defined as follows.
Let $\lambda:\bigwedge^n T_{x_0} B_X(x_0,\xi)\rightarrow \R$ be any nonzero linear functional;
then
\begin{equation}\label{Eqn::QualRes::DivideWedge}
\frac{\bigwedge X_J(x_0)}{\bigwedge X_{J_0}(x_0)}:=\frac{\lambda\left(\bigwedge X_J(x_0)\right)}{\lambda \left(\bigwedge X_{J_0}(x_0)\right)}.
\end{equation}
Because $\bigwedge^n T_{x_0}B_X(x_0,\xi)$ is one dimensional, \cref{Eqn::QualRes::DivideWedge} is independent of the choice of $\lambda$; see \cite[\SSDivideWedge]{StovallStreet} for more details.
Note that a $J_0\in \sI(n,q)$ satisfying \cref{Eqn::QuantRes::PickJ0} always exists--one can pick $J_0$ so that \cref{Eqn::QuantRes::PickJ0} holds with $\zeta=1$; however, it is important
for some applications
 to have the flexibility to choose $\zeta<1$ (this is needed, for example, in \cite{StreetNN}).
Without loss of generality, reorder the vector fields so that $J_0=(1,\ldots, n)$.

\begin{itemize}
\item Let $\eta>0$ be such that $X_{J_0}$ satisfies $\sC(x_0,\eta,\fM)$.
\item Let $\delta_0>0$ be such that for $\delta\in (0,\delta_0]$ the following holds:  if $z\in B_{X_{J_0}}(x_0,\xi)$ is such that $X_{J_0}$
satisfies $\sC(z,\delta, B_{X_{J_0}}(x_0,\xi))$ and if $t\in B^n(\delta)$ is such that $e^{t_1 X_1+\cdots+t_nX_n} z= z$ and if
$X_1(z),\ldots, X_n(z)$ are linearly independent, then $t=0$.
\end{itemize}

\begin{rmk}
Because $X_1,\ldots, X_n$ are $C^1$, such an $\eta>0$ and $\delta_0>0$ always exist; see \cref{Lemma::PfQual::AlwaysHaveEtaDelta,Rmk::PfQual::AlwaysHaveEtaDelta}.  However,
in general one can only guarantee that $\eta$, $\delta_0$ are small in terms of the $C^1$ norms
of $X_1,\ldots, X_n$ in some coordinate system--and this is not a diffeomorphic invariant quantity.
Thus, we state our results in terms of $\delta_0$ an $\eta$ to preserve the diffeomorphic invariance.
See \cite[\SSSectionMoreOnAssumptions]{StovallStreet} for more details on $\eta$ and $\delta_0$.
\end{rmk}

\begin{defn}\label{Defn::Results::0admissible}
We say $C$ is a $0$-admissible constant if $C$ can be chosen to depend only on upper bounds
for $q$, $\zeta^{-1}$, $\xi^{-1}$, and $\CNorm{c_{j,k}^l}{B_{X_{J_0}}(x_0,\xi)}$, $1\leq j,k,l\leq q$.
\end{defn}

For the remainder of this section, fix $s_0>1$.  The results which follow depend on this choice of $s_0$, and are stronger as $s_0$ approaches $1$.

\begin{defn}\label{Defn::Results::sadmissible}
For $s\geq s_0$, if we say $C$ is an $\Zygad{s}$-admissible constant, it means that we assume $c_{j,k}^l\in \ZygXSpace{X_{J_0}}{s}[B_{X_{J_0}}(x_0,\xi)]$
for $1\leq j,k,l\leq q$.  $C$ is then allowed to depend on $s$, $s_0$, lower bounds $>0$ for $\zeta$, $\xi$, $\eta$,
and $\delta_0$, and upper bounds for $q$ and $\ZygXNorm{c_{j,k}^l}{X_{J_0}}{s}[B_{X_{J_0}}(x_0,\xi)]$, $1\leq j,k,l\leq q$.
For $s<s_0$, we define $\Zygad{s}$-admissible constants to be $\Zygad{s_0}$-admissible constants.
\end{defn}

We write $A\lesssim_{\Zygad{s}} B$ for $A\leq CB$ where $C$ is a positive $\Zygad{s}$-admissible constant.
We write $A\approx_{\Zygad{s}} B$ for $A\lesssim_{\Zygad{s}} B$ and $B\lesssim_{\Zygad{s}} A$.
Similarly we define $\lesssim_0$ and $\approx_0$ for the same comparisons with $0$-admissible constants
in place of $\Zygad{s}$-admissible constants.

\begin{thm}[The Quantitative Theorem]\label{Thm::QuantRes::MainThm}
There exists a $0$-admissible constant $\chi\in (0,\xi]$ such that:
\begin{enumerate}[label=(\alph*),series=maintheoremenumeration]
\item\label{Item::QuantRes::NonzeroJ0} $\forall y\in B_{X_{J_0}}(x_0,\chi)$, $\bigwedge X_{J_0}(y)\ne 0$.
\item\label{Item::QuantRest::J0Big} $\forall y\in B_{X_{J_0}}(x_0,\chi)$,
\begin{equation*}
\sup_{J\in \sI(n,q)}\left|\frac{\bigwedge X_J(y)}{\bigwedge X_{J_0}(y)}\right|\approx_0 1.
\end{equation*}
\item\label{Item::QuantRes::Submanifold} $\forall \chi'\in (0,\chi]$, $B_{X_{J_0}}(x_0,\chi')$ is an open subset of $B_X(x_0,\xi)$ and is therefore a submanifold.
\end{enumerate}
For the rest of the theorem, we assume $c_{j,k}^l\in \ZygXSpace{X_{J_0}}{s_0}[B_{X_{J_0}}(x_0,\xi)]$, for $1\leq j,k,l\leq q$.
There exists a $C^2$ map $\Phi:B^n(1)\rightarrow B_{X_{J_0}}(x_0,\chi)$ and $\Zygad{s_0}$-admissible constants
$\xi_1,\xi_2>0$ such that:
\begin{enumerate}[resume*=maintheoremenumeration]
\item\label{Item::QuantRes::ImageOpen} $\Phi(B^n(1))$ is an open subset of $B_{X_{J_0}}(x_0,\chi)$, and is therefore a submanifold of $B_X(x_0,\xi)$.
\item\label{Item::QuantRes::Diffeo} $\Phi:B^n(1)\rightarrow \Phi(B^n(1))$ is a $C^2$ diffeomorphism.
\item\label{Item::QuantRes::Xis} $B_X(x_0,\xi_2)\subseteq B_{X_{J_0}}(x_0,\xi_1)\subseteq \Phi(B^n(1))\subseteq B_X(x_0,\xi)$.
\item\label{Item::QuantRes::At0} $\Phi(0)=x_0$.
\end{enumerate}
Let $Y_j=\Phi^{*}X_j$ and let $\M^{n\times n}$ denote the Banach space of $n\times n$ real matrices endowed with the operator norm.  There exists an $\Zygad{s_0}$-admissible
$K\geq 1$
and a matrix $A\in \ZygSpace{s_0}[B^n(1)][\M^{n\times n}]$
such that:
\begin{enumerate}[resume*=maintheoremenumeration]
\item\label{Item::QuantRes::YMatrix} $Y_{J_0} = K(I+A)\grad$, where $\grad$ denotes the gradient in $\R^n$ (thought of as a column vector) and we are identifying $Y_{J_0}$ with the
column vector of vector fields $[Y_1, Y_2,\ldots, Y_n]^{\transpose}$.
\item\label{Item::QuantRes::BoundA} $A(0)=0$ and $\sup_{t\in B^n(1)} \Norm{A(t)}[\M^{n\times n}]\leq \frac{1}{2}$.
\item\label{Item::QuantRes::YZyg} For $s>0$, $1\leq j\leq q$,
\begin{equation}\label{Eqn::QuantRes::RegularityYs}
\ZygNorm{Y_j}{s+1}[B^n(1)][\R^n]\lesssim_{\Zygad{s}} 1.
\end{equation}
\item\label{Item::QuantRes::EquivNorms} We have the following equivalence of norms, for $f\in \CSpace{B^n(1)}$, $s>0$,
\begin{equation*}
\ZygNorm{f}{s}[B^n(1)] \approx_{\Zygad{s-2}} \ZygXNorm{f}{Y_{J_0}}{s}[B^n(1)]\approx_{\Zygad{s-2}} \ZygXNorm{f}{Y}{s}[B^n(1)].
\end{equation*}
\item\label{Item::QuantRes::BoundByInvNorm}
 For $f\in \CSpace{B_{X_{J_0}}(x_0,\chi)}$, $s>0$,
\begin{equation*}
\ZygNorm{f\circ \Phi}{s}[B^n(1)] \lesssim_{\Zygad{s-2}} \ZygXNorm{f}{X_{J_0}}{s}[B_{X_{J_0}}(x_0,\chi)].
\end{equation*}
\end{enumerate}
\end{thm}

\begin{rmk}\label{Rmk::QuantRes::DiffeoInv}
The main results of this paper (including \cref{Thm::QuantRes::MainThm}) are invariant under arbitrary $C^2$ diffeomorphisms.
This is true \textit{quantitatively}--all of the estimates are unchanged when pushed forward under an arbitrary $C^2$ diffeomorphism; this is a consequence of \cref{Eqn::Results::NormsDiffoInv}.
More precisely, take $\fM$ and $X_1,\ldots, X_q$ as above.  Let $N$ be another $C^2$ manifold and let $\Psi:M\rightarrow N$ be a $C^2$ diffeomorphism.  Then,
$X_1,\ldots, X_q$ satisfy the hypotheses of  \cref{Thm::QuantRes::MainThm} at the base point $x_0$ if and only if $\Psi_{*} X_1,\ldots, \Psi_{*} X_q$ satisfy them at $\Psi(x_0)$.  Moreover, admissible constants (of any kind)
when defined in terms of $X_1,\ldots, X_q$ are the same as admissible constants when defined in terms of $\Psi_{*}X_1,\ldots, \Psi_{*} X_q$.  Also, if $\Phi$ is the map
guaranteed by \cref{Thm::QuantRes::MainThm} when applied to $X_1,\ldots, X_q$ at the point $x_0$, then $\Psi\circ \Phi$ is the map guaranteed by \cref{Thm::QuantRes::MainThm} when applied to $\Psi_{*}X_1,\ldots, \Psi_{*}X_q$ at the point $\Psi(x_0)$ (as can be checked by tracing through the proof).  Thus, the conclusions of \cref{Thm::QuantRes::MainThm} (and the other main results of this paper) remain completely unchanged when the setting is pushed forward under a $C^2$ diffeomorphsim.
See \cite{StovallStreet} for more details.
\end{rmk}

\begin{rmk}\label{Rmk::QuantRes::SharpApps}
As mentioned before, \cite[\SSScaling]{StovallStreet} contains several applications for results like \cref{Thm::QuantRes::MainThm,Thm::Densities::MainThm}.
Many of the applications
in \cite[\SSScaling]{StovallStreet} provide results in an infinitely smooth setting.  By using the results in this paper (e.g., \cref{Thm::QuantRes::MainThm}) in place of the corresponding results in \cite{StovallStreet} one can immediately obtain analogous results regarding a finite level of smoothness
using the same proofs, which are in many ways sharp.  This sharpness may be useful when studying certain non-linear PDEs defined by vector fields--where the vector fields may be defined
in terms of the solution to the PDE and one does not have a priori access to smoothness estimates.
\end{rmk}

\begin{rmk}\label{Rmk::QuantRes::SharpApps2}
In  \cref{Thm::QuantRes::MainThm,Thm::Densities::MainThm} we have been explicit about what each constant depends on (by using the various kinds of admissible constants).
In applications, what turns out to be important is what the constants do not depend on.
Two simple examples of how this can work are as follows:
\begin{itemize}
 \item
 We describe the setting of the foundational work of Nagel, Stein, and Wainger \cite{NagelSteinWaingerBallsAndMetrics}.
 Let $Z_1,\ldots, Z_q$ be smooth vector fields on a smooth manifold $M$,
where each vector field $Z_j$ is paired with a formal degree $d_j\in [1,\infty)$.  Suppose, for $1\leq j,k\leq q$,
\begin{equation*}
	[Z_j, Z_k]=\sum_{d_l\leq d_{j}+d_k} c_{j,k}^l Z_l, \quad c_{j,k}^{l}\in \ZygSpaceloc{\infty}[M].
\end{equation*}
Set $X_j^{\delta}:=\delta^{d_j} Z_j$.  Then it easy easy to see that $X_1^{\delta},\ldots, X_q^{\delta}$ satisfy the hypotheses of \cref{Thm::QuantRes::MainThm}
uniformly in $\delta\in (0,1]$ and uniformly as the base point $x_0$ ranges over compact subsets of $M$.
Thus, the conclusions of \cref{Thm::QuantRes::MainThm} hold uniformly in the same way; i.e., the various kinds of admissible constants can be chosen independent of $\delta\in (0,1]$
and $x_0$ (as $x_0$ ranges over a compact set).  See \cite[\SSNSW]{StovallStreet} for more details on this application.
One can proceed more generally by letting the $X_j^{\delta}$ depend on $\delta$ in a more complicated way; see \cite[\SSGenSubR]{StovallStreet}.
\item Let $X_1,\ldots, X_q$ be $C^\infty$ vector fields on a smooth manifold $M$.  Suppose $[X_j,X_k]=\sum_{l=1}^q c_{j,k}^l X_l$, where $c_{j,k}^l\in \ZygSpaceloc{\infty}[M]$.
The classical Frobenius theorem applies to foliate $M$ into leaves.  This may be a singular foliation:  the dimension of the leaves might not be constant.
The classical proofs of the Frobenius theorem give coordinate charts which define these leaves; however these coordinate charts ``blow-up'' as one approaches a singular point (i.e., a point where the dimension of the leaves is not constant on any neighborhood of the point).
The quantitative nature of \cref{Thm::QuantRes::MainThm}  shows that it gives coordinate charts that avoid this blow-up in a certain sense.
See \cite{StreetNN} for a detailed discussion of this.
\end{itemize}
The above two examples work with $C^\infty$ vector fields, however (as in \cref{Rmk::QuantRes::SharpApps}) it is straightforward to work with $C^1$ vector fields and instead assume the hypotheses
of \cref{Thm::QuantRes::MainThm} hold uniformly in the relevant parameters.  This allows 
one to obtain results which are in many ways sharp in terms of regularity.  We leave further details to the reader.
\end{rmk}

		\subsubsection{Densities}\label{Section::Results::Densities}
Let $\chi\in (0,\xi]$ be as in \cref{Thm::QuantRes::MainThm}.  In many applications (e.g., \cite[\SSNSW]{StovallStreet}), one is given a density on $B_{X_{J_0}}(x_0,\chi)$
and it is of interest to measure certain sets with respect to this density.  For a quick introduction to the basics of densities,
we refer the reader to \cite{GuilleminNotes} (see also \cite{NicolaescuLecturesOnTheGeometryOfManifolds} where densities
are called $1$-densities).

Let $\nu$ be a $C^1$ density on $B_{X_{J_0}}(x_0,\chi)$.  Suppose
\begin{equation*}
\Lie{X_j} \nu = f_j \nu, \quad 1\leq j\leq n, \quad f_j\in \CSpace{B_{X_{J_0}}(x_0,\chi)},
\end{equation*}
where $\Lie{X_j}$ denotes the Lie derivative with respect to $X_j$.  Our goal is to understand $\Phi^{*}\nu$ and $\nu(B_{X}(x_0,\xi_2))$,
where $\Phi$ and $\xi_2$ are as in \cref{Thm::QuantRes::MainThm}.



\begin{rmk}
Recall, in \cref{Thm::QuantRes::MainThm} we fixed some $s_0>1$ and all of the estimates in \cref{Thm::QuantRes::MainThm} were in terms
of this fixed $s_0$.  Similarly, all of the results in this section depend on this fixed choice of $s_0$.
\end{rmk}

\begin{defn}\label{Defn::Densities::BasicAdmissibleConsts}
If we say $C$ is a $\Zygsonu$-admissible constant, it means that $C$ is a $\Zygad{s_0}$-admissible constant which is also allowed to depend on upper bounds for
$\CNorm{f_j}{B_{X_{J_0}}(x_0,\chi)}$, $1\leq j\leq n$.
\end{defn}

\begin{defn}\label{Defn::Densitities::AdmissibleConsts}
For $s\in (0,\infty)$, if we say $C$ is an $\Zygad{s;\nu}$-admissible constant, it means that we assume $f_j\in \ZygXSpace{X_{J_0}}{s}[B_{X_{J_0}}(x_0,\chi)]$,
and $C$ is a $\Zygad{s}$-admissible constant which is also allowed to depend on upper bounds for $\ZygXNorm{f_j}{X_{J_0}}{s}[B_{X_{J_0}}(x_0,\chi)]$, $1\leq j\leq n$.
For $s\leq 0$, we define $\Zygad{s;\nu}$-admissible constants to be $\Zygsonu$-admissible constants.
\end{defn}

We write $A\lesssim_{\Zygad{s;\nu}} B$ for $A\leq CB$ where $C$ is a positive $\Zygad{s;\nu}$-admissible constant.
We write $A\approx_{\Zygad{s;\nu}} B$ for $A\lesssim_{\Zygad{s;\nu}} B$ and $B\lesssim_{\Zygad{s;\nu}} A$.
We similarly define $\lesssim_{\Zygsonu}$ and $\approx_{\Zygsonu}$.

\begin{thm}\label{Thm::Densities::MainThm}
Define $h\in \CjSpace{1}[B^n(1)]$ by $\Phi^{*} \nu = h\LebDensity$, where $\LebDensity$ denotes the usual Lebesgue density on $\R^n$.
\begin{enumerate}[(a)]
\item\label{Item::Density::HConst} $h(t)
\approx_{\Zygsonu}
\nu(X_1,\ldots, X_n)(x_0)$, $\forall t\in B^n(1)$.  In particular, $h(t)$ always has the same sign, and is either never zero or always zero.
\item\label{Item::Density::HZyg} For $s>0$, $\ZygNorm{h}{s}[B^n(1)]\lesssim_{\Zygad{s-1;\nu}} |\nu(X_1,\ldots, X_n)(x_0)|$.
\end{enumerate}
\end{thm}

\begin{cor}\label{Cor::Densities::MainCor}
Let $\xi_2$ be as in \cref{Thm::QuantRes::MainThm}.  Then,
\begin{equation}\label{Eqn::Densities::EstimateVolumes}
\nu(B_{X_{J_0}}(x_0,\xi_2)) \approx_{\Zygsonu} \nu(B_X(x_0,\xi_2))\approx_{\Zygsonu} \nu(X_1,\ldots, X_n)(x_0),
\end{equation}
and therefore,
\begin{equation*}
\begin{split}
&|\nu(B_{X_{J_0}}(x_0,\xi_2))|\approx_{\Zygsonu} |\nu(B_X(x_0,\xi_2))|
\approx_{\Zygsonu} |\nu(X_1,\ldots, X_n)(x_0)|
\\&\approx_{0} \max_{(j_1,\ldots, j_n)\in \sI(n,q)} |\nu(X_{j_1},\ldots, X_{j_n})(x_0)|.
\end{split}
\end{equation*}
\end{cor}

\section{Outline of the proof}\label{Section::Outline}
The proof of \cref{Thm::QuantRes::MainThm} is somewhat technical.  This is partially due to its quantitative nature:
we keep careful track of what each constant depends on at every step.  As mentioned before, this is essential for
the applications we have in mind (see, e.g., \cref{Rmk::QuantRes::SharpApps2}).  In this section, we present an outline of the
proof where we do not keep track of such dependencies.  We hope this will help give the reader an overview of the proof
before we enter into the technical details.  For this section, we write $A\lesssim B$ to mean $A\leq CB$, where $C$ is a constant ``which
only depends on the right things;'' we will make such estimates precise in the rigourous proof in later sections.  To keep things simple, we outline the proof of \cref{Thm::QualRes::LocalQual} \cref{Item::QualRes::LocalQual::Spanning}$\Rightarrow$\cref{Item:QualRes::Local::Diffeo} which is essentially
a qualitative version of \cref{Thm::QuantRes::MainThm}.

Fix an $n$-dimensional $C^2$ manifold $M$, and suppose we are given $C^1$ vector fields $X_1,\ldots, X_q$
on $M$ which span the tangent space at every point.  Fix a point $x_0\in M$, and reorder the vector fields
so that $X_1(x_0),\ldots, X_n(x_0)$ form a basis of $T_{x_0}M$ (\eqref{Eqn::QuantRes::PickJ0} is the assumption
that $X_1,\ldots, X_n$ have nearly ``maximal determinant'' among all such choices).  Fix $s\in (1,\infty]$.  Our main assumption is
\begin{equation*}
    [X_i, X_j]= \sum_{k=1}^q c_{i,j}^k X_k
\end{equation*}
near $x_0$, where $c_{i,j}^k\in \ZygXSpace{X}{s}$, near $x_0$.

\noindent\textbf{Goal:}  Our goal is to find a $C^2$ diffeomorphism $\Phi:B^n(1)\xrightarrow{\sim} \Phi(B^n(1))$,
where $\Phi(B^n(1))\subseteq M$ is an open neighborhood of $x_0$, such that
$\Phi^{*}X_j\in \ZygSpace{s+1}[B^n(1)][\R^n]$, and moreover
$\ZygNorm{X_j}{s+1}[B^n(1)][\R^n]\lesssim 1$.

A main problem we face is that our assumptions are in terms of the diffeomorphically  invariant spaces $\ZygXSpace{X}{s}$,
and not in terms of standard spaces, and so we cannot initially apply standard techniques.  The first step gets around this issue.

\noindent\textbf{Step 1:} The results of \cite{StovallStreet} (see, also, \cref{Prop::PartI::MainResult}) provide a $C^2$ diffeomorphism
$\Phi_0:B^n(\eta_0)\xrightarrow{\sim} \Phi_0(B^n(\eta_0))$, where $\eta_0\gtrsim 1$, $\Phi_0(B^n(\eta_0))$ is an open neighborhood
of $x_0$, $\Phi_0(0)=x_0$, and such that if $Y_j:=\Phi_0^{*}X_j$, then the following holds.
\begin{enumerate}[(i)]
\item\label{Item::Outline::Step1::a} We have
\begin{equation*}
    \begin{bmatrix}
    Y_1\\
    Y_2\\
    \vdots\\
    Y_n
    \end{bmatrix}
    = (I+A)\grad,
    \end{equation*}
    where $\grad$ denotes the gradient in $\R^n$ (thought of as a column vector), and $A(t)$ is an $n\times n$ matrix satisfying
    $\ZygNorm{A}{s}[B^n(\eta_0)][\M^{n\times n}]\lesssim 1$ and $A(0)=0$.

    \item\label{Item::Outline::Step1::b} For $n+1\leq k\leq q$, $Y_k=\sum_{l=1}^n b_k^l Y_l$, where $b_k^l\in \ZygSpace{s+1}[B^n(\eta_0)]$.

    \item\label{Item::Outline::Step1::c} For $1\leq j,k\leq n$, $[Y_j,Y_k]=\sum_{l=1}^n \check{c}_{j,k}^l Y_l$, with $\check{c}_{j,k}^l\in \ZygSpace{s}[B^n(\eta_0)]$.
\end{enumerate}

Step 1 achieves the goal, except with a loss of one derivative:  the vector fields $Y_1,\ldots, Y_q$ are only $\ZygSpace{s}$, not
$\ZygSpace{s+1}$.  However, more is true:  if all we knew was that $Y_1,\ldots, Y_q$ were $\ZygSpace{s}$, then
we would only have \eqref{Item::Outline::Step1::b} and \eqref{Item::Outline::Step1::c} with $s$ replaced by $s-1$.  We will leverage
this extra regularity to find a new coordinate system which completes the proof.  To do this, we use methods
adapted from Malgrange's work \cite{MalgrangeSurLIntegbrabilite}.

\noindent\textbf{Reduction 1:}  It suffices to find a $\ZygSpace{s+1}$ diffeomorphism $\Phi_2:B^n(1)\xrightarrow{\sim} \Phi_2(B^n(1))\subseteq B^n(\eta_0)$,
where $\Phi_2(0)=0$, and such that $\BZygNorm{\Phi_2^{*} Y_j}{s+1}\lesssim 1$, for $1\leq j\leq n$.  Indeed, given such a map
$\Phi_2$, the goal is achieved by taking $\Phi:=\Phi_0\circ \Phi_2$.

\noindent\textbf{Step 2:}  Fix $\gamma_2>0$ small, to be chosen later.  Let $\Psi_{\gamma}(t):=\gamma t$.
For $1\leq j\leq n$, set
$\Yt_j:=\gamma \Psi_{\gamma}^{*}Y_j$.  We have
\begin{equation*}
    \begin{bmatrix}
    \Yt_1\\
    \Yt_2\\
    \vdots\\
    \Yt_n
    \end{bmatrix}
    =(I+\At)\grad,
\end{equation*}
where $\At(t) = A(\gamma t)$.  Since $A(0)=0$, by taking $\gamma=\gamma(\gamma_2,\eta_0)>0$ sufficiently small, we have
\begin{enumerate}
\item $\BZygNorm{\At}{s}[B^n(5)][\M^{n\times n}]\leq \gamma_2$.
\item For $1\leq j,k\leq n$, $[\Yt_j, \Yt_k]=\sum_{l=1}^n \ct_{j,k}^l \Yt_l$, where
$\BZygNorm{\ct_{j,k}^l}{s}[B^n(5)]\lesssim 1$
(since
    $\ct_{j,k}^l = \gamma \Psi_{\gamma}^{*} \check{c}_{j,k}^l$).
\end{enumerate}

\noindent\textbf{Reduction 2:}
It suffices to find a $\ZygSpace{s+1}$ diffeomorphism $\Phi_1:B^n(1)\xrightarrow{\sim}\Phi_1(B^n(1))\subseteq B^n(5)$, with $\Phi_1(0)=0$, and such that
$\ZygNorm{\Phi_1^{*} \Yt_j}{s+1}\lesssim 1$, $1\leq j\leq n$.  Here, we may take $\gamma_2$ as small as we like when finding $\Phi_1$.
We then take $\Phi_2:=\Psi_{\gamma_1}\circ \Phi_1$ to complete the proof.  This is \cref{Prop::Proofs::Phi1}.

\noindent\textbf{Step 3:}  This step is \cref{Lemma::PfPhi1::AuxLemma}.
Fix $\gamma_1>0$ small to be chosen later.
By taking $\gamma_2=\gamma_2(\gamma_1)>0$
sufficiently small, we find a $\ZygSpace{s+1}$ diffeomorphism $\Phi_1:B^n(2)\xrightarrow{\sim}\Phi_1(B^n(2))\subseteq B^n(5)$, with $\Phi_1(0)=0$, such that if $\Yh_j:=\Phi_1^{*} \Yt_j$, then
\begin{equation*}
    \begin{bmatrix}
    \Yh_1\\
    \Yh_2\\
    \vdots\\
    \Yh_n
    \end{bmatrix}
    =(I+\Ah)\grad,
\end{equation*}
where
\begin{itemize}
\item If $\Ah_j$ is the $j$th row of $\Ah$, then $\sum_{j=1}^n \diff{v_j} \Ah_j(v)=0$.
\item If $\ah_j^k$ is the $(j,k)$ component of $\Ah$, then $\LpNorm{\infty}{\ah_j^k}\leq \gamma_1$.
\end{itemize}
We find $\Phi_1$ by solving a nonlinear elliptic PDE satisfied by $\Phi_1^{-1}$.  See \cref{Lemma::PfPhi1::AuxLemma}.

All that remains to show is that the map $\Phi_1$ given in Step 3 satisfies the conditions of Reduction 2, provided $\gamma_1$
is taken small enough.  This is covered in \cref{Prop::PfRegularity::MainProp}.  The idea is the following.
We have
\begin{equation}\label{Eqn::Outline::CommYh}
[\Yh_j, \Yh_k]=\sum_{l=1}^n \ch_{j,k}^l \Yh_l,
\end{equation}
 where $\ch_{j,k}^l=\Phi_1^{*} \ct_{j,k}^l\in \ZygSpace{s}[B^n(2)]$.
Also, we know $A\in \ZygSpace{s}[B^n(2)]$, but we wish to show that $A\in \ZygSpace{s+1}[B^n(1)]$.
To do this, note that \cref{Eqn::Outline::CommYh} can be re-written as
\begin{equation*}
\diff{t_j} \Ah_k - \diff{t_k} \Ah_j + \Ah_j \diff{t} \Ah_k - \Ah_k \diff{t} A_j = D_{j,k},
\end{equation*}
where $\diff{t}=[\diff{t_1},\ldots, \diff{t_n}]^{\transpose}$ and $D_{j,k}\in \ZygSpace{s}$.  Combining this with $\sum_{j=1}^n \diff{v_j} \Ah_j(v)=0$, we see
that $\Ah$ satisfies
the system of equations:
\begin{equation*}
\sE A + \Gamma(A,\grad A) = D,
\end{equation*}
where $D\in \ZygSpace{s}$, $\Gamma$ is an explicit constant coefficient bilinear form, and
\begin{equation*}
\sE \Ah = \left( \left( \diff{t_j} \Ah_k- \diff{t_k} \Ah_j \right)_{1\leq j<k\leq n}, \sum_{j=1}^n \diff{t_j} \Ah_j \right).
\end{equation*}
By \cref{Lemma::Elliptic::Ex}, $\sE$ is elliptic.  If $\gamma_1$ is chosen sufficiently small, standard elliptic theory shows
$A\in \ZygSpace{s+1}$, completing the proof.

\begin{rmk}
When we turn to the rigorous proof, we present the steps in the reverse order.  This is because it is much easier to make explicit the quantitative nature of each step when they are presented in the reverse order.
\end{rmk} 		
		
\section{Results from the first paper}
In this section, we describe the main result of \cite{StovallStreet}; namely, \cite[\SSMainResult]{StovallStreet}.  We do not state the full result and instead state an immediate consequence
of it, which is what is relevant for this paper.  The setting is the same as \cref{Thm::QuantRes::MainThm}, so that we have fixed some $s_0>1$
and defined $0$-admissible constants and $\Zygad{s}$-admissible constants as in \cref{Defn::Results::0admissible,Defn::Results::sadmissible}.
As in \cref{Thm::QuantRes::MainThm} we, without loss of generality, reorder the vector fields so that $J_0=(1,\ldots, n)$.
Set $\eta_0:=\min\{\eta,\xi\}$ and define $\Phi_0:B^n(\eta_0)\rightarrow B_{X_{J_0}}(x_0,\xi)$ by
\begin{equation}\label{Eqn::PartI::DefnPhi0}
\Phi_0(t_1,\ldots, t_n):= e^{t_1 X_1+\cdots+ t_n X_n} x_0.
\end{equation}

\begin{prop}\label{Prop::PartI::MainResult}
There exists a $0$-admissible constant $\chi\in (0,\xi]$ such that:
\begin{enumerate}[label=(\alph*),series=partonetheoremenumeration]
\item\label{Item::PartI::NonzeroJ0} $\forall y\in B_{X_{J_0}}(x_0,\chi)$, $\bigwedge X_{J_0}(y)\ne 0$.
\item\label{Item::PartI::J0Big} $\forall y\in B_{X_{J_0}}(x_0,\chi)$,
\begin{equation*}
\sup_{J\in \sI(n,q)}\left|\frac{\bigwedge X_J(y)}{\bigwedge X_{J_0}(y)}\right|\approx_0 1.
\end{equation*}
\item\label{Item::PartI::Submanifold} $\forall \chi'\in (0,\chi]$, $B_{X_{J_0}}(x_0,\chi')$ is an open subset of $B_X(x_0,\xi)$ and is therefore a submanifold.
\end{enumerate}
For the rest of the proposition, we assume $c_{j,k}^l\in \ZygXSpace{X_{J_0}}{s_0}[B_{X_{J_0}}(x_0,\xi)]$, for $1\leq j,k,l\leq q$.
There exists an $\Zygad{s_0}$-admissible constant $\eta_1\in (0,\eta_0]$ such that:
\begin{enumerate}[resume*=partonetheoremenumeration]
\item\label{Item::PartI::OpenImage} $\Phi_0(B^n(\eta_1))$ is an open subset of $B_{X_{J_0}}(x_0,\chi)$ and is therefore a submanifold of $B_X(x_0,\xi)$.
\item\label{Item::PartI::PhiDiffeo} $\Phi_0:B^n(\eta_1)\rightarrow \Phi_0(B^n(\eta_1))$ is a $C^2$ diffeomorphism.
\end{enumerate}
Let $Y_j:=\Phi_0^{*}X_j$, and write $Y_{J_0}=(I+A)\grad$, where $\grad$ denotes the gradient in $\R^n$ (thought of as a column vector) and we are identifying $Y_{J_0}$ with the
column vector of vector fields $[Y_1, Y_2,\ldots, Y_n]^{\transpose}$.
\begin{enumerate}[resume*=partonetheoremenumeration]
\item\label{Item::PartI::ABound} $A(0)=0$ and $\sup_{t\in B^n(\eta_1)} \Norm{A(t)}[\M^{n\times n}]\leq \frac{1}{2}$.
\item\label{Item::PartI::Yreg} For $s>0$, $1\leq j\leq q$,
\begin{equation}\label{Eqn::PartI::RegularityOfYs}
\ZygNorm{Y_j}{s}[B^n(\eta_1)][\R^n]\lesssim_{\Zygad{s}} 1.
\end{equation}
\item\label{Item::PartI::Bs} 
There exist $b_k^l\in \ZygSpace{s_0+1}[B^n(\eta_1)]$, $n+1\leq k\leq q$, $1\leq l\leq n$, such that $Y_k=\sum_{l=1}^n b_k^l Y_l$ and
$\ZygNorm{b_k^l}{s}[B^n(\eta_1)]\lesssim_{\Zygad{s-1}} 1$, $\forall s>0$.
\item\label{Item::PartI::Commutators} For $1\leq j,k\leq n$, $[Y_j,Y_k]=\sum_{l=1}^n \ct_{j,k}^l Y_l$, where for $s>0$,
\begin{equation*}
\ZygNorm{\ct_{j,k}^l}{s}[B^n(\eta_1)]\lesssim_{\Zygad{s}} 1.
\end{equation*}
\end{enumerate}
\end{prop}

The statement of \cite[\SSMainResult]{StovallStreet} uses ``$1$-admissible constants'' which we have not defined here.  However,
it is easy to see that $1$-admissible constants are $\Zygad{s_0}$-admissible constants for $s_0>1$, and so \cref{Prop::PartI::MainResult}
follows from \cite[\SSMainResult]{StovallStreet}.

\begin{rmk}
The main difference between \cref{Prop::PartI::MainResult} and \cref{Thm::QuantRes::MainThm} can be seen by comparing
\cref{Eqn::PartI::RegularityOfYs} and \cref{Eqn::QuantRes::RegularityYs}:  \cref{Eqn::QuantRes::RegularityYs}
is stronger than \cref{Eqn::PartI::RegularityOfYs} by one derivative.
The central point of this paper is to obtain this stronger (sharp) result.
\end{rmk} 

	\subsection{Densities}
We describe the results on densities from \cite[\SSDensities]{StovallStreet} needed in this paper.
The setting is the same as in \cref{Section::Results::Densities}; thus we are given a $C^1$ density $\nu$
on $B_{X_{J_0}}(x_0,\xi)$ satisfying $\Lie{X_j}\nu = f_j\nu$.
$\Zygsonu$ and $\Zygad{s;\nu}$-admissible constants are defined as in that section (\cref{Defn::Densities::BasicAdmissibleConsts,Defn::Densitities::AdmissibleConsts}).  We also use another type of admissible constant.
As before, we reorder the vector fields so that $J_0=(1,\ldots, n)$.

\begin{defn}
We say $C$ is a $0;\nu$-admissible constant if $C$ is a $0$-admissible constant which is also allowed to depend on
upper bounds for $\CNorm{f_j}{B_{X_{J_0}}(x_0,\chi)}$, $1\leq j\leq n$.  We write $A\lesssim_{0;\nu} B$
for $A\leq CB$, where $C$ is a $0;\nu$-admissible constant, and write $A\approx_{0;\nu} B$ for
$A\lesssim_{0;\nu} B$ and $B\lesssim_{0;\nu} A$.  Note that $0;\nu$-admissible constants are $\Zygsonu$-admissible constants.
\end{defn}

We introduce a distinguished density on $B_{X_{J_0}}(x_0,\chi)$ given by
\begin{equation}\label{Eqn::PartI::Densities::Defn::nu0}
\nu_0(Z_1,\ldots, Z_n) := \left|\frac{Z_1\wedge Z_2\wedge \cdots \wedge Z_n}{X_1\wedge X_2\wedge \cdots \wedge X_n}\right|.
\end{equation}
Note that $X_1\wedge X_2\wedge \cdots \wedge X_n$ is never zero on $B_{X_{J_0}}(x_0,\chi)$ (by \cref{Prop::PartI::MainResult} \cref{Item::PartI::NonzeroJ0}), so that $\nu_0$
is defined on $B_{X_{J_0}}(x_0,\chi)$.  It is clearly a density.

\begin{prop}\label{Prop::PartI::Densities::MainProp}
Given a $C^1$ density $\nu$ as above, there exists $g\in \CSpace{B_{X_{J_0}}(x_0,\chi)}$ such that $\nu=g\nu_0$ and
\begin{enumerate}[(a)]
\item\label{Item::PartI::Densities::gconst} $g(x)\approx_{0;\nu} g(x_0)=\nu(X_1,\ldots, X_n)(x_0)$, $\forall x\in B_{X_{J_0}}(x_0,\chi)$.  In particular, $g$ always has the same sign,
and is either never zero or always zero.
\item\label{Item::PartI::Densities::gZyg} For $s>0$, $1\leq j\leq n$, we have $\ZygXNorm{g}{X_{J_0}}{s}[B_{X_{J_0}}(x_0,\chi)] \lesssim_{\Zygad{s-1;\nu}} |\nu(X_1,\ldots, X_n)(x_0)|$.
\end{enumerate}
\end{prop}
\begin{proof}This is an immediate consequence of \cite[\SSDensitiesTheorem]{StovallStreet}.\end{proof} 
	
\section{Function Spaces}\label{Section::FuncSpace}
In this section, we define the function spaces which are used in this paper as well as discuss the main properties we use.
These spaces were all defined in \cite{StovallStreet}, and we refer the reader to that paper for a more detailed discussion these spaces.
As in that paper, we make a distinction between function spaces on open subsets of $\R^n$ and function spaces on a $C^2$ manifold $M$.
Open subsets of $\R^n$ have a natural smooth structure, and it makes sense to talk about the usual function spaces on these open sets.
On a $C^2$ manifold $M$, it does not make sense to talk about, for example, $C^\infty$ functions.  However, if we are also given $C^1$ vector fields
$X_1,\ldots, X_q$ on $M$, it makes sense to talk about functions which are smooth with respect to these vector fields, and that is how we proceed.

	\subsection{Function Spaces on Euclidean Space}\label{Section::Zygmund}
In this section, we describe the standard function spaces on $\R^n$ which we use.
Let $\Omega\subset \R^n$ be a bounded, connected, open set (we will almost always be considering the case when $\Omega$ is a ball in $\R^n$).
We have the following classical Banach spaces of functions on $\Omega$:
\begin{equation*}
\CSpace{\Omega}= \CjSpace{0}[\Omega]:=\{f:\Omega\rightarrow \C\: \big|\: f\text{ is continuous and bounded}\}, \quad \CNorm{f}{\Omega}=\CjNorm{f}{0}[\Omega]:=\sup_{x\in \Omega}|f(x)|.
\end{equation*}
For $m\in \N$,
\begin{equation*}
	\CjSpace{m}[\Omega]:=\{ f\in \CjSpace{0}[\Omega]\: \big|\: \partial_x^{\alpha} f\in \CjSpace{0}[\Omega], \forall |\alpha|\leq m\}, \quad \CjNorm{f}{m}[\Omega]:=\sum_{|\alpha|\leq m} \CjNorm{\partial_x^\alpha f}{0}[\Omega].
\end{equation*}
Next we define the classical Lipschitz-H\"older spaces.  For $s\in [0,1]$,
\begin{equation}\label{Eqn::FuncSpacesEuclid::Defn::HNorm}
	\HNorm{f}{0}{s}[\Omega]:=\CjNorm{f}{0}[\Omega] + \sup_{\substack{ x,y\in \Omega \\ x\ne y}} |x-y|^{-s} |f(x)-f(y)|, \quad \HSpace{0}{s}[\Omega]:=\{ f\in \CjSpace{0}[\Omega] : \HNorm{f}{0}{s}[\Omega]<\infty\}.
\end{equation}
For $m\in \N$, $s\in [0,1]$,
\begin{equation*}
	\HNorm{f}{m}{s}[\Omega]:=\sum_{|\alpha|\leq m} \HNorm{\partial_x^{\alpha} f}{0}{s}[\Omega],\quad \HSpace{m}{s}[\Omega]:=\{f\in \CjSpace{m}[\Omega] : \HNorm{f}{m}{s}[\Omega]<\infty\}.
\end{equation*}
Next, we turn to the Zygmund-H\"older spaces.
Given $h\in \R^n$ define $\Omega_h:=\{x\in \R^n : x,x+h,x+2h\in \Omega\}$.
For $s\in (0,1]$ set
\begin{equation*}
	\ZygNorm{f}{s}[\Omega]:=\HNorm{f}{0}{s/2}[\Omega]+\sup_{\substack{0\ne h\in \R^n \\ x\in \Omega_h}} |h|^{-s} |f(x+2h)-2f(x+h)+f(x)|,
\end{equation*}
\begin{equation*}
	\ZygSpace{s}[\Omega]:=\{f\in \CjSpace{0}[\Omega]: \ZygNorm{f}{s}[\Omega]<\infty\}.
\end{equation*}
For $m\in \N$, $s\in (0,1]$, set
\begin{equation*}
	\ZygNorm{f}{m+s}[\Omega]:=\sum_{|\alpha|\leq m} \ZygNorm{\partial_x^{\alpha} f}{s}[\Omega], \quad \ZygSpace{s+m}[\Omega]:=\{f\in \CjSpace{m}[\Omega] : \ZygNorm{f}{m+s}[\Omega]<\infty\}.
\end{equation*}
We set
\begin{equation*}
	\ZygSpace{\infty}[\Omega]:=\bigcap_{s>0} \ZygSpace{s}[\Omega] ,\quad   \CjSpace{\infty}[\Omega]:= \bigcap_{m\in \N} \CjSpace{m}[\Omega].
\end{equation*}
It is straightforward to verify that for a ball $B$, $\ZygSpace{\infty}[B]=\CjSpace{\infty}[B]$.
For a Banach space $V$, we let $\CSpace{\Omega}[V]$, $\CjSpace{m}[\Omega][V]$, $\HSpace{m}{s}[\Omega][V]$, and $\ZygSpace{s}[\Omega][V]$
denote the analogous spaces of functions taking values in $V$.
By identifying a vector field $Y=\sum_{j=1}^n a_j \diff{t_j}$ on $\Omega$ with the function $(a_1,\ldots, a_n):\Omega\rightarrow \R^n$, it makes sense
to write, for example, $\ZygNorm{Y}{s}[\Omega][\R^n]$.

\begin{rmk}\label{Rmk::ZygSpace::WeirdDefn}
The term $\HNorm{f}{0}{s/2}[\Omega]$ in the definition of $\ZygNorm{f}{s}[\Omega]$ is somewhat unusual, and is usually replaced by $\CNorm{f}{\Omega}$.
As is well-known, if $\Omega$ is a
bounded Lipschitz domain,
these two choices yield equivalent norms (this follows easily from \cite[Theorem 1.118 (i)]{TriebelTheoryOfFunctionSpacesIII}).  However,
the constants involved in this equivalence depend on $\Omega$.  In this paper, we will almost always be considering the case $\Omega=B^n(\eta)$,
for some explicit choice of $\eta$.  Thus, the difference between these two possible definitions of $\ZygNorm{f}{s}[\Omega]$ will not affect any of
the results in this paper.
The choice we have made here is slightly more convenient for some of our purposes; see \cite[\SSStrangeZygSpace]{StovallStreet} for more comments on this.
\end{rmk}

\begin{defn}
For $s\in (0,\infty]$, we say $f\in \ZygSpaceloc{s}[\Omega]$ if $\forall x\in \Omega$, there exists an open ball $B\subseteq \Omega$, centered at $x$,
with $f\big|_B\in \ZygSpace{s}[B]$.
\end{defn}

\begin{rmk}\label{Rmk::FuncSpaceEuclid::HolderAndZygTheSame}
If $\Omega$ is a bounded Lipschitz domain, $m\in \N$, $s\in (0,1)$, the spaces $\HSpace{m}{s}[\Omega]$ and $\ZygSpace{m+s}[\Omega]$ are the same--see
 \cite[Theorem 1.118 (i)]{TriebelTheoryOfFunctionSpacesIII};  however, if $s\in \{0,1\}$, these spaces differ.  As a consequence, for \textit{any} open set
 $\Omega\subseteq \R^n$, for $m\in \N$, $s\in (0,1)$, we have $\ZygSpaceloc{m+s}[\Omega]$ equals the space of functions which are locally in $\HSpace{m}{s}$.
 The space $\ZygSpaceloc{\infty}[\Omega]$ equals the usual space of functions which are locally smooth on $\Omega$.
\end{rmk} 
	
	\subsection{Function Spaces on Manifolds}\label{Section::FuncSpace::Manifold}
Let $X_1,\ldots, X_q$ be $C^1$ vector fields on a connected $C^2$ manifold $M$.  Corresponding to $X_1,\ldots, X_q$, we have a sub-Riemannian metric given by \cref{Eqn::Res::rho}.
We use ordered multi-index notation: $X^\alpha$.  Here, $\alpha$ denotes a list of elements $\{1,\ldots, q\}$ and $|\alpha|$ denotes the length of the list.
For example $X^{(2,1,3,1)}=X_2X_1X_3X_1$ and $|(2,1,3,1)|=4$.

Associated to the vector fields $X_1,\ldots, X_q$, we have the following
Banach spaces of functions on $M$.
\begin{equation*}
    \CSpace{M}=\CXjSpace{X}{0}[M]:=\{f:M\rightarrow \C\:\big|\: f\text{ is continuous and bounded}\}, \quad
    \CNorm{f}{M}=\CXjNorm{f}{X}{0}[M]:=\sup_{x\in M} |f(x)|.
\end{equation*}
For $m\in \N$, we define
\begin{equation*}
    \CXjSpace{X}{m}[M]:=\{f\in \CSpace{M}\: \big|\: X^{\alpha} f\text{ exists and } X^\alpha f\in \CSpace{M},\forall |\alpha|\leq m\}, \quad \CXjNorm{f}{X}{m}[M]:=\sum_{|\alpha|\leq m} \CNorm{X^{\alpha} f}{M}.
\end{equation*}
For $s\in [0,1]$, we define the Lipschitz-H\"older space associated to $X$
by
\begin{equation*}
    \HXNorm{f}{X}{0}{s}[M]:=\CNorm{f}{M}+\sup_{\substack{x,y\in M \\ x\ne y}} \rho(x,y)^{-s} |f(x)-f(y)|, \quad \HXSpace{X}{0}{s}[M]:=\{f\in \CSpace{M} : \HXNorm{f}{X}{0}{s}[M]<\infty\}.
\end{equation*}
For $m\in \N$ and $s\in [0,1]$, set
\begin{equation*}
    \HXNorm{f}{X}{m}{s}[M]:=\sum_{|\alpha|\leq m} \HXNorm{X^{\alpha}f}{X}{0}{s}[M], \quad \HXSpace{X}{m}{s}[M]:=\{f\in \CXjSpace{X}{m}[M] : \HXNorm{f}{X}{m}{s}[M]<\infty\}.
\end{equation*}

We turn to the Zygmund-H\"older spaces.  For this, we use the H\"older spaces
$\HSpace{0}{s}[[a,b]]$ for a closed interval $[a,b]\subset \R$; $\HNorm{\cdot}{0}{s}[[a,b]]$
is defined via the same formula  as in \cref{Eqn::FuncSpacesEuclid::Defn::HNorm}.
Given $h>0$, $s\in (0,1)$ define
\begin{equation*}
    \sP_{X,s}^M(h):=\left\{ \gamma:[0,2h]\rightarrow M\:\bigg|\: \gamma'(t)=\sum_{j=1}^q d_j(t) X_j(\gamma(t)), d_j\in \HSpace{0}{s}[[0,2h]], \sum_{j=1}^q \HNorm{d_j}{0}{s}[[0,2h]]^2<1 \right\}.
\end{equation*}
For $s\in (0,1]$ set
\begin{equation*}
    \ZygXNorm{f}{X}{s}[M]:=\HXNorm{f}{X}{0}{s/2}[M]+\sup_{\substack{h>0 \\ \gamma\in \sP^M_{X,s/2 }(h) }}h^{-s}\left|f(\gamma(2h))-2f(\gamma(h))+f(\gamma(0))\right|,
\end{equation*}
and for $m\in \N$,
\begin{equation*}
    \ZygXNorm{f}{X}{m+s}[M]:=\sum_{|\alpha|\leq m} \ZygXNorm{X^{\alpha}f}{X}{s}[M],
\end{equation*}
and we set
\begin{equation*}
    \ZygXSpace{X}{s+m}[M]:=\{f\in \CXjSpace{X}{m}[M] : \ZygXNorm{f}{X}{m+s}[M]<\infty\}.
\end{equation*}
Set
\begin{equation*}
    \ZygXSpace{X}{\infty}[M]:=\bigcap_{s>0} \ZygXSpace{X}{s}[M]\text{ and } \CXjSpace{X}{\infty}[M]:=\bigcap_{m\in \N} \CXjSpace{X}{m}[M].
\end{equation*}
It is a consequence of \cite[\SSCompareFunctionSpaces]{StovallStreet} that $\ZygXSpace{X}{\infty}[M]=\CXjSpace{X}{\infty}[M]$;
indeed, $\ZygXSpace{X}{\infty}[M]\subseteq \CXjSpace{X}{\infty}[M]$ is clear while the reverse containment
follows from \cite[\SSCompareFunctionSpaces]{StovallStreet}.
For more details on these spaces, we refer the reader to \cite{StovallStreet}.

\begin{rmk}
When we write $Vf$ for a $C^1$ vector field $V$ and $f:M\rightarrow \R$, we define this as $Vf(x):=\frac{d}{dt}\big|_{t=0} f(e^{tV} x)$.
When we say $Vf$ exists, it mean that this derivative exists in the classical sense, $\forall x$.  If we have several $C^1$ vector fields $V_1,\ldots, V_L$,
we define $V_1V_2\cdots V_L f:= V_1(V_2(\cdots V_L (f)))$ and to say that this exists means that at each stage the derivatives exist.
\end{rmk}

\begin{rmk}
For certain subsets of $M$ which are not themselves manifolds, we can still define the
above norms.  Indeed, let $X_1,\ldots, X_q$ be $C^1$ vector fields on a $C^2$ manifold $M$ and fix $\xi>0$.
In this setting, $B_X(x_0,\xi)$ might not be a manifold (though it sometimes is--see \cref{Prop::QualRes::InjectiveImmers}).
$B_X(x_0,\xi)$ is a metric space, with the metric $\rho$.
For a function $f:B_X(x_0,\xi)\rightarrow \C$ and $x\in B_X(x_0,\xi)$, it makes sense to consider
$X_j f(x):=\frac{d}{dt}\big|_{t=0} f(e^{tX_j }x)$.
Using this, we can define the spaces
$\HXSpace{X}{m}{s}[B_X(x_0,\xi)]$ and $\ZygXSpace{X}{s}[B_X(x_0,\xi)]$,
and their corresponding norms, with the same formulas as above.
\end{rmk} 
	
	\subsection{Some Results on Function Spaces}
In this section, we present some results concerning the above function spaces which we need later in the paper.
Many of these results are standard and easy to prove; however a main goal of this section is to precisely state what each
estimate depends on, as that is essential for our main results.

\begin{lemma}\label{Lemma::ZygSpaces::EquivHolderZyg}
For $m\in \N$, $s\in (0,1)$, $\eta>0$,
\begin{equation}\label{Eqn::ToShow::EquivHolderZyg}
\HNorm{f}{m}{s}[B^n(\eta)]\approx \ZygNorm{f}{m+s}[B^n(\eta)],
\end{equation}
where the implicit constants depend on $n$, $m$, $s$, and an upper bound for $\eta^{-1}$.  Furthermore,
for $m\in \N$, $s\in (0,1]$, $r\in (m+s,\infty)$,
\begin{equation}\label{Eqn::ToShow::HolderLessZyg}
\HNorm{f}{m}{s}[B^n(\eta)]\lesssim \ZygNorm{f}{r}[B^n(\eta)],
\end{equation}
where the implicit constant depends on $n$, $m$, $s$, $r$, and an upper bound for $\eta^{-1}$.
\end{lemma}
\begin{proof}
It suffices to prove \cref{Eqn::ToShow::EquivHolderZyg} in the case $m=0$.
When $\eta=1$, \cref{Eqn::ToShow::EquivHolderZyg} (with $m=0$) follows easily from \cite[Theorem 1.118 (i)]{TriebelTheoryOfFunctionSpacesIII} (by considering the cases $M=1,2$ in that theorem).
For general $\eta$, \cref{Eqn::ToShow::EquivHolderZyg} (with $m=0$) follows from the case $\eta=1$ and a simple scaling argument which we leave to the reader.
\Cref{Eqn::ToShow::HolderLessZyg} follows immediately from \cref{Eqn::ToShow::EquivHolderZyg}.
\end{proof}

\begin{lemma}\label{Lemma::ZygSpace::Algebra}
The spaces $\HXSpace{X}{m}{s}[M]$, $\ZygXSpace{X}{s}[M]$, $\HSpace{m}{s}[\Omega]$, and $\ZygSpace{s}[\Omega]$
are algebras.  In fact, we have for $m\in \N$, $s\in [0,1]$,
\begin{equation*}
    \HXNorm{fg}{X}{m}{s}[M]\leq C_{m,q} \HXNorm{f}{X}{m}{s}[M]\HXNorm{g}{X}{m}{s}[M],
\end{equation*}
where $C_{m,q}$ is a constant depending only on $m$ and $q$.  And for $m\in \N$, $s\in (m,m+1]$,
\begin{equation}\label{Eqn::FuncSpacesRev::ZygAlg}
    \ZygXNorm{fg}{X}{s}[M]\leq C_{m,q} \ZygXNorm{f}{X}{s}[M]\ZygXNorm{g}{X}{s}[M].
\end{equation}
Moreover, these algebras have multiplicative inverses for functions which are bounded away from zero.  If $f\in \HXSpace{X}{m}{s}[M]$ with $\inf_{x\in M} |f(x)|\geq c_0>0$
then $f(x)^{-1}=\frac{1}{f(x)}\in \HXSpace{X}{m}{s}[M]$ with
\begin{equation*}
	\HXNorm{f(x)^{-1}}{X}{m}{s}[M]\leq C,
\end{equation*}
where $C$ can be chosen to depend only on $m$, $q$, $c_0$, and an upper bound for $\HXNorm{f}{X}{m}{s}[M]$.
And for $m\in \N$, $s\in (m,m+1]$ if $f\in \ZygXSpace{X}{s}[M]$ with $\inf_{x\in M}|f(x)|\geq c_0>0$ then
$f(x)^{-1}\in \ZygXSpace{X}{s}[M]$ with
\begin{equation}\label{Eqn::FuncSpacesRev::ZygInv}
	\ZygXNorm{f(x)^{-1}}{X}{s}[M]\leq C,
\end{equation}
where $C$ can be chosen to depend only on $m$, $q$, $c_0$, and an upper bound for $\ZygXNorm{f}{X}{s}[M]$.
The same results hold with $\HXSpace{X}{m}{s}[M]$ replaced by $\HSpace{m}{s}[\Omega]$ and $\ZygXSpace{X}{s}[M]$ replaced by $\ZygSpace{s}[\Omega]$ (with $n$ playing
the role of $q$).
\end{lemma}
\begin{proof}
	This is \cite[\SSZygIsAlgebra]{StovallStreet}.
\end{proof}

\begin{lemma}\label{Lemma::ZygSpace::Comp}
Let $D_1,D_2>0$, $s_1>0$, $s_2\geq s_1$, $s_2>1$,
$f\in \ZygSpace{s_1}[B^n(D_1)]$, $g\in \ZygSpace{s_2}[B^m(D_2)][\R^n]$ with $g(B^m(D_2))\subseteq B^n(D_1)$.
Then, $f\circ g\in \ZygSpace{s_1}[B^m(D_2)]$ and $\ZygNorm{f\circ g}{s_1}[B^m(D_2)]\leq C \ZygNorm{f}{s_1}[B^n(D_1)]$ where $C$
can be chosen to depend only on $s_1$, $s_2$, $D_1$, $D_2$, $m$, $n$, and an upper bound for
$\ZygNorm{g}{s_2}[B^m(D_2)]$.

Furthermore, if $s_1\in (0,1)$, $f$ is as above, and $g\in \CjSpace{1}[B^m(D_2)][\R^n]$ with $g(B^m(D_2))\subseteq B^n(D_1)$,
then $f\circ g\in \ZygSpace{s_1}[B^m(D_2)]$ and $\ZygNorm{f\circ g}{s_1}[B^m(D_2)]\leq C \ZygNorm{f}{s_1}[B^n(D_1)]$ where $C$
can be chosen to depend only on $s_1$, $D_1$, $D_2$, $n$, and an upper bound for
$\CjNorm{g}{1}[B^m(D_2)]$.
\end{lemma}
\begin{proof}
We use the notation $A\lesssim B$ for $A\leq CB$ where $C$ is as in the statement of the lemma.
Without loss of generality, we assume $\ZygNorm{f}{s_1}[B^n(D_1)]=1$.
We prove the first claim by induction on $k$, where $s_1\in (k,k+1]$.

We begin with the base case $k=0$
so that $s_1\in (0,1]$.  We use $y$ to denote elements of $\R^n$ and $x$ to denote elements of $\R^m$.
Since $s_1\in (0,1]$, we may, without loss of generality, assume $s_2\in (1,2)$; indeed, if $s_2\geq 2$ we may replace $s_2$ with $3/2$ in the proof that follows.
Since $\CjNorm{g}{1}[B^m(D_2)][\R^n]\leq \ZygNorm{g}{s_2}[B^m(D_2)][\R^n]\lesssim 1$, it is immediate to verify that
$\HNorm{f\circ g}{0}{s_1/2}[B^m(D_2)]\lesssim 1$.    Let $x,h\in \R^m$ be such that $x,x+h,x+2h\in B^m(D_2)$.
We wish to show
\begin{equation}\label{Eqn::ZygSpace::ToShowBase}
|f\circ g(x+2h)-2f\circ g(x+h)+f\circ g(x)|\lesssim |h|^{s_1},
\end{equation}
which will complete the proof of the base case.
Define $\gamma:[0,2h]\rightarrow B^n(D_1)$ by $\gamma(t)=g\left(x+\frac{h}{|h|} t\right)$.
Let $r:=(s_2-1)/2\in (0,s_2-1)$.  We use the classical fact that $\HNorm{g}{1}{r}[B^m(D_2)][\R^n]\lesssim \ZygNorm{g}{s_2}[B^m(D_2)][\R^n]\lesssim 1$
(see \cref{Lemma::ZygSpaces::EquivHolderZyg}).
Thus, $\HNorm{\gamma}{1}{r}[[0,2|h|]][\R^n]\lesssim 1$.

Set $\gammat(t)=\frac{t}{2|h|} g(x+2h)+\left(1-\frac{t}{2|h|}\right) g(x) = \frac{t}{2|h|} \gamma(2|h|)+\left(1-\frac{t}{2|h|}\right)\gamma(0)$, so that
$\gammat:[0,2|h|]\rightarrow B^n(D_1)$ is a line segment of length $|g(x+2h)-g(x)|\leq 2|h|\CjNorm{g}{1} \lesssim |h|$.  Thus, we have
\begin{equation*}
	|f(\gammat(2|h|)) - 2f(\gammat(|h|)) + f(\gammat(0))| \lesssim |h|^{s_1}.
\end{equation*}
For $t\in [0,2|h|]$, we have
\begin{equation*}
|\gammat(t)-\gamma(t)| = t\left| \frac{\gamma(2|h|)-\gamma(0)}{2|h|} - \frac{\gamma(t)-\gamma(0)}{t} \right| = t|\gamma'(c_1)-\gamma'(c_2)|,
\end{equation*}
for some $c_1,c_2\in [0,2|h|]$ by the mean value theorem.  Thus,
\begin{equation*}
|\gammat(t)-\gamma(t)| \leq t |c_1-c_2|^r \HNorm{\gamma}{1}{r}\lesssim |h|^{1+r}.
\end{equation*}
We again use the classical fact that $\HNorm{f}{0}{s_1/(1+r)}[B^n(D_1)]\lesssim \ZygNorm{f}{s_1}[B^n(D_1)]\leq 1$ (see \cref{Lemma::ZygSpaces::EquivHolderZyg}).
Thus, we have
\begin{equation*}
\begin{split}
&|f\circ g(x+2h)-2f\circ g(x+h)+f\circ g(x)| = |f(\gamma(2|h|))-2f(\gamma(|h|))+f(\gamma(0))|
\\&\leq |f(\gammat(2|h|)) - 2f(\gammat(|h|))+f(\gammat(0))| + 2|f(\gammat(|h|))-f(\gamma(|h|))|
\\& \lesssim |h|^{s_1} + |\gammat(|h|)-\gamma(|h|)|^{s_1/(1+r)} \HNorm{f}{0}{s_1/(1+r)}[B^n(D_1)] \lesssim |h|^{s_1},
\end{split}
\end{equation*}
completing the proof of \cref{Eqn::ZygSpace::ToShowBase}, and therefore the proof of the base case.

Now take $s_1>1$ and we assume the result for $s_1-1$.  We have,
\begin{equation*}
\begin{split}
\ZygNorm{f\circ g}{s_1}[B^m(D_2)] \leq \sum_{j=1}^m \BZygNorm{\diff{x_j} (f\circ g)}{s_1-1}[B^m(D_2)] + \ZygNorm{f\circ g}{s_1-1}[B^m(D_2)].
\end{split}
\end{equation*}
$\ZygNorm{f\circ g}{s_1-1}[B^m(D_2)]\lesssim 1$ by the inductive hypothesis, so it suffices to estimate $\BZygNorm{\diff{x_j} (f\circ g)}{s_1-1}[B^m(D_2)]$.
We have, using \cref{Lemma::ZygSpace::Algebra},
\begin{equation*}
\begin{split}
&\BZygNorm{\diff{x_j} (f\circ g)}{s_1-1}[B^m(D_2)] \leq \sum_{l=1}^n \BZygNorm{\left(\frac{\partial f}{\partial y_l}\circ g\right) \frac{\partial g_l}{\partial x_j}}{s_1-1}[B^m(D_2)]
\\&\lesssim  \sum_{l=1}^n\BZygNorm{\frac{\partial f}{\partial y_l}\circ g }{s_1-1}[B^m(D_2)]  \BZygNorm{\frac{\partial g_l}{\partial x_j}}{s_1-1}[B^m(D_2)].
\end{split}
\end{equation*}
The inductive hypothesis shows $\BZygNorm{\frac{\partial f}{\partial y_l}\circ g }{s_1-1}[B^m(D_2)]\lesssim 1$, and
$\BZygNorm{\frac{\partial g_l}{\partial x_j}}{s_1-1}[B^m(D_2)]\lesssim \ZygNorm{g}{s_1}[B^m(D_2)]\lesssim \ZygNorm{g}{s_2}[B^m(D_2)]\lesssim 1$, since $s_2\geq s_1$.
Combining the above estimates shows $\ZygNorm{f\circ g}{s_1}[B^m(D_2)] \lesssim 1$, and completes the proof of the induction.

Finally, we turn to the case when $s_1\in (0,1)$ and $g\in \CjSpace{1}[B^m(D_2)][\R^n]$.  In this case, the same proof as the base case above works, by taking $r=0$
throughout.
Here, we use the
\cref{Lemma::ZygSpaces::EquivHolderZyg}
to see $\HNorm{f}{0}{s_1}[B^n(D_1)]\lesssim \ZygNorm{f}{s_1}[B^n(D_1)]$, for $s_1\in (0,1)$. 
\end{proof}

\begin{lemma}\label{Lemma::ZygSpace::Inverse}
Fix $s>1$, $D_1,D_2>0$.  Suppose $H\in \ZygSpace{s}[B^n(D_1)][\R^n]$ is such that $B^n(D_2)\subseteq H(B^n(D_1))$,
$H:B^n(D_1)\rightarrow H(B^n(D_1))$ is a homeomorphism, and $\inf_{t\in B^n(D_1)} \left|\det dH(t)\right|\geq c_0>0$.
Then, $H^{-1}\in \ZygSpace{s}[B^n(D_2)][\R^n]$, with $\ZygNorm{H^{-1}}{s}[B^n(D_2)][\R^n]\leq C$,
where $C$ can be chosen to depend only on $n$, $s$, $D_1$, $D_2$, $c_0$, and an upper bound
for $\ZygNorm{H}{s}[B^n(D_1)][\R^n]$.
\end{lemma}
\begin{proof}
We use $A\lesssim B$ for $A\leq C B$, where $C$ is as in the statement of the lemma.
Since $\CjNorm{H}{1}[B^n(D_1)][\R^n]\leq \ZygNorm{H}{s}[B^n(D_1)][\R^n]\lesssim 1$, we have
$H^{-1}\in \CjSpace{1}[B^n(D_2)][\R^n]$ and
$\CjNorm{H^{-1}}{1}[B^n(D_2)][\R^n]\lesssim 1$.
Thus, it suffices to show
\begin{equation}\label{Eqn::ZygSpace::Inverse::ToShow}
d(H^{-1})\in \ZygSpace{s-1}[B^n(D_2)][\M^{n\times n}]\text{ with }\ZygNorm{d(H^{-1})}{s-1}[B^n(D_2)][\M^{n\times n}]\lesssim 1.
\end{equation}
We use the formula
\begin{equation}\label{Eqn::ZygSpace::Inverse::ChainRule}
d(H^{-1})(t) = (dH(H^{-1}(t)))^{-1}.
\end{equation}
  From our hypotheses, we have $\ZygNorm{dH}{s-1}[B^n(D_1)][\M^{n\times n}]\lesssim 1$.
Since $\inf_{t\in B^n(D_1)} \left|\det dH(t)\right|\gtrsim 1$, using the cofactor representation of $v\mapsto (dH(v))^{-1}$ and applying \cref{Lemma::ZygSpace::Algebra},
we have
\begin{equation}\label{Eqn::ZygSpace::Inverse::BeforeChain}
\ZygNorm{(dH)^{-1}}{s-1}[B^n(D_1)][\M^{n\times n}]\lesssim 1.
\end{equation}
%

We begin by proving \cref{Eqn::ZygSpace::Inverse::ToShow} in the case $s\in (1,2)$.  Since $\ZygNorm{(dH)^{-1}}{s-1}[B^n(D_1)][\M^{n\times n}]\lesssim 1$ and $\CjNorm{H^{-1}}{1}[B^n(D_1)][\R^n]\lesssim 1$,
it follows from \cref{Lemma::ZygSpace::Comp} (using \cref{Eqn::ZygSpace::Inverse::ChainRule}) that $\ZygNorm{d(H^{-1})}{s-1}[B^n(D_2)][\M^{n\times n}]\lesssim 1$, which completes the proof of \cref{Eqn::ZygSpace::Inverse::ToShow} in this case.

We now proceed by induction.  Take $m\geq 2$ and suppose we know the lemma for $s\in (1,m)$ and we wish to prove \cref{Eqn::ZygSpace::Inverse::ToShow} for $s\in [m,m+1)$.
Fix $s\in [m,m+1)$.  Take $s_1=\frac{m+1 +s}{2}-1\in (m-1,m)$; note that $s-1<s_1$.  By our inductive hypothesis, we have
$H^{-1}\in \ZygSpace{s_1}[B^n(D_2)][\R^n]$, with $\ZygNorm{H^{-1}}{s_1}[B^n(D_2)][\R^n]\lesssim 1$.  Combining this with $\ZygNorm{(dH)^{-1}}{s-1}[B^n(D_1)][\M^{n\times n}]\lesssim 1$ (as shown in \cref{Eqn::ZygSpace::Inverse::BeforeChain}) and using \cref{Eqn::ZygSpace::Inverse::ChainRule},
 \cref{Lemma::ZygSpace::Comp} shows that $\ZygNorm{d(H^{-1})}{s-1}[B^n(D_2)][\M^{n\times n}]\lesssim 1$, which completes the proof.
\end{proof}

\begin{lemma}\label{Lemma::FuncSpaceRev::Scale}
Let $m\in \N$ with $m\geq 1$, $s\in (0,1]$, and $\eta_1>0$.  For $f\in \ZygSpace{m+s}[B^n(\eta_1)]$ and $\gamma\in (0,1]$, set $f_\gamma(t):=f(\gamma t)$.  Then, for $0<\gamma\leq \min\{\frac{\eta_1}{5},1\}$, we have
for $f\in \ZygSpace{m+s}[B^n(\eta_1)]$ with $f(0)=0$,
\begin{equation*}
\ZygNorm{f_\gamma}{m+s}[B^n(5)] \leq \gamma 91 \ZygNorm{f}{m+s}[B^n(\eta_1)].
\end{equation*}
\end{lemma}
\begin{proof}
Using $\gamma\in (0,1]$, it follows immediately from the definitions that
\begin{equation}\label{Eqn::FuncSpaceRev::EstHigherDerivs}
\begin{split}
&\sum_{1\leq |\alpha|\leq m} \ZygNorm{\partial_x^{\alpha} f_\gamma}{s}[B^n(5)] = \sum_{1\leq |\alpha|\leq m} \gamma^{|\alpha|}\ZygNorm{ (\partial_x^{\alpha} f)(\gamma \cdot)}{s}[B^n(5)]
\\&\leq \sum_{1\leq |\alpha|\leq m} \gamma^{|\alpha|}\ZygNorm{\partial_x^{\alpha} f}{s}[B^n(\eta_1)]
\leq \gamma \ZygNorm{f}{m+s}[B^n(\eta_1)].
\end{split}
\end{equation}
Since $f_\gamma(0)=f(0)=0$, we have (using the Fundamental Theorem of Calculus)
\begin{equation}\label{Eqn::FuncSpaceRev::EstC1}
\CjNorm{f_\gamma}{1}[B^n(5)] = \CjNorm{f_\gamma}{0}[B^n(5)] + \sum_{|\alpha|=1} \CjNorm{\partial_x^{\alpha} f_\gamma}{0}[B^n(5)]
\leq 6 \sum_{|\alpha|=1} \CjNorm{\partial_x^{\alpha} f_{\gamma}}{0}[B^n(5)]
\leq 6\gamma \CjNorm{f}{1}[B^n(\eta_1)].
\end{equation}
Directly from the definitions (see also \cite[\SSCompareFunctionSpaces]{StovallStreet}), we have (for any ball $B$ and any function $g$)
\begin{equation*}
\ZygNorm{g}{s}[B]\leq 5 \HNorm{g}{0}{s}[B]  \leq 15 \HNorm{g}{0}{1}[B] \leq 15 \CjNorm{g}{1}[B] \leq 15 \ZygNorm{g}{m+s}[B].
\end{equation*}
Thus, using \cref{Eqn::FuncSpaceRev::EstC1}, we have
\begin{equation*}
\ZygNorm{f_\gamma}{s}[B^n(5)]\leq 15 \CjNorm{f_\gamma}{1}[B^n(5)] \leq 90 \gamma \CjNorm{f}{1}[B^n(\eta_1)] \leq 90\gamma \ZygNorm{f}{m+s}[B^n(\eta_1)].
\end{equation*}
Combining this with \cref{Eqn::FuncSpaceRev::EstHigherDerivs} yields the result.
\end{proof}

\begin{rmk}\label{Rmk::FuncSpaceRev::Convention}
For the next two results, we use the convention that for $s\in (-1,0]$ we set $\ZygSpace{s}=\HSpace{0}{(s+1)/2}$ and for $m<0$ we set $\HSpace{m}{s}=\CjSpace{0}$, with equality of norms.
\end{rmk}

\begin{prop}\label{Prop::FuncSpaceRev::CompareEucldiean}
Fix $\eta\in (0,1]$, and let $Y_1,\ldots, Y_q$ be vector fields on $B^n(\eta)$.  We suppose $Y_j=\sum_{j=1}^n a_j^k \diff{t_k}$ and $\diff{t_k}=\sum_{j=1}^q b_k^j Y_j$, for $1\leq j\leq q$, $1\leq k\leq n$,
where $a_j^k\in \CjSpace{1}[B^n(\eta)]$ and $b_k^j \in \CSpace{B^n(\eta)}$.
\begin{itemize}
\item Let $m\in \N$, $s\in [0,1]$.  Suppose $a_j^k, b_k^j\in \HSpace{m-1}{s}[B^n(\eta)]$, $\forall j,k$.  Then, $\HSpace{m}{s}[B^n(\eta)]= \HXSpace{Y}{m}{s}[B^n(\eta)]$, and
\begin{equation*}
\HNorm{f}{m}{s}[B^n(\eta)]\approx \HXNorm{f}{Y}{m}{s}[B^n(\eta)],
\end{equation*}
where the implicit constants can be chosen to depend only on upper bounds for $q$, $m$, and $\HNorm{a_j^k}{m-1}{s}[B^n(\eta)]$, $\HNorm{b_k^j}{m-1}{s}[B^n(\eta)]$, $\forall j,k$.
\item Let $s>0$.  Suppose $a_j^k, b_k^j\in \ZygSpace{s-1}[B^n(\eta)]$, $\forall j,k$.  Then, $\ZygSpace{s}[B^n(\eta)]=\ZygXSpace{Y}{s}[B^n(\eta)]$, and
\begin{equation*}
\ZygNorm{f}{s}[B^n(\eta)]\approx \ZygXNorm{f}{Y}{s}[B^n(\eta)],
\end{equation*}
where the implicit constants can be chosen to depend only on $s$ and upper bounds for $q$, $\eta^{-1}$, and $\ZygNorm{a_{j}^k}{s-1}[B^n(\eta)]$, $\ZygNorm{b_k^j}{s-1}[B^n(\eta)]$, $\forall j,k$.
\end{itemize}
\end{prop}
\begin{proof}This is \cite[\SSCompareEuclidNorms]{StovallStreet}.\end{proof}

\begin{cor}\label{Cor::FuncSpaceRev::CompareEuclidean}
Let $0<\eta_1<\eta_2$.  Let $Y_1,\ldots, Y_q$ be $C^1$ vector fields on $B^n(\eta_2)$ which span then tangent space to $B^n(\eta_2)$ at every point.
\begin{enumerate}[(i)]
\item\label{Item::FuncSpaceRev::Cor::EquivSpaces::Holder} For $m\in \N$, $s\in [0,1]$, if $Y_1,\ldots, Y_q\in \HSpace{m-1}{s}[B^n(\eta_2)][\R^n]$, then $\HSpace{m}{s}[B^n(\eta_1)]=\HXSpace{Y}{m}{s}[B^n(\eta_1)]$.
\item\label{Item::FuncSpaceRev::Cor::EquivSpaces::Zyg} For $s>0$, if $Y_1,\ldots, Y_q\in \ZygSpace{s-1}[B^n(\eta_2)][\R^n]$, then $\ZygSpace{s}[B^n(\eta_1)]=\ZygXSpace{Y}{s}[B^n(\eta_1)]$.
\end{enumerate}
\end{cor}
\begin{proof}
We describe the proof for \cref{Item::FuncSpaceRev::Cor::EquivSpaces::Holder}; the proof for \cref{Item::FuncSpaceRev::Cor::EquivSpaces::Zyg} is similar.
Since $Y_1,\ldots, Y_q\in \HSpace{m-1}{s}[B^n(\eta_2)][\R^n]$, we have (by definition), $Y_j=\sum_{j=1}^n a_j^k \diff{t_k}$ with $a_j^k\in \HSpace{m-1}{s}[B^n(\eta_2)]$.
Moreover, since $Y_1,\ldots, Y_q$ span the tangent space at every point of $B^n(\eta_2)$, we may write
$\diff{t_k}=\sum_{j=1}^q b_k^j Y_j$, where $b_k^j$ is locally in $\HSpace{m-1}{s}$.  Since $B^n(\eta_1)$ is a relatively compact subset of $B^n(\eta_2)$,
we see $a_j^k, b_k^j\in \HSpace{m-1}{s}[B^n(\eta_1)]$.  From here, \cref{Prop::FuncSpaceRev::CompareEucldiean} yields \cref{Item::FuncSpaceRev::Cor::EquivSpaces::Holder},
completing the proof.
\end{proof}

	\subsection{Manifolds with Zygmund regularity}\label{Section::FuncSpace::ZygManifolds}
In this paper we use $\ZygSpace{s}$ manifolds; the definition is exactly what one would expect, though a little care is needed due to the subtleties of Zygmund spaces. For example,
one must define the Zygmund maps in the right way to ensure that the composition of two Zygmund maps is again a Zygmund map.
For completeness, we present the relevant (standard) definitions here.

\begin{defn}
Let $U_1\subseteq \R^{n_1}$ and $U_2\subseteq \R^{n_2}$ be open sets.  For $s\in (0,\infty]$, we say $f:U_1\rightarrow U_2$ is a $\ZygSpaceloc{s}$ map
if $f\in \ZygSpaceloc{s}[U_1][\R^{n_2}]$.
\end{defn}

\begin{lemma}\label{Lemma::FuncSpaceZyg::Compose::Euclid}
Let $U_1\subseteq \R^{n_1}$, $U_2\subseteq \R^{n_2}$, and $U_3\subseteq \R^{n_3}$ be open sets.  For $s_1\in (0,\infty]$, $s_2\geq s_1$, $s_2\in (1,\infty]$,
if $f_1:U_1\rightarrow U_2$ is a $\ZygSpaceloc{s_1}$ map and $f_2:U_2\rightarrow U_3$ is a $\ZygSpaceloc{s_2}$ map, then $f_2\circ f_1:U_1\rightarrow U_3$
is a $\ZygSpaceloc{s_1}$ map.
\end{lemma}
\begin{proof}
For $s_1=\infty$, the result is obvious.  For $s_1\in (0,\infty)$, because the notion of being a $\ZygSpaceloc{s}$ map is local, is suffices
to check $f_1\circ f_2$ is in $\ZygSpace{s_1}$ on sufficiently small balls.  This is described in \cref{Lemma::ZygSpace::Comp}.
\end{proof}

\begin{lemma}\label{Lemma::FuncSpaceZyg::Inverse::Euclid}
For $s\in (1,\infty]$ if $f:U_1\rightarrow U_2$ is a $\ZygSpaceloc{s}$ map which is also a $C^1$ diffeomorphism, then $f^{-1}:U_2\rightarrow U_1$
is a $\ZygSpaceloc{s}$ map.
\end{lemma}
\begin{proof}
For $s=\infty$, this is standard.  For $s\in(1,\infty)$ it suffices to check $f^{-1}$ is in $\ZygSpace{s}$ when restricted to sufficiently small balls.
This is described in \cref{Lemma::ZygSpace::Inverse}.
\end{proof}

\begin{defn}
Fix $s\in (1,\infty]$ and let $M$ be a topological space.  We say $\{(\phi_\alpha, V_\alpha) : \alpha\in \sI\}$ (where $\sI$ is some index set)
is a $\ZygSpace{s}$ atlas of dimension $n$ if $\{V_\alpha: \alpha\in \sI\}$ is an open cover for $M$, $\phi_{\alpha}:V_\alpha\rightarrow U_\alpha$
is a homeomorphism where $U_\alpha\subseteq \R^n$ is open, and $\phi_{\beta}\circ \phi_{\alpha}^{-1}:\phi_\alpha (V_\beta\cap V_\alpha)\rightarrow U_\beta$
is a $\ZygSpaceloc{s}$ map.
\end{defn}

\begin{defn}
For $s\in (1,\infty]$, a $\ZygSpace{s}$ manifold of dimension $n$ is a Hausdorff, paracompact
topological space $M$ endowed with a $\ZygSpace{s}$ atlas of dimension $n$.
\end{defn}

\begin{rmk}
In this paper we assume all manifolds are paracompact.
This is used in the proofs of \cref{Thm::QualRes::GlobalQual} and \cref{Cor::Holder::GlobalResult} where a partition of unity is used.  Otherwise, paracompactness is not used in this paper
\end{rmk}

\begin{rmk}
Note that an open set $\Omega\subseteq \R^n$ is naturally a $\ZygSpace{\infty}$ manifold of dimension $n$; where we take the atlas consisting of a single coordinate chart
(namely, the identity map $\Omega\rightarrow \Omega$).  We henceforth give open sets this manifold structure.
\end{rmk}

\begin{rmk}
A $\ZygSpace{s}$ manifold is a $\CjSpace{m}$ manifold for any $m<s$.  In light of \cref{Rmk::FuncSpaceEuclid::HolderAndZygTheSame}, $\ZygSpace{\infty}$ manifolds
and $\CjSpace{\infty}$ manifolds are the same.
\end{rmk}

\begin{defn}
For $s\in (0,\infty]$, let $M$ and $N$ be $\ZygSpace{s+1}$ manifolds with $\ZygSpace{s+1}$ atlases $\{(\phi_\alpha, V_{\alpha})\}$
and $\{(\psi_\beta, W_\beta)\}$, respectively.  We say $f:M\rightarrow N$ is a $\ZygSpaceloc{s+1}$ map if $\psi_\beta\circ f\circ \phi_{\alpha}^{-1}$
is a $\ZygSpaceloc{s+1}$ map, $\forall \alpha,\beta$.
\end{defn}

\begin{lemma}
For $s\in (0,\infty]$, suppose $M_1$, $M_2$, and $M_3$ are $\ZygSpace{s+1}$ manifolds, and $f:M_1\rightarrow M_2$ and $f_2:M_2\rightarrow M_3$
are $\ZygSpaceloc{s+1}$ maps.  Then, $f_2\circ f_1:M_1\rightarrow M_3$ is a $\ZygSpaceloc{s+1}$ map.
\end{lemma}
\begin{proof}
This follows from \cref{Lemma::FuncSpaceZyg::Compose::Euclid}.
\end{proof}

\begin{lemma}
Suppose $s\in (0,\infty]$,  $M_1$ and $M_2$ are $\ZygSpace{s+1}$ manifolds, and $f:M_1\rightarrow M_2$ is a $\ZygSpaceloc{s+1}$ map which is also
a $C^1$ diffeomorphism.  Then, $f^{-1}:M_2\rightarrow M_1$ is a $\ZygSpaceloc{s+1}$ map.
\end{lemma}
\begin{proof}
This follows from \cref{Lemma::FuncSpaceZyg::Inverse::Euclid}.
\end{proof}

\begin{defn}
Suppose $s\in (0,\infty]$, and $M_1$ and $M_2$ are $\ZygSpace{s+1}$ manifolds.  We say $f:M_1\rightarrow M_2$ is a $\ZygSpace{s+1}$ diffemorphism if
$f:M_1\rightarrow M_2$ is a bijection and $f:M_1\rightarrow M_2$ and $f^{-1}:M_2\rightarrow M_1$ are $\ZygSpaceloc{s+1}$ maps.
\end{defn}

\begin{rmk}
For $s\in (0,\infty]$, $\ZygSpace{s+1}$ manifolds form a category, where the morphisms are given by $\ZygSpaceloc{s+1}$ maps.  The isomorphisms in this category
are exactly the $\ZygSpace{s+1}$ diffeomorphisms.
\end{rmk}

For $s\in (0,\infty]$, a $\ZygSpace{s+1}$ manifold is a $C^1$ manifold, and it therefore makes sense to talk about vector fields on such a manifold.

\begin{defn}
For $s\in (0,\infty]$ let $M$ be a $\ZygSpace{s+1}$ manifold of dimension $n$ with $\ZygSpace{s+1}$ atlas $\{(\phi_\alpha, V_\alpha)\}$;
here $\phi_\alpha:V_\alpha\rightarrow U_\alpha$ is a $\ZygSpace{s+1}$ diffeomorphism and $U_\alpha\subseteq \R^n$ is open.
We say a $\CjSpace{0}$ vector field $X$ on $M$ is a $\ZygSpace{s}$ vector field if $(\phi_\alpha)_{*} X\in \ZygSpaceloc{s}[U_\alpha][\R^n]$, $\forall \alpha$.
\end{defn}

	
\section{Proofs}
We turn to the proofs of the main results in this paper; as in the statement of \cref{Thm::QuantRes::MainThm}, we fix some $s_0>1$ throughout.
The most difficult part is constructing the map $\Phi$ from \cref{Thm::QuantRes::MainThm}.
We will construct $\Phi$ by seeing it as a composition of two maps $\Phi=\Phi_0\circ \Phi_2$,
where $\Phi_0$ is the map from \cref{Prop::PartI::MainResult} and $\Phi_2$ is described in
\cref{Section::Proofs::MainProp}.  $\Phi_2$ itself will be constructed as a composition of two maps
$\Phi_2= \Psi_\gamma \circ \Phi_1$, which will be described in \cref{Section::Proofs::PfOfMainProp}.

In the some of the sections below, we introduce new notions of $\Zygad{s}$-admissible constants.
We will be explicit in each section which notion we are using.
These notions will be defined in such a way that the compositions described above give the proper result.
For example, we prove \cref{Thm::QuantRes::MainThm} by reducing it to \cref{Prop::Phi2::MainProp}, below.
\Cref{Thm::QuantRes::MainThm} and \cref{Prop::Phi2::MainProp} use different notions of $\Zygad{s}$-admissible constants.
However, in the application of \cref{Prop::Phi2::MainProp} to prove \cref{Thm::QuantRes::MainThm}, constants which are
$\Zygad{s}$-admissible in the sense of \cref{Thm::QuantRes::MainThm}, will be $\Zygad{s}$-admissible in the
sense of \cref{Prop::Phi2::MainProp}.
A similar situation occurs when we reduce \cref{Prop::Phi2::MainProp} to \cref{Prop::Proofs::Phi1}.
Thus, the various notions of $\Zygad{s}$-admissible constants will seamlessly glue
together to yield the main results of this paper.
In each setting, once we have defined $\Zygad{s}$-admissible constants, we use the notation
$A\lesssim_{\Zygad{s}} B$ to mean $A\leq CB$ where $C$ is a positive $\Zygad{s}$-admissible constant.
And we write $A\approx_{\Zygad{s}} B$ for $A\lesssim_{\Zygad{s}} B$ and $B\lesssim_{\Zygad{s}} A$.

In \cref{Section::Proofs::MainProp} we describe the map $\Phi_2$.  In \cref{Section::Proofs::MainTheorem} we
show how \cref{Thm::QuantRes::MainThm} follows by setting $\Phi=\Phi_0\circ \Phi_2$.
In \cref{Section::Proofs::Densities} we prove the results on densities, namely \cref{Thm::Densities::MainThm,Cor::Densities::MainCor}.
In \cref{Section::Proofs::ARegularityResult} we state and prove a result on how to recognize the regularity of vector fields
by considering their commutators.
In \cref{Section::Proofs::Phi1} we describe and construct the map $\Phi_1$.
In \cref{Section::Proofs::PfOfMainProp} we construct the map $\Phi_2$.
Finally, in \cref{Section::Proofs::Qual} we prove
the qualitative results; namely \cref{Thm::QualRes::LocalQual,Thm::QualRes::GlobalQual}.
As mentioned in the introduction, the proofs which follow take many ideas from the work of Malgrange \cite{MalgrangeSurLIntegbrabilite}.

The main idea is the following.  In \cref{Prop::PartI::MainResult} we only have $\ZygNorm{Y_j}{s}[B^n(\eta_1)][\R^n]\lesssim_{\Zygad{s}} 1$,
but we wish to have $\ZygNorm{Y_j}{s+1}[B^n(\eta_1)][\R^n]\lesssim_{\Zygad{s}} 1$.
However, \cref{Prop::PartI::MainResult} gives us additional information:  namely, \cref{Item::PartI::Commutators},
where we have $[Y_j,Y_k]=\sum_{l=1}^n \ct_{j,k}^l Y_l$, $1\leq j,k\leq n$, with $\ZygNorm{\ct_{j,k}^l}{s}[B^n(\eta_1)]\lesssim_{\Zygad{s}} 1$.
Notice, if all we knew was $\ZygNorm{Y_j}{s}[B^n(\eta_1)][\R^n]\lesssim_{\Zygad{s}} 1$ then the best we could say in
general is that $\ZygNorm{\ct_{j,k}^l}{s-1}[B^n(\eta_1)]\lesssim_{\Zygad{s}} 1$; thus
\cref{Item::PartI::Commutators} gives us additional regularity information on $Y_1,\ldots, Y_n$.
This is not enough to conclude that $\ZygNorm{Y_j}{s+1}[B^n(\eta_1)][\R^n]\lesssim_{\Zygad{s}} 1$; indeed
it is easy to find two non-smooth vector fields on $\R^2$, $Z_1$, $Z_2$, which span the tangent space at every point, such that $[Z_1,Z_2]=0$ (take $Z_j =\Psi^{*} \diff{x_j}$ where $\Psi:\R^2\rightarrow \R^2$ is a $C^2$ diffeomorphism).
However, as we will describe in \cref{Section::Proofs::MainProp}, this is enough to conclude that there is a different coordinate system (denoted by $\Phi_2$)
in which we have $\ZygNorm{\Phi_2^{*}Y_j}{s+1}[B^n(1)][\R^n]\lesssim_{\Zygad{s}} 1$,which will complete the proof. 

	\subsection{\texorpdfstring{$\Phi_2$}{The Main Technical Proposition}}\label{Section::Proofs::MainProp}
Fix $\eta_1>0$ and suppose we are given vector fields $Y_1,\ldots, Y_n$ on $B^n(\eta_1)$ of
the form
\begin{equation*}
Y = \diff{t} + A \diff{t} = (I+A) \grad, \quad A(0)=0.
\end{equation*}
Here, we are writing $Y$ for the column vector of vector fields $Y=[Y_1,\ldots, Y_n]^{\transpose}$,
$\diff{t}$ is the column vector $\diff{t}=[\diff{t_1},\ldots, \diff{t_n}]^{\transpose}$ (which we also write as $\grad$), and $A$ is an $n\times n$
matrix depending on $t\in B^n(\eta_1)$.
Fix $s_0>1$ and suppose $A\in \ZygSpace{s_0}[B^n(\eta_1)][\M^{n\times n}]$ and
\begin{equation*}
[Y_j,Y_k]=\sum_{l=1}^n \ct_{j,k}^l Y_l,
\end{equation*}
where $\ct_{j,k}^l\in \ZygSpace{s_0}[B^n(\eta_1)]$.

\begin{defn}\label{Defn::Phi2::Admissible}
For $s\geq s_0$, if we say $C$ is a $\Zygad{s}$-admissible constant, it means $A\in \ZygSpace{s}[B^n(\eta_1)][\M^{n\times n}]$
and $\ct_{j,k}^l\in \ZygSpace{s}[B^n(\eta_1)]$, $1\leq j,k,l\leq n$.  $C$ can be chosen to depend only on $s_0$, $s$ and
upper bounds for $n$, $\eta_1^{-1}$, $\ZygNorm{A}{s}[B^n(\eta_1)][\M^{n\times n}]$, and $\ZygNorm{\ct_{j,k}^l}{s}[B^n(\eta_1)]$.
For $s<s_0$, we define $\Zygad{s}$-admissible constants to be $\Zygad{s_0}$-admissible constants.
\end{defn}

\begin{rmk}
In the definition of $\Zygad{s}$-admissible constants, the vector fields $Y_j$ and the functions $\ct_{j,k}^l$ are assumed to have the same regularity.
Usually, one would expect the functions $\ct_{j,k}^l$ to be one derivative worse than the vector fields $Y_j$.  What the following proposition shows
is that one can pick a different coordinate system in which the vector fields $Y_j$ have one more derivative of regularity, thereby achieving this expectation.
\end{rmk}

\begin{prop}\label{Prop::Phi2::MainProp}
There exists an $\Zygad{s_0}$-admissible constant $K\geq 1$ and a map $\Phi_2:B^n(1)\rightarrow B^n(\eta_1)$ such that
\begin{enumerate}[label=(\alph*),series=phitwotheoremenumeration]
\item\label{Item::Phi2::Zygmund} $\Phi_2\in \ZygSpace{s_0+1}[B^n(1)][\R^n]$, and
\begin{equation*}
\ZygNorm{\Phi_2}{s+1}[B^n(1)][\R^n]\lesssim_{\Zygad{s}} 1, \quad \forall s>0.
\end{equation*}
\item\label{Item::Phi2::At0} $\Phi_2(0)=0$, $d\Phi_2(0)=K^{-1}I$.
\item\label{Item::Phi2::Diffeo} $\Phi_2(B^n(1))\subseteq B^n(\eta_1)$ is open and $\Phi_2:B^n(1)\rightarrow \Phi_2(B^n(1))$ is a $\ZygSpace{s_0+1}$ diffeomorphism.
\end{enumerate}
Let $\Yh_j=\Phi_2^{*} Y_j$.  Then,
\begin{enumerate}[resume*=phitwotheoremenumeration]
\item\label{Item::Phi2::YMatrix} $\Yh = K(I+\Ah) \grad$, and $\Ah(0)=0$.
\item\label{Item::Phi2::BoundA} $\sup_{u\in B^n(1)} \Norm{\Ah(u)}[\M^{n\times n}]\leq \frac{1}{2}$.
\item\label{Item::Phi2::ZygNormYs} $\ZygNorm{\Yh_j}{s+1}[B^n(1)][\R^n]\lesssim_{\Zygad{s}} 1$, for $s>0$, $1\leq j\leq n$.
\end{enumerate}
\end{prop}

We defer the proof of \cref{Prop::Phi2::MainProp} to \cref{Section::Proofs::PfOfMainProp}.
	
	\subsection{Proof of \texorpdfstring{\cref{Thm::QuantRes::MainThm}}{the Quantitative Theorem}}\label{Section::Proofs::MainTheorem}
In this section, we prove \cref{Thm::QuantRes::MainThm} by combining \cref{Prop::PartI::MainResult,Prop::Phi2::MainProp}.
We take the same setting as in \cref{Thm::QuantRes::MainThm}, and define $0$-admissible and $\Zygad{s}$-admissible constants as in \cref{Defn::Results::0admissible,Defn::Results::sadmissible}.
Take $\Phi_0$, $Y_1,\ldots, Y_q$, $A$, $\eta_1$, and $\chi$ be as in \cref{Prop::PartI::MainResult}, so that
$\Phi_0:B^n(\eta_1)\rightarrow B_{X_{J_0}}(x_0,\chi)$.
Note that \cref{Eqn::PartI::RegularityOfYs} implies $\ZygNorm{A}{s}[B^n(\eta_1)][\M^{n\times n}]\lesssim_{\Zygad{s}} 1$.
Hence, using \cref{Prop::PartI::MainResult} \cref{Item::PartI::ABound,Item::PartI::Yreg,Item::PartI::Commutators}, we see
that \cref{Prop::Phi2::MainProp} applies to $Y_1,\ldots, Y_n$ (with this choice of $\eta_1$), and every constant which is $\Zygad{s}$-admissible
in the sense of \cref{Prop::Phi2::MainProp} is $\Zygad{s}$-admissible in the sense of this section.
Thus we obtain a map $\Phi_2:B^n(1)\rightarrow B^n(\eta_1)$ as in \cref{Prop::Phi2::MainProp}.  Let $K$, $\Ah$, and $\Yh_1,\ldots, \Yh_n$ be as in that proposition.
Notationally, we prove \cref{Thm::QuantRes::MainThm} with $\Yh$ in place of $Y$ and $\Ah$ in place of $A$.

With $\chi\in (0,\xi]$ as in \cref{Prop::PartI::MainResult}, 
\cref{Thm::QuantRes::MainThm} \cref{Item::QuantRes::NonzeroJ0}, \cref{Item::QuantRest::J0Big}, and \cref{Item::QuantRes::Submanifold}
follow immediately from \cref{Prop::PartI::MainResult} \cref{Item::PartI::NonzeroJ0}, \cref{Item::PartI::J0Big}, and \cref{Item::PartI::Submanifold}.
Set $\Phi=\Phi_0\circ \Phi_2:B^n(1)\rightarrow B_{X_{J_0}}(x_0,\chi)$.

By \cref{Prop::PartI::MainResult} \cref{Item::PartI::OpenImage,Item::PartI::PhiDiffeo}, $\Phi_0$ takes open subsets of $B^n(\eta_1)$ to open
subsets of $B_{X_{J_0}}(x_0,\chi)$.   By \cref{Prop::Phi2::MainProp} \cref{Item::Phi2::Diffeo}, $\Phi_2(B^n(1))$ is open in $B^n(\eta_1)$.
\Cref{Thm::QuantRes::MainThm} \cref{Item::QuantRes::ImageOpen} follows.
\Cref{Thm::QuantRes::MainThm} \cref{Item::QuantRes::Diffeo} follows by combining 
\cref{Prop::PartI::MainResult} \cref{Item::PartI::PhiDiffeo} and
\cref{Prop::Phi2::MainProp} \cref{Item::Phi2::Diffeo}.

By the definition of $\Phi_0$, \cref{Eqn::PartI::DefnPhi0}, we have $\Phi_0(0)=x_0$.  By \cref{Prop::Phi2::MainProp} \cref{Item::Phi2::At0},
we have $\Phi_2(0)=0$.  Hence, $\Phi(0)=x_0$, proving \cref{Thm::QuantRes::MainThm} \cref{Item::QuantRes::At0}.
The existence of $\xi_1$ as in \cref{Thm::QuantRes::MainThm} \cref{Item::QuantRes::Xis} follows just as in \cite[\SSExistXiOne]{StovallStreet},
while the existence of $\xi_2$ follows from \cite[\SSExistXiTwo]{StovallStreet}.

For $1\leq j\leq n$, we have $\Phi^{*} X_j = \Phi_2^{*} \Phi_0^{*} X_j = \Phi_2^{*} Y_j = \Yh_j$.
For $n+1\leq j\leq q$, we define $\Yh_j:=\Phi^{*} X_j$.
\Cref{Prop::Phi2::MainProp} \cref{Item::Phi2::YMatrix,Item::Phi2::BoundA} 
shows $\Yh_{J_0} = K(I+\Ah)\grad$ and proves \cref{Thm::QuantRes::MainThm} \cref{Item::QuantRes::YMatrix,Item::QuantRes::BoundA}.

\Cref{Prop::Phi2::MainProp} \cref{Item::Phi2::ZygNormYs} proves \cref{Thm::QuantRes::MainThm} \cref{Item::QuantRes::YZyg}
for $1\leq j\leq n$.  For $n+1\leq j\leq q$, we proceed as follows.  Let $b_j^l$ be as in \cref{Prop::PartI::MainResult} \cref{Item::PartI::Bs}.
Then, we have
\begin{equation}\label{Eqn::PfMain::TmpYhAsSpan}
\Yh_j = \Phi_2^{*} Y_j = \sum_{k=1}^n \Phi_2^* \left(b_j^k Y_k\right) = \sum_{k=1}^n (b_j^k\circ \Phi_2) \Yh_k.
\end{equation}
We have already shown 
\begin{equation}\label{Eqn::PfMain::BoundForYh}
\ZygNorm{\Yh_k}{s+1}[B^n(1)][\R^n]\lesssim_{\Zygad{s}} 1, \quad 1\leq k\leq n.
\end{equation}
Since $\ZygNorm{b_j^k}{s+1}[B^n(\eta_1)]\lesssim_{\Zygad{s}} 1$ by \cref{Prop::PartI::MainResult} \cref{Item::PartI::Bs}
and $\ZygNorm{\Phi_2}{s+1}[B^n(1)][\R^n] \lesssim_{\Zygad{s}} 1$ by \cref{Prop::Phi2::MainProp} \cref{Item::Phi2::Zygmund},
we have $\ZygNorm{b_j^k\circ \Phi_2}{s+1}[B^n(1)]\lesssim_{\Zygad{s}} 1$ for $s>0$ (see \cref{Lemma::ZygSpace::Comp}).
Combining this with \cref{Eqn::PfMain::TmpYhAsSpan} and \cref{Eqn::PfMain::BoundForYh} completes the proof of  \cref{Thm::QuantRes::MainThm} \cref{Item::QuantRes::YZyg}.

Notice that \cref{Thm::QuantRes::MainThm} \cref{Item::QuantRes::YZyg} (which we have already shown) implies
$\ZygNorm{\Ah}{s+1}[B^n(1)][\M^{n\times n}]\lesssim_{\Zygad{s}} 1$.
We have $\Yh_{J_0}=K(I+\Ah) \grad$.  Since $\Norm{\Ah(u)}[\M^{n\times n}]\leq \frac{1}{2}$, $\forall u\in B^n(1)$, $(I+\Ah(u))$ is invertible for all $u\in B^n(1)$
and we have $\ZygNorm{(I+\Ah)^{-1}}{s+1}[B^n(1)]\lesssim_{\Zygad{s}} 1$ (this uses \cref{Lemma::ZygSpace::Algebra} and the cofactor representation of $(I+\Ah)^{-1}$).
Hence, $\grad = K^{-1}(I+\Ah)^{-1} \Yh_{J_0}$.  I.e., for each $1\leq j\leq n$, $\diff{t_j}$ can be written as a linear combination, with coefficients
in $\ZygSpace{s+1}[B^n(1)]$ of $\Yh_1,\ldots, \Yh_n$, and the $\ZygSpace{s+1}$ norms of the coefficients are $\lesssim_{\Zygad{s}} 1$.  Combining this with \cref{Thm::QuantRes::MainThm} \cref{Item::QuantRes::YZyg},
\cref{Prop::FuncSpaceRev::CompareEucldiean}
applies to prove \cref{Thm::QuantRes::MainThm} \cref{Item::QuantRes::EquivNorms}.

For  \cref{Thm::QuantRes::MainThm} \cref{Item::QuantRes::BoundByInvNorm}, we already know by  \cref{Thm::QuantRes::MainThm} \cref{Item::QuantRes::EquivNorms}
that $\ZygNorm{f\circ \Phi}{s}[B^n(1)] \approx_{\Zygad{s-2}} \ZygXNorm{f\circ \Phi}{\Yh_{J_0}}{s}[B^n(1)]$.
That $\ZygXNorm{f\circ\Phi}{\Yh_{J_0}}{s}[B^n(1)]\leq \ZygXNorm{f}{X_{J_0}}{s}[B_{X_{J_0}}(x_0,\chi)]$ follows from
\cite[\SSBiggerNormMap]{StovallStreet}; \cref{Thm::QuantRes::MainThm} \cref{Item::QuantRes::BoundByInvNorm} follows.
	
	\subsection{Densities}\label{Section::Proofs::Densities}
In this section, we prove \cref{Thm::Densities::MainThm,Cor::Densities::MainCor}.  We take the setting of \cref{Thm::Densities::MainThm}
and therefore we have a $C^1$ density $\nu$ and
 a notion of $\Zygad{s;\nu}$-admissible constants, as in \cref{Defn::Densitities::AdmissibleConsts}.  We let $\Phi$, $Y_1,\ldots, Y_q$, $K$, and $A$ be
as in \cref{Thm::QuantRes::MainThm}, and we let $\nu_0$ be as in \cref{Eqn::PartI::Densities::Defn::nu0}.

\begin{lemma}\label{Lemma::PfDense::h0}
Define $h_0$ by $\Phi^{*} \nu_0 = h_0 \LebDensity$.  Then, 
$h_0=\det\left(K(I+A)\right)^{-1}$.
In particular, $h_0(t)\approx_{\Zygad{s_0}} 1$, $\forall t\in B^n(1)$, and
\begin{equation}\label{Eqn::PfDens::ToShowh0Reg}
\ZygNorm{h_0}{s}[B^n(1)]\lesssim_{\Zygad{s-1}} 1, \quad s>0.
\end{equation}
\end{lemma}
\begin{proof}
Because $\sup_{t\in B^n(1)}\Norm{A(t)}[\M^{n\times n}]\leq \frac{1}{2}$ and $K\approx_{\Zygad{s_0}} 1$ by \cref{Thm::QuantRes::MainThm}, we have
$|\det (K(I+A))^{-1}| =\det(K(I+A))^{-1}$, and $\det(K(I+A))^{-1}\approx_{\Zygad{s_0}} 1$.  Using that $\Phi_{*} Y_j =X_j$,
\begin{equation*}
\begin{split}
&h_0(t) = (\Phi^{*}\nu_0)(t)\left(\diff{t_1},\diff{t_2},\ldots, \diff{t_n}\right)
=(\Phi^{*}\nu_0)(t)( (K(I+A(t)))^{-1} Y_1(t),\ldots, (K(I+A(t)))^{-1}Y_n(t)) 
\\&=|\det (K(I+A(t)))^{-1}| (\Phi^{*} \nu_0)(t)(Y_1(t),\ldots, Y_n(t))
=\det(K(I+A(t)))^{-1} \nu_0(\Phi(t))(X_1(\Phi(t)),\ldots, X_n(\Phi(t)))
\\&=\det(K(I+A(t)))^{-1}.
\end{split}
\end{equation*}
This proves $h_0=\det\left(K(I+A)\right)^{-1}$ and therefore $h_0(t)\approx_{\Zygad{s_0}} 1$.
\Cref{Thm::QuantRes::MainThm} \cref{Item::QuantRes::YZyg} implies $\ZygNorm{A}{s}[B^n(1)][\M^{n\times n}]\lesssim_{\Zygad{s-1}} 1$;
\cref{Eqn::PfDens::ToShowh0Reg} follows from this using \cref{Lemma::ZygSpace::Algebra}, 
completing the proof.
\end{proof}

\begin{proof}[Proof of \cref{Thm::Densities::MainThm}]
Let $g$ be as in \cref{Prop::PartI::Densities::MainProp} so that $\nu=g\nu_0$.
Hence, $h\LebDensity = \Phi^{*}\nu = \Phi^{*} g \nu_0 = (g\circ \Phi) h_0 \LebDensity$, where $h_0$ is as in  \cref{Lemma::PfDense::h0}.
Thus, $h=(g\circ \Phi) h_0$.
\Cref{Prop::PartI::Densities::MainProp} \cref{Item::PartI::Densities::gconst} implies
$g\circ \Phi(t)\approx_{\Zygsonu} \nu(X_1,\ldots,X_n)(x_0)$ and \cref{Lemma::PfDense::h0} shows $h_0(t)\approx_{\Zygad{s_0}} 1$.
\Cref{Item::Density::HConst} follows.

\Cref{Thm::QuantRes::MainThm} \cref{Item::QuantRes::BoundByInvNorm} combined with \cref{Prop::PartI::Densities::MainProp} \cref{Item::PartI::Densities::gZyg}
shows $\ZygNorm{g\circ \Phi}{s}[B^n(1)]\lesssim_{\Zygad{s-1;\nu}} 1$.  Combining this with \cref{Eqn::PfDens::ToShowh0Reg} and the formula $h=(g\circ \Phi) h_0$, and using
\cref{Lemma::ZygSpace::Algebra},
proves \cref{Item::Density::HZyg} and completes the proof.
\end{proof}

To prove \cref{Cor::Densities::MainCor}, we introduce a corollary of \cref{Thm::QuantRes::MainThm}.
\begin{cor}\label{Cor::Densitites::Containments}
Let $\Phi$, $\xi_1$, and $\xi_2$ be as in \cref{Thm::QuantRes::MainThm}.  Then, there exist $\Zygad{s_0}$-admissible constants $0<\xi_4\leq \xi_3\leq \xi_2$
and a map $\Phih:B^n(1)\rightarrow B_{X_{J_0}}(x_0,\xi_2)$ which satisfies all the same estimates as $\Phi$ so that
\begin{equation*}
\begin{split}
&B_X(x_0,\xi_4)\subseteq B_{X_{J_0}}(x_0,\xi_3) \subseteq \Phih(B^n(1)) \subseteq B_{X_{J_0}}(x_0,\xi_2)\subseteq B_X(x_0,\xi_2)
\\& \subseteq B_{X_{J_0}}(x_0,\xi_1)\subseteq \Phi(B^n(1))\subseteq B_{X_{J_0}}(x_0,\chi)\subseteq B_{X_{J_0}}(x_0,\xi).
\end{split}
\end{equation*}
\end{cor}
\begin{proof}
After obtaining $\xi_1$, $\xi_2$, and $\Phi$ from \cref{Thm::QuantRes::MainThm}, we apply \cref{Thm::QuantRes::MainThm} again with $\xi$ replaced by $\xi_2$,
to yield the map $\Phih$ and $\Zygad{s_0}$-admissible constants $\xi_3$ and $\xi_4$ as above.
\end{proof}

\begin{proof}[Proof of \cref{Cor::Densities::MainCor}]
Using \cref{Thm::Densities::MainThm} \cref{Item::Density::HConst}, we have
\begin{equation*}
\nu(\Phi(B^n(1))) = \int_{\Phi(B^n(1))} \nu = \int_{B^n(1)} \Phi^{*} \nu = \int_{B^n(1)} h(t) \: dt \approx_{\Zygsonu} \nu(X_1,\ldots, X_n)(x_0),
\end{equation*}
and we have the same estimate for $\Phi$ replaced by $\Phih$, where $\Phih$ is as in \cref{Cor::Densitites::Containments}.
Since 
$$\Phih(B^n(1)) \subseteq B_{X_{J_0}}(x_0,\xi_2)\subseteq B_X(x_0,\xi_2) \subseteq  \Phi(B^n(1)),$$
and since $h(t)$ always has the same sign (by \cref{Thm::Densities::MainThm} \cref{Item::Density::HConst}), 
\cref{Eqn::Densities::EstimateVolumes} follows.

To complete the proof, we need to show
\begin{equation}\label{Eqn::Desnitiy::ToShow::CompareVol}
|\nu(X_1,\ldots, X_n)(x_0)|
\approx_{0} \max_{(j_1,\ldots, j_n)\in \sI(n,q)} |\nu(X_{j_1},\ldots, X_{j_n})(x_0)|.
\end{equation}
However, either both sides of this equation equal $0$, or \cref{Prop::PartI::Densities::MainProp} shows
\begin{equation*}
\frac{|\nu(X_{j_1},\ldots, X_{j_n})(x_0)|}{|\nu(X_1,\ldots, X_n)(x_0)|} = \frac{|\nu_0(X_{j_1},\ldots, X_{j_n})(x_0)|}{|\nu_0(X_1,\ldots, X_n)(x_0)|}=\mleft|\frac{X_{j_1}(x_0)\wedge \cdots \wedge X_{j_n}(x_0)}{X_1(x_0)\wedge \cdots \wedge X_n(x_0)}\mright| \leq \zeta^{-1}\lesssim_0 1,
\end{equation*}
where we have used the definition of $\zeta$ (see \cref{Eqn::QuantRes::PickJ0}).
Since the left hand side of \cref{Eqn::Desnitiy::ToShow::CompareVol} is $\leq$ the right hand side,
this completes the proof.
\end{proof}

	
	\subsection{A Regularity Result}\label{Section::Proofs::ARegularityResult}
Let $Y_1,\ldots, Y_n$ be vector fields on $B^n(2)$.  Using the vector notation from \cref{Section::Proofs::MainProp}, write
\begin{equation*}
Y=\diff{t}+A\diff{t},
\end{equation*}
where $A:B^n(2)\rightarrow \M^{n\times n}$.  Let $a_j^k$ denote the $(j,k)$ component of $A$, and define
$A_j=[a_j^1,\ldots, a_j^n]$; i.e., $A_j$ is the $j$th row of $A$.  We have
\begin{equation*}
Y_j=\diff{t_j} + A_j \diff{t}.
\end{equation*}

Suppose
\begin{equation}\label{Eqn::PfReg::CommutatorAssump}
[Y_j,Y_k]=\sum_{l=1}^n c_{j,k}^l Y_l
\end{equation}
and
\begin{equation}\label{Eqn::PfReg::DivAssump}
\sum_{j=1}^n \diff{t_j} A_j =0.
\end{equation}

\begin{prop}\label{Prop::PfRegularity::MainProp}
In the above setting, there exists $\gamma_1=\gamma_1(n)>0$ (depending only on $n$) such that
the following holds.
If $s>1$ is such that $c_{j,k}^l,a_j^k\in \ZygSpace{s}[B^n(2)]$, $\forall j,k,l$,
and  $\LpNorm{\infty}{a_j^k}[B^n(2)]\leq \gamma_1$, $\forall j,k$,
then $a_j^k\in \ZygSpace{s+1}[B^n(1)]$ and
\begin{equation*}
\max_{j,k} \ZygNorm{a_j^k}{s+1}[B^n(1)]\leq D_{n,s}.
\end{equation*}
where $D_{n,s}$ can be chosen to depend only on $s$, and upper bounds for $n$, $ \ZygNorm{a_j^k}{s}[B^n(2)]$, and $\ZygNorm{c_{j,k}^l}{s}[B^n(2)]$ (for all
$j,k,l$).
\end{prop}
\begin{proof}
Set $C_{j,k}=[c_{j,k}^1,\ldots, c_{j,k}^n]$.  Then \cref{Eqn::PfReg::CommutatorAssump} can be rewritten as
\begin{equation*}
\diff{t_j} A_k - \diff{t_k} A_j + A_j \diff{t} A_k - A_k \diff{t} A_j = C_{j,k}(I+A).
\end{equation*}
Combining this with \cref{Eqn::PfReg::DivAssump} shows that $A$ satisfies the following system of equations:
\begin{equation*}
\sE A + \Gamma(A,\grad A) = \Ch,
\end{equation*}
where
\begin{equation*}
\sE A = \left( \left( \diff{t_j} A_k- \diff{t_k} A_j \right)_{1\leq j<k\leq n}, \sum_{j=1}^n \diff{t_j} A_j \right),
\end{equation*}
$\Gamma$ is a constant coefficient bilinear form, depending only on $n$, and 
$\Ch= ((C_{j,k}(I+A))_{1\leq j<k\leq n}, 0)$.

By \cref{Lemma::Elliptic::Ex}, $\sE$ is elliptic.  Also, $\ZygNorm{\Ch}{s}\leq D_{n,s}$, where $D_{n,s}$ is as in the statement of the proposition (see \cref{Lemma::ZygSpace::Algebra}).
From here, the result follows from \cref{Prop::Appendix::Regularity::MainProp} (taking $s_1=s-1$ and $s_2=s$ in that proposition).
\end{proof}
	
	\subsection{\texorpdfstring{$\Phi_1$}{An Auxiliary Map}}\label{Section::Proofs::Phi1}
Fix $s_0>1$.
Let $Y_1,\ldots, Y_n$ be $\ZygSpace{s_0}$ vector fields on $B^n(5)$.  Using the matrix notation of \cref{Section::Proofs::MainProp},
we assume $Y_1,\ldots, Y_n$ have the form
\begin{equation*}
Y=\diff{t}+A\diff{t},\quad A(0)=0,
\end{equation*}
where $A:B^n(5)\rightarrow \M^{n\times n}$.  We assume
\begin{equation*}
[Y_j,Y_k]=\sum_{l=1}^n c_{j,k}^l Y_l.
\end{equation*}

\begin{defn}
For $s\geq s_0$, 
if we say $C$ is a $\Zygad{s}$-admissible constant it means that
$A\in \ZygSpace{s}[B^n(5)][\M^{n\times n}]$ and $c_{j,k}\in \ZygSpace{s}[B^n(5)]$, $1\leq j,k,l\leq n$.
$C$ can be chosen to depend only on $s$, $s_0$, $n$,
and upper bounds for $\ZygNorm{A}{s}[B^n(5)][\M^{n\times n}]$ and $\ZygNorm{c_{j,k}^l}{s}[B^n(5)]$, $1\leq j,k,l\leq n$.
For $s<s_0$, we define $\Zygad{s}$-admissible constants to be $\Zygad{s_0}$-admissible constants.
\end{defn}

\begin{prop}\label{Prop::Proofs::Phi1}
There exists $\gamma_2=\gamma_2(n,s_0)>0$ ($\gamma_2$ depending only on $n$ and $s_0$)
such that if $\ZygNorm{A}{s_0}[B^n(5)][\M^{n\times n}]\leq \gamma_2$
then there exists $\Phi_1:B^n(1)\rightarrow B^n(5)$ such that:
\begin{enumerate}[label=(\alph*),series=phionetheoremenumeration]
\item\label{Item::Phi1::Zygs0} $\Phi_1\in \ZygSpace{s_0+1}[B^n(1)][\R^n]$ and $\ZygNorm{\Phi_1}{s_0+1}[B^n(1)][\R^n]\leq D_{n,s_0}$, where $D_{n,s_0}$
depends only on $n$ and $s_0$.
\item\label{Item::Phi1::Zygs} $\ZygNorm{\Phi_1}{s+1}[B^n(1)]\lesssim_{\Zygad{s}} 1$, $\forall s>0$.
\item\label{Item::Phi1::At0} $\Phi_1(0)=0$ and $d\Phi_1(0)=I$.
\item\label{Item::Phi1::Open} $\Phi_1(B^n(1))\subseteq B^n(5)$ is open.
\item\label{Item::Phi1::Diffeo} $\Phi_1:B^n(1)\rightarrow \Phi_1(B^n(1))$ is a $\ZygSpace{s_0+1}$ diffeomorphism.
\end{enumerate}
Let $\Yh_j:=\Phi_1^{*} Y_j$, then
$$\Yh = \diff{t}+\Ah\diff{t},$$
where
\begin{enumerate}[resume*=phionetheoremenumeration]
\item\label{Item::Phi1::A0} $\Ah(0)=0$ and $\sup_{u\in B^n(1)}\Norm{\Ah(u)}[\M^{n\times n}]\leq \frac{1}{2}$.
\item\label{Item::Phi1::As} $\ZygNorm{\Ah}{s+1}[B^n(1)][\M^{n\times n}]\lesssim_{\Zygad{s}} 1$, $s>0$.
\item\label{Item::Phi1::Yhs} $\ZygNorm{\Yh_j}{s+1}[B^n(1)][\R^n]\lesssim_{\Zygad{s}} 1$, $s>0$.
\end{enumerate}
\end{prop}

The rest of this section is devoted to the proof of \cref{Prop::Proofs::Phi1}.  

\begin{lemma}\label{Lemma::PfPhi1::AuxLemma}
Fix $\sigma,\gamma_1>0$.  There exists $\gamma_2=\gamma_2(n,s_0,\sigma,\gamma_1)>0$ ($\gamma_2$ depending only on $n$, $s_0$, $\sigma$, and $\gamma_1$)
such that if $\ZygNorm{A}{s_0}[B^n(5)][\M^{n\times n}]\leq \gamma_2$ then there exists
 $H\in \ZygSpace{s_0+1}[B^n(4)][\R^n]$ of the form
$H(t) = t+R(t)$ where
\begin{enumerate}[label=(\alph*),series=phionelemmaenumeration]
\item\label{Item::PfPhi1::HIsDiffeo} $H(B^n(4))\subseteq \R^n$ is open and $H:B^n(4)\rightarrow H(B^n(4))$ is a $\ZygSpace{s_0+1}$ diffeomorphism.
\item $R(0)=0$ and $dR(0)=0$
\item\label{Item::PfPhi1::BoundRs0} $R\in \ZygSpace{s_0+1}[B^n(4)][\R^n]$ with $\ZygNorm{R}{s_0+1}[B^n(4)][\R^n]\leq \sigma$.
\item\label{Item::PfPhi1::BoundRs} $\ZygNorm{R}{s+1}[B^n(3)][\R^n]\lesssim_{\Zygad{s}} 1$ for all $s>0$.
\end{enumerate}
Moreover, let $\Yh_j= H_{*} Y_j$.  Then $\Yh = \diff{v}+ \Ah \diff{v}$ and
\begin{enumerate}[resume*=phionelemmaenumeration]
\item\label{Item::PfPhi1::DivAhs} If $\Ah_j$ is the $j$th row of $\Ah$, then $\sum_{j=1}^n \diff{v_j} \Ah_j(v) =0$ for $v\in H(B^n(4))$.
\item\label{Item::PfPhi1::BoundLinfas} If $\ah_j^k$ is the $(j,k)$ component of $\Ah$, then $\LpNorm{\infty}{\ah_{j}^k}[H(B^n(4))]\leq \gamma_1$.
\end{enumerate} 
\end{lemma}
\begin{proof}
If $\sigma>0$ is sufficiently small, depending only on $s_0$ and $n$, and if \cref{Item::PfPhi1::BoundRs0} holds,
the Inverse Function Theorem implies \cref{Item::PfPhi1::HIsDiffeo}.
Thus, without loss of generality, we shrink $\sigma>0$ so that \cref{Item::PfPhi1::HIsDiffeo} holds.
\Cref{Item::PfPhi1::BoundRs} for $s<s_0$ follows from the result for $s=s_0$ (by the definition of $\Zygad{s}$-admissible constants).
Thus it suffices to prove \cref{Item::PfPhi1::BoundRs} for $s\geq s_0$.

To begin, let $R\in \ZygSpace{s_0+1}[B^n(4)][\R^n]$ be any function
satisfying $R(0)=0$, $dR(0)=0$, and $\ZygNorm{R}{s_0+1}[B^n(4)][\R^n]\leq \sigma$ (we will later specialize to a specific choice of $R$).
To emphasize the dependance of $H$ on $R$, we write $H_R$ in place of $H$, so that $H_R(t)=t+R(t)$.
Using the standard notation if $R=(R_1,\ldots, R_n)$, we have
\begin{equation*}
dR(t) =
\begin{bmatrix}
\frac{\partial R_1}{\partial t_1}(t)& \cdots &\frac{\partial R_1}{\partial t_n}(t)\\
\vdots & \ddots & \vdots\\
\frac{\partial R_n}{\partial t_1}(t)& \cdots &\frac{\partial R_n}{\partial t_n}(t)
\end{bmatrix}.
\end{equation*}
Setting $\Yh_j:=(H_R)_{*} Y_j$, a direct computation shows
\begin{equation*}
\Yh = \diff{v} + \Ah(v) \diff{v},
\end{equation*}
where 
\begin{equation}\label{Eqn::PfPhi1::FormulaAh}
\Ah(v)=\left(dR(t)^{\transpose} + A(t) (I+dR(t)^{\transpose})\right)\bigg|_{t=H_R^{-t}(v)}, \quad v\in H_R(B^n(4)).
\end{equation}
Without loss of generality, we take $\sigma\leq \frac{\gamma_1}{2}$, and by taking $\gamma_2>0$ sufficiently small
\cref{Eqn::PfPhi1::FormulaAh} implies \cref{Item::PfPhi1::BoundLinfas}.

We wish to pick $R$ so that
\begin{equation}\label{Eqn::PfPhi1::ToShowDerivAh}
\sum_{j=1}^n \diff{v_j}\Ah_j(v) =0, \quad v\in H_R(B^n(4)).
\end{equation}
Define $\Psi(A,R)(t):=(\Psi_1(A,R)(t),\ldots, \Psi_n(A,R)(t))$ by
\begin{equation*}
\Psi_k(A,R)(t):=\sum_{j=1}^n \diff{v_j} \left( dR(H_R^{-1}(v))^{\transpose} + A(H_R^{-1}(v))( I+dR(H_R^{-1}(v))^{\transpose}) )  \right)_{j,k}\bigg|_{v=H_R(t)};
\end{equation*}
where the subscript $j,k$ denotes taking the $(j,k)$ component of the matrix.
In light of \cref{Eqn::PfPhi1::FormulaAh}, \cref{Eqn::PfPhi1::ToShowDerivAh} is equivalent to $\Psi(A,R)(t)=0$, $t\in B^n(4)$.

For any function $K(t)$, the chain rule shows
\begin{equation}\label{Eqn::PfPhi1::DerivInvvj}
\diff{v_j} K(H_R^{-1}(v))\bigg|_{v=H_R(t)} = dK(t) (I+dR(t))^{-1}e_j,
\end{equation}
where $e_j$ denotes the $j$th standard basis element--the point is that the right hand side of \cref{Eqn::PfPhi1::DerivInvvj} is a function of $dK(t)$ and $dR(t)$.  Thus, using the notation of \cref{Section::ExistNonlinear}, we have
\begin{equation*}
\Psi(A,R)(t)= g(\Deriv^1 A(t), \Deriv^2 R(t))
\end{equation*}
for some smooth function $g$ defined near the origin, with $g(0,0)=0$.  Furthermore, it is easy to see that $g(\Deriv^1 A(t), \Deriv^2 R(t))$ is quasilinear in $R$
in the sense of \cref{Eqn::Elliptic::Exist::Quasilinear}.

We wish to solve for $R$ in terms of $A$ so that $\Psi(A,R)=0$, provided $\ZygNorm{A}{s_0}[B^n(5)][\M^{n\times n}]\leq \gamma_2$, where
$\gamma_2$ is a chosen small as in the statement of the lemma.  To do this, we apply \cref{Prop::Elliptic::Exist::MainProp}; thus
we need to make sure $g(\Deriv^1 A(t), \Deriv^2 R(t))$ is elliptic in the sense of that proposition (where we are replacing $B$ with $R$
in the statement of that proposition).  Define $\sE_2$ as in \cref{Eqn::Elliptic::Exist::DefnsE2}; we wish to show $\sE_2$ is elliptic.
Note that
\begin{equation*}
R\mapsto \frac{d}{d\epsilon} \bigg|_{\epsilon=0} \Psi(0,\epsilon R)
\end{equation*}
is a second order, constant coefficient, differential operator acting on $R$ whose principal symbol is $\sE_2$.  Thus, we
wish to show that this differential operator is elliptic.  It suffices to compute this operator in the special case when $R\in C^\infty$. 

Assuming $R$ is $C^\infty$, we have
\begin{equation*}
H_{\epsilon R}(t) = t+\epsilon R(t), \quad H_{\epsilon R}^{-1}(v) = v-\epsilon R(v) +O(\epsilon^2),
\end{equation*}
where $O(\epsilon^{2})$ denotes a term which is $C^\infty$ in the variable $t$ or $v$, and all of whose derivatives in this variable (of all orders $\geq 0$) are $O(\epsilon^2)$.
Thus,
\begin{equation*}
\begin{split}
&\Psi_k(0,\epsilon R) = \sum_{j=1}^n \diff{v_j} \left(\epsilon dR(H_{\epsilon R}^{-1}(v))^{\transpose}\right)_{j,k}\bigg|_{v=H_{\epsilon R}(t)}
=\sum_{j=1}^n \diff{v_j} \left(\epsilon dR(v)^{\transpose}\right)_{j,k}\bigg|_{v=t+\epsilon R(t)} + O(\epsilon^{2})
\\&=\sum_{j=1}^n \epsilon \diff{v_j} \frac{\partial R_k}{\partial v_j} (v)\bigg|_{v=t+\epsilon R(t)} + O(\epsilon^2)
=\sum_{j=1}^n \epsilon \frac{\partial^2}{\partial t_j^2} R_k(t) + O(\epsilon^2).
\end{split}
\end{equation*}
Thus,
\begin{equation*}
\frac{d}{d\epsilon}\bigg|_{\epsilon=0} \Psi(0,\epsilon R) = \left( \sum_{j=1}^n \frac{\partial^2}{\partial t_j^2} R_1,\sum_{j=1}^n \frac{\partial^2}{\partial t_j^2} R_2,\ldots, \sum_{j=1}^n \frac{\partial^2}{\partial t_j^2} R_n\right),
\end{equation*}
and we conclude $g(\Deriv^1 A(t), \Deriv^2 R(t))$ is elliptic in the sense of \cref{Prop::Elliptic::Exist::MainProp}.

We apply \cref{Prop::Elliptic::Exist::MainProp} with $D=4$, $\eta=3$, and 
$$N=\{R\in \ZygSpace{s_0+1}[B^n(4)][\R^n] : \ZygNorm{R}{s_0+1}[B^n(4)][\R^n]<\sigma\}.$$
Thus, if $\gamma_2>0$ is sufficiently small, and if $\ZygNorm{A}{s_0}[B^n(5)][\M^{n\times n}]\leq \gamma_2$,
we may solve for $R=R(A)\in N$ such that $\Psi(A,R)=0$, $R(0)=0$, $dR(0)=0$, and
\cref{Item::PfPhi1::BoundRs0,Item::PfPhi1::BoundRs} hold.
As we saw earlier, $\Psi(A,R)=0$ is equivalent to \cref{Item::PfPhi1::DivAhs}, and
\cref{Item::PfPhi1::HIsDiffeo,Item::PfPhi1::BoundLinfas} have already been verified.  This completes the proof.
\end{proof}

\begin{rmk}\label{Rmk::PfPhi1::WhyGreater1}
Throughout this paper, we fixed $s_0>1$.  It would be nice if we could achieve the same results for $s_0>0$, however technical
issues arise if we try to follow the methods of this paper with $s_0\in (0,1]$.
This is particularly notable in the proof of \cref{Lemma::PfPhi1::AuxLemma}.
When $s_0>1$, the solutions we consider to the PDE which arises in that lemma
are classical, however if $s_0\in (0,1]$, it seems likely one would have to consider some kind
of generalized solution.  A similar  problem occurs in the proof of \cref{Prop::PfRegularity::MainProp}.
\end{rmk}

\begin{proof}[Proof of \cref{Prop::Proofs::Phi1}]
Let $\gamma_1=\gamma_1(n)>0$ be as in \cref{Prop::PfRegularity::MainProp}.
We shrink $\gamma_1>0$, if necessary, to ensure that if $\Ah$ is an $n\times n$ matrix with components $\ah_j^k$ and $|\ah_j^k|\leq \gamma_1$,
then $\Norm{\Ah}[\M^{n\times n}]\leq \frac{1}{2}$.
We take $\sigma_{n,s_0}>0$ to be so small that if $\ZygNorm{R}{s_0+1}[B^n(4)][\R^n]\leq \sigma_{n,s_0}$
we have
\begin{itemize}
\item If $H(t)=t+R(t)$, then $B^n(2)\subseteq H(B^n(3))$.
\item $\det dH(t)\geq \frac{1}{2}$, $\forall t\in B^n(3)$.
\end{itemize}

Applying \cref{Lemma::PfPhi1::AuxLemma} with this choice of $\gamma_1$ and with $\sigma=\sigma_{n,s_0}$
yields $\gamma_2$ and $H$ as in that theorem.  Since $B^n(2)\subseteq H(B^n(3))$, by the choice
of $\sigma_{n,s_0}$, and in light of \cref{Lemma::PfPhi1::AuxLemma} \cref{Item::PfPhi1::HIsDiffeo},
we may define $\Phi_1:B^n(2)\rightarrow B^n(3)\subseteq B^n(5)$ by
$\Phi_1(t)=H^{-1}(t)$.
\Cref{Item::Phi1::At0}, \cref{Item::Phi1::Open}, and \cref{Item::Phi1::Diffeo} follow from the corresponding properties of $H$
described in \cref{Lemma::PfPhi1::AuxLemma}.

Since $\ZygNorm{H}{s+1}[B^n(3)][\R^n]\lesssim_{\Zygad{s}} 1$ (by \cref{Lemma::PfPhi1::AuxLemma} \cref{Item::PfPhi1::BoundRs})
and because $\det dH(t)\geq \frac{1}{2}$, $\forall t\in B^n(3)$ (by the choice of $\sigma=\sigma_{n,s_0}$), we have $\ZygNorm{\Phi_1}{s+1}[B^n(2)][\R^n]\lesssim_{\Zygad{s}} 1$ (see \cref{Lemma::ZygSpace::Inverse}),
proving \cref{Item::Phi1::Zygs}; the same proof gives \cref{Item::Phi1::Zygs0}.
Moreover, if $\Yh_j = \Phi_1^{*} Y_j=H_{*} Y_j$, we have $\ZygNorm{\Yh_j}{s}[B^n(2)][\R^n]\lesssim_{\Zygad{s}} 1$.
Writing $\Yh=\diff{t}+\Ah \diff{t}$, that $\Ah(0)=0$ follows from  \cref{Item::Phi1::At0} and the fact that $A(0)=0$.
That $\sup_{u\in B^n(1)} \Norm{\Ah(u)}[\M^{n\times n}]\leq \frac{1}{2}$ follows from the choice of $\gamma_1$ and
\cref{Lemma::PfPhi1::AuxLemma} \cref{Item::PfPhi1::BoundLinfas}.  This establishes \cref{Item::Phi1::A0}.

All that remains to establish are the two (clearly equivalent) statements \cref{Item::Phi1::As,Item::Phi1::Yhs}.
For this, we use \cref{Prop::PfRegularity::MainProp}.
Since $\ZygNorm{\Yh_j}{s}[B^n(2)][\R^n]\lesssim_{\Zygad{s}} 1$, we have $\ZygNorm{\ah_j^k}{s}[B^n(2)]\lesssim_{\Zygad{s}} 1$.
Also, we have
\begin{equation*}
[\Yh_j,\Yh_k] = \Phi_1^{*} [Y_j, Y_k] = \Phi_1^{*} \sum_{l=1}^n c_{j,k}^l Y_l = \sum_{l=1}^n \ch_{j,k}^l \Yh_l,
\end{equation*}
where $\ch_{j,k}^l = c_{j,k}^l\circ \Phi_1$.  Using \cref{Item::Phi1::Zygs}, \cref{Lemma::ZygSpace::Comp},
and the assumption $\ZygNorm{c_{j,k}^l}{s}[B^n(5)]\lesssim_{\Zygad{s}} 1$, this implies
$\ZygNorm{\ch_{j,k}^l}{s}[B^n(2)]\lesssim_{\Zygad{s}} 1$.
Finally,  \cref{Lemma::PfPhi1::AuxLemma} \cref{Item::PfPhi1::DivAhs,Item::PfPhi1::BoundLinfas}
show that all of the hypotheses of \cref{Prop::PfRegularity::MainProp} hold for $\Yh_1,\ldots, \Yh_n$.
Applying \cref{Prop::PfRegularity::MainProp} yields  \cref{Item::Phi1::As,Item::Phi1::Yhs}, completing the proof.
\end{proof}


	\subsection{\texorpdfstring{Construction of $\Phi_2$}{Proof of Main Technical Proposition}}\label{Section::Proofs::PfOfMainProp}
In this section, we prove \cref{Prop::Phi2::MainProp}, and we take the same setting and notation as in that proposition;
thus, we have vector fields $Y_1,\ldots, Y_n$ and functions $\ct_{i,j}^k$ as in that proposition, and we have a notion of $\Zygad{s}$-admissible constants given in \cref{Defn::Phi2::Admissible}.  Because of this definition of $\Zygad{s}$-admissible constants, it suffices
to assume $s\geq s_0$ in all of \cref{Prop::Phi2::MainProp}.  Thus, in this section we consider only $s\geq s_0$.

\begin{lemma}\label{Prop::Phi2::AuxLemma}
Define, for $\gamma\in (0,1]$, $\Psi_\gamma:B^n(\eta_1/\gamma)\rightarrow B^n(\eta_1)$ by $\Psi_\gamma(t) = \gamma t$.
Let $Y_j^{\gamma}:=\gamma\Psi_\gamma^{*} Y_j$.  Then,
$Y_j^{\gamma}=\diff{t} + A_{\gamma} \diff{t}$ 
and $[Y_j^{\gamma},Y_k^{\gamma}]=\sum_{l=1}^n c_{j,k}^{l,\gamma} Y_l^{\gamma}$, where for $\gamma\in (0, \min\{\frac{\eta_1}{5},1\}]$, $s\geq s_0$,
\begin{equation}\label{Eqn::Phi1::AuxLemma}
\ZygNorm{A_\gamma}{s}[B^n(5)][\M^{n\times n}]\lesssim_{\Zygad{s}} \gamma,\quad \ZygNorm{c_{j,k}^{l,\gamma}}{s}[B^n(5)]\lesssim_{\Zygad{s}} \gamma.
\end{equation}
\end{lemma}
\begin{proof}
Since $A_\gamma(t)=A(\gamma t)$ and $A(0)=0$, that $\ZygNorm{A_\gamma}{s}[B^n(5)][\M^{n\times n}]\lesssim_{\Zygad{s}} \gamma$ follows from \cref{Lemma::FuncSpaceRev::Scale}
(this uses $s\geq s_0>1$).
Since $c_{j,k}^{l,\gamma}(t)= \gamma c_{j,k}^{l}(\gamma t)$, $\ZygNorm{c_{j,k}^{l,\gamma}}{s}[B^n(5)]\lesssim_{\Zygad{s}} \gamma$ follows directly from the definitions (this uses $\gamma\in (0,1]$).
\end{proof}

\begin{proof}[Proof of \cref{Prop::Phi2::MainProp}]
Let $A_\gamma$, $c_{j,k}^{l,\gamma}$, and $Y_j^{\gamma}$ be as in \cref{Prop::Phi2::AuxLemma}.
Fix $\gamma_2=\gamma_2(n,s_0)>0$ as in \cref{Prop::Proofs::Phi1}.
Take $\gamma\approx_{\Zygad{s_0}} 1$ so small $\gamma\leq \min\{\frac{\eta_1}{5},1\}$ and
$\ZygNorm{A_{\gamma}}{s_0}[B^n(5)][\M^{n\times n}]\leq \gamma_2$ (this is clearly possible by \cref{Eqn::Phi1::AuxLemma}).
With this choice of $\gamma$, we have $\ZygNorm{c_{j,k}^{l,\gamma}}{s}[B^n(5)]\lesssim_{\Zygad{s}} \gamma\leq 1$
and $\ZygNorm{A_\gamma}{s}[B^n(5)]\lesssim_{\Zygad{s}} \gamma\leq 1$, for $s\geq s_0$, by \cref{Eqn::Phi1::AuxLemma}.

In light of these remarks, \cref{Prop::Proofs::Phi1} applies to $Y_1^{\gamma},\ldots, Y_n^{\gamma}$
to yield a map $\Phi_1:B^n(1)\rightarrow B^n(5)$ as in that proposition (and constants which are $\Zygad{s}$-admissible in the sense of that proposition
are $\Zygad{s}$-admissible in the sense of this section).
Let $\Yh_j^{\gamma}=\Phi_1^{*} Y_j^{\gamma}$.  

Set $\Phi_2:=\Psi_\gamma\circ \Phi_1:B^n(1)\rightarrow B^n(\eta_1)$,
and let $\Yh_j=\Phi_2^{*} Y_j$.  Note that $\Yh_j= K \Yh_j^{\gamma}$,
where $K:=\frac{1}{\gamma}\geq 1$ is an $\Zygad{s_0}$-admissible constant.
With this choice of $K$ and $\Phi_2$, the proposition follows from the corresponding results 
for $\Phi_1$ and $\Yh_1^{\gamma},\ldots, \Yh_n^{\gamma}$ given in \cref{Prop::Proofs::Phi1}.
\end{proof}
	
	\subsection{Qualitative Results}\label{Section::Proofs::Qual}
We now turn to the qualitative results; i.e., \cref{Thm::QualRes::LocalQual,Thm::QualRes::GlobalQual}.  These are simple
consequences of \cref{Thm::QuantRes::MainThm}.  We begin with \cref{Thm::QualRes::LocalQual}.  For this we recall
\cite[\SSLemmaMoreOnAssump]{StovallStreet}.

\begin{lemma}[\SSLemmaMoreOnAssump{} of \cite{StovallStreet}]\label{Lemma::PfQual::AlwaysHaveEtaDelta}
Let $X_1,\ldots, X_q$ be $C^1$ vector fields on a $C^2$ manifold $\fM$.
\begin{itemize}
\item $\forall x_0\in \fM$, $\exists \eta>0$, such that $X_1,\ldots, X_q$ satisfy $\sC(x_0,\eta,\fM)$.
\item Let $K\Subset \fM$ be a compact set.  Then, there exists $\delta_0>0$ such that
$\forall \theta\in S^{q-1}$ if $x\in K$ is such that $\theta_1 X_1(x)+\cdots+\theta_q X_q(x)\ne 0$,
then $\forall r\in (0,\delta_0]$,
$$e^{r\theta_1 X_1+\cdots + r\theta_q X_q}x\ne x.$$
\end{itemize}
\end{lemma}

\begin{rmk}\label{Rmk::PfQual::AlwaysHaveEtaDelta}
\Cref{Lemma::PfQual::AlwaysHaveEtaDelta} shows that we always have $\eta$ and $\delta_0$ as in the assumptions of \cref{Thm::QuantRes::MainThm}.
Thus, if we wish to apply \cref{Thm::QuantRes::MainThm} to obtain a qualitative result, we do not need to verify the existence of $\eta$ and $\delta_0$.
\end{rmk}


\begin{proof}[Proof of \cref{Thm::QualRes::LocalQual}]
\Cref{Item:QualRes::Local::Diffeo}$\Rightarrow$\cref{Item::QualRes::LocalQual::Basis}:  First we prove the result with $s<\infty$.
Let $U$, $V$, $x_0$, and $\Phi$ be as in \cref{Item:QualRes::Local::Diffeo}.  Without loss of generality
assume $0\in U$ and $\Phi(0)=x_0$.
Reorder $X_1,\ldots, X_q$ so that $X_1(x_0),\ldots, X_n(x_0)$ are linearly independent and let $Y_j=\Phi^{*}X_j$, so
that $Y_j\in \ZygSpace{s+1}[U][\R^n]$, $1\leq j\leq q$.  Note that $Y_1(0),\ldots, Y_n(0)$
span the tangent space $T_0U$.  Let $\eta>0$ be so small $B^n(2\eta)\subset U$ and $Y_1,\ldots, Y_n$ form a basis for the tangent space
on $B^n(2\eta)$.
It is immediate to verify, for $1\leq j,k\leq q$, that 
\begin{equation}\label{Eqn::PfQual::YjCommute::TempEqn::1}
[Y_j,Y_k]=\sum_{l=1}^n \ct_{j,k}^l Y_l,
\end{equation} 
where $\ct_{j,k}^l\in \ZygSpace{s}[B^n(\eta)]$.
Because $Y_1,\ldots, Y_q$ span the tangent space at every point of $B^n(2\eta)$
and $Y_j\in \ZygSpace{s+1}[B^n(2\eta)][\R^n]$, $1\leq j\leq q$,  \cref{Cor::FuncSpaceRev::CompareEuclidean}
implies 
\begin{equation}\label{Eqn::PfQual::YjCommute::TempEqn::2}
\ct_{i,j}^k\in  \ZygSpace{s}[B^n(\eta)]=\ZygXSpace{Y}{s}[B^n(\eta)].
\end{equation}
Pushing \cref{Eqn::PfQual::YjCommute::TempEqn::1} forward via $\Phi$ shows $[X_j,X_k]=\sum_{l=1}^n \ch_{j,k}^l X_l$, with $\ch_{j,k}^l=\ct_{j,k}^l\circ \Phi^{-1}$. 
\Cref{Eqn::Results::NormsDiffoInv} and \cref{Eqn::PfQual::YjCommute::TempEqn::2}
combine to give $\ch_{j,k}^l\in \ZygXSpace{X}{s}[\Phi(B^n(\eta))]$.

Using that $Y_1,\ldots, Y_n$ span the tangent space at every point of $B^n(2\eta)$ and that 
$Y_j\in \ZygSpace{s+1}[U][\R^n]$, $1\leq j\leq q$,
for $n+1\leq j\leq q$, we may write 
\begin{equation}\label{Eqn::PfQual::TmpYasSpan}
Y_j=\sum_{k=1}^n \bt_j^k Y_k
\end{equation} 
where $\bt_j^k\in \ZygSpace{s+1}[B^n(\eta)]$.
By \cref{Cor::FuncSpaceRev::CompareEuclidean}, 
$\bt_j^k\in \ZygSpace{s+1}[B^n(\eta)]=\ZygXSpace{Y}{s+1}[B^n(\eta)]$, and by
\cref{Eqn::Results::NormsDiffoInv},
$b_j^k=\bt_j^k\circ\Phi^{-1}\in \ZygXSpace{X}{s+1}[\Phi(B^n(\eta))]$.
Pushing \cref{Eqn::PfQual::TmpYasSpan} forward via $\Phi$, we have $X_j=\sum_{k=1}^n b_j^k X_k$ on $\Phi(B^n(\eta))$ this completes the proof of \cref{Item::QualRes::LocalQual::Basis}
with $V$ replaced by $\Phi(B^n(\eta))$, when $s<\infty$.  

If $s=\infty$ note that in the above proof $\eta$, $\ch_{j,k}^l$, and $b_j^k$ can be chosen independent of $s$,
thus when $s=\infty$ the above proof applied to each $s<\infty$ completes the proof of \cref{Item::QualRes::LocalQual::Basis}.

\Cref{Item::QualRes::LocalQual::Basis}$\Rightarrow$\cref{Item::QualRes::LocalQual::Spanning}:  
Suppose \cref{Item::QualRes::LocalQual::Basis} holds.  We wish to show for $1\leq i,j\leq q$,
\begin{equation}\label{Eqn::PfDesnity::LocalQual::ToShowCommute}
[X_i,X_j]=\sum_{k=1}^q c_{i,j}^k X_k, \quad c_{i,j}^k\in \ZygXSpace{X}{s}[V].
\end{equation}
where $s$ and $V$ are as in \cref{Item::QualRes::LocalQual::Basis}.  For $1\leq i,j\leq n$, \cref{Eqn::PfDesnity::LocalQual::ToShowCommute} is contained in \cref{Item::QualRes::LocalQual::Basis}.
We prove the result for $n+1\leq i,j\leq q$.  The remaining cases ($1\leq i\leq n$ and $n+1\leq j\leq q$, or $n+1\leq i\leq q$ and $1\leq j\leq n$) are similar and easier.
We have
\begin{equation*}
[X_i,X_j]=\left[ \sum_{k_1=1}^n b_i^{k_1} X_{k_1}, \sum_{k_2=1}^n b_j^{k_2} X_{k_2} \right] =\sum_{k_1,k_2=1}^n \left( b_i^{k_1} (X_{k_1} b_j^{k_2}) X_{k_2} - b_j^{k_2}(X_{k_2} b_i^{k_1}) X_{k_1} +\sum_{l=1}^n b_i^{k_1} b_j^{k_2} \ch_{k_1,k_2}^l X_l \right).
\end{equation*}
We are given $b_j^k\in \ZygXSpace{X}{s+1}[V]$ and $\ch_{k_1,k_2}^l\in \ZygXSpace{X}{s}[V]$.  It follows immediately from the definition of $\ZygXSpace{X}{s+1}$ that
$X_l b_j^k\in \ZygXSpace{X}{s}[V]$.  From here, \cref{Eqn::PfDesnity::LocalQual::ToShowCommute} follows from the fact that $\ZygXSpace{X}{s}[V]$ is an algebra (see \cref{Lemma::ZygSpace::Algebra}), completing the proof
of \cref{Item::QualRes::LocalQual::Spanning}.

\Cref{Item::QualRes::LocalQual::Spanning}$\Rightarrow$\cref{Item:QualRes::Local::Diffeo}:  This is a consequence of \cref{Thm::QuantRes::MainThm}.  We make a few comments to this end.
First of all, as discussed in \cref{Lemma::PfQual::AlwaysHaveEtaDelta,Rmk::PfQual::AlwaysHaveEtaDelta}, there exist $\eta$ and $\delta_0$ as in the hypotheses
of \cref{Thm::QuantRes::MainThm}.  Fix any $s_0\in (1,s]\setminus\{\infty\}$ and take $\xi>0$ so small $B_X(x_0,\xi)\subseteq V$.  Take $J_0$ as in \cref{Thm::QuantRes::MainThm} (with $\zeta=1$).
We have, directly from the definitions, 
$$c_{i,j}^k\in \ZygXSpace{X}{s}[V]\subseteq \ZygXSpace{X}{s}[B_{X}(x_0,\xi)]\subseteq \ZygXSpace{X_{J_0}}{s}[B_{X_{J_0}}(x_0,\xi)]\subseteq \ZygXSpace{X_{J_0}}{s_0}[B_{X_{J_0}}(x_0,\xi)].$$
Thus, all of the hypotheses of \cref{Thm::QuantRes::MainThm} hold for this choice of $s_0$.  This yields a map $\Phi$ as in \cref{Thm::QuantRes::MainThm}.
This map satisfies the conclusions of \cref{Item:QualRes::Local::Diffeo}, and this completes the proof.
\end{proof}

We now turn to \cref{Thm::QualRes::GlobalQual}.  The uniqueness of the $\ZygSpace{s+2}$ structure described in that theorem follows from the next lemma.

\begin{lemma}\label{Lemma:QualRes::Global::AuxUnique}
Fix $s\in (0,\infty]$.
Let $M$ and $N$ be two $n$-dimensional $\ZygSpace{s+2}$ manifolds, and suppose $X_1,\ldots, X_q$ are $\ZygSpace{s+1}$ vector fields
 on $M$
which span the tangent space at every point, and $Z_1,\ldots, Z_q$ are $\ZygSpace{s+1}$ vector fields on $N$.
Let $\Psi:M\rightarrow N$ be a $C^2$ diffeomorphism such that $\Psi_{*}X_j=Z_j$.  Then $\Psi$ is a $\ZygSpace{s+2}$ diffeomorphism.
\end{lemma}
\begin{proof}
We first prove the result in the special case when $M$ and $N$ are open subsets of $\R^n$; in this case we can identify the vector fields with
$\R^n$ valued functions, in the usual way.  We use $x$ to denote points in $M\subseteq \R^n$
and $y$ to denote a points in $N\subseteq \R^n$.

Fix a point $x_0\in M$, we will show $\Psi\in \ZygSpaceloc{s+1}$ on a neighborhood of $x_0$; since $x_0\in M$ is arbitrary, this will complete the proof of
the case when $M$ and $N$ are open subsets of $\R^n$.  Reorder $X_1,\ldots, X_q$ so that $X_1(x_0),\ldots, X_n(x_0)$ are linearly independent;
and
reorder $Z_1,\ldots, Z_q$ in the same way to that we still have $\Psi_{*} X_j =Z_j$.  Since $X_1(x_0),\ldots, X_n(x_0)$ form a basis of $T_{x_0}M$,
we may pick an open neighborhood $U$ of $x_0$ so that $X_1(x),\ldots, X_n(x)$ form a basis for the tangent space at every $x\in U$.

Let $\Matrix{X}(x):= \mleft(X_1(x)|\cdots |X_n(x)\mright)$; i.e., $\Matrix{X}$ is the $n\times n$ matrix whose columns are given by the vectors $X_1,\ldots, X_n$.
Similarly, let $\Matrix{Z}(y)=\mleft(Z_1(y)|\ldots |Z_n(y)\mright)$.  By hypothesis, we have
$\Matrix{X}\in \ZygSpaceloc{s+1}[M][\M^{n\times n}]$ and $\Matrix{Z}\in \ZygSpaceloc{s+1}[N][\M^{n\times n}]$.
Since $\Psi_{*} X_j=Z_j$, we have the matrix equation
\begin{equation}\label{Eqn::PfQual::TmpMatrixEqn1}
	d\Psi(x) \Matrix{X}(x) = \Matrix{Z}(\Psi(x)), \quad x\in M.
\end{equation}
Since $X_1,\ldots, X_n$ span the tangent space at every point of $U$, the matrix $\Matrix{X}$ is invertible, $\forall x\in U$.
It follows from \cref{Lemma::ZygSpace::Algebra} (by using the cofactor formula for $\Matrix{X}(x)^{-1}$), that
$\Matrix{X}(\cdot)^{-1}\in \ZygSpaceloc{s+1}[U][\M^{n\times n}]$.  From \cref{Eqn::PfQual::TmpMatrixEqn1}, we obtain
\begin{equation}\label{Eqn::PfQual::TmpMatrixEqn2}
	d\Psi(x) = \Matrix{Z}(\Psi(x))\Matrix{X}(x)^{-1}, \quad x\in U.
\end{equation}

Suppose $\Psi\in \ZygSpaceloc{s'+2}[U][\R^n]$, for some $s'\geq 0$.  We will show
$\Psi\in \ZygSpaceloc{\min\{s'+3,s+2\}}[U][\R^n]$; and then it will follow by iteration that $\Psi\in \ZygSpaceloc{s+2}[U][\R^n]$, as desired.
This will complete the proof since $C^2_{\mathrm{loc}}(U;\R^n)\subset \ZygSpaceloc{2}[U][\R^n]$.
Since $\Psi\in \ZygSpaceloc{s'+2}[U][\R^n]$ and $\Matrix{Z}\in \ZygSpaceloc{s+1}[N][\M^{n\times n}]$, it follows from \cref{Lemma::FuncSpaceZyg::Compose::Euclid}
that $\Matrix{Z}\circ \Psi\in \ZygSpaceloc{\min\{s'+2,s+1\}}[U][\M^{n\times n}]$.
Since we have already shown $\Matrix{X}(\cdot)^{-1}\in \ZygSpaceloc{s+1}[U][\M^{n\times n}]$, it follows from  \cref{Lemma::ZygSpace::Algebra} and \cref{Eqn::PfQual::TmpMatrixEqn2}
that
\begin{equation*}
	d\Psi(x) =  \Matrix{Z}(\Psi(x))\Matrix{X}(x)^{-1} \in \ZygSpaceloc{\min\{s'+2, s+1\}}[U][\M^{n\times n}].
\end{equation*}
Since we also have $\Psi\in \ZygSpaceloc{s'+2}[U][\R^n]$, it follows that $\Psi\in \ZygSpaceloc{\min\{s'+3,s+2\}}[U][\R^n]$, as desired.  This completes the proof
in the case when $M$ and $N$ are open subsets of $\R^n$.

We now turn to the general case, where $M$ and $N$ are $\ZygSpace{s+2}$ manifolds of dimension $n$, and $X_1,\ldots, X_q$, $Z_1,\ldots, Z_q$, and $\Psi$
are as in the statement of the lemma.  Since $M$ and $N$ are $\ZygSpace{s+2}$ manifolds they have associated $\ZygSpace{s+2}$ atlases $\{ (\phi_{\alpha}, V_\alpha)\}$ and
$\{ (\psi_\beta, W_\beta)\}$, respectively.  We wish to show, $\forall \alpha, \beta$, 
\begin{equation*}
	\Psi_{\alpha,\beta}:= \psi_\beta\circ\Psi\circ \phi_{\alpha}^{-1}: \phi_\alpha \mleft( V_\alpha\bigcap \Psi^{-1}(W_\beta)\mright) \rightarrow \psi_\beta\mleft( \Psi(V_\alpha)\bigcap W_\beta\mright)
\end{equation*}
is a $\ZygSpace{s+2}$ diffeomorphism, and this will complete the proof.

By hypothesis, we have
\begin{equation*}
	(\Psi_{\alpha,\beta})_{*} \mleft( (\phi_{\alpha})_{*} X_j\mright) = (\psi_\beta)_{*} Z_j.
\end{equation*}
Since $(\phi_{\alpha})_{*}X_1,\ldots, (\phi_{\alpha})_{*}X_q$ and $(\psi_\beta)_{*}Z_1,\ldots, (\psi_\beta)_{*} Z_q$ are $\ZygSpace{s+1}$ vector fields, by hypothesis,
and $(\phi_{\alpha})_{*}X_1,\ldots, (\phi_{\alpha})_{*}X_q$ span the tangent space at every point of $\phi_{\alpha}(V_\alpha)$, it follows from the above case
(when $M$ and $N$ are open subsets of $\R^n)$, that $\Psi_{\alpha,\beta}$ is a $\ZygSpace{s+2}$ diffeomorphism.  This completes the proof.
\end{proof}

\begin{proof}[Proof of \cref{Thm::QualRes::GlobalQual}]
\Cref{Item::QualRes::Global::Charts}$\Rightarrow$\cref{Item::QualRes::Global::Atlas}:  Under the condition \cref{Item::QualRes::Global::Charts}, for each $x\in M$,
there exist open sets $U_x\subseteq \R^n$, $V_x\subseteq M$, and a $C^2$ diffeomorphism $\Phi_x:U_x\rightarrow V_x$ such that
if $Y_j^{x} = \Phi_x^{*} X_j$, then $Y_j^x\in \ZygSpace{s+1}[U_x][\R^n]$.  We wish to show that the collection
$\{ (\Phi_x^{-1}, V_x) : x\in M\}$ forms a $\ZygSpace{s+2}$ atlas on $M$; once that is shown, \cref{Item::QualRes::Global::Atlas} will follow since
the $X_j$ will be $\ZygSpace{s+1}$ with respect to this atlas by definition, and this atlas is clearly 
compatible with the $C^2$ structure on $M$.
Hence, we need only verify that the transition functions are $\ZygSpaceloc{s+2}$.  Take $x_1,x_2\in M$ such that $V_{x_1}\cap V_{x_2}\ne \emptyset$.
Set 
$\Psi=\Phi_{x_2}^{-1}\circ\Phi_{x_1}:U_{x_1}\cap \Phi_{x_1}^{-1}(V_{x_2})\rightarrow U_{x_2} \cap \Phi_{x_2}^{-1}(V_{x_1})$.
We wish to show $\Psi$ is a $\ZygSpace{s+2}$ diffeomorphism.  We already know $\Psi$ is a $C^2$ diffeomorphism and
$\Psi_{*} Y_j^{x_1} = Y_j^{x_2}$.  That $\Psi$ is a $\ZygSpace{s+2}$ diffeomorphism now follows from \cref{Lemma:QualRes::Global::AuxUnique},
completing the proof of \cref{Item::QualRes::Global::Atlas}.

\Cref{Item::QualRes::Global::Atlas}$\Rightarrow$\cref{Item::QualRes::Global::Commutator}:  Suppose \cref{Item::QualRes::Global::Atlas} holds.
Using a simple partition of unity argument, we may write $[X_j,X_k]=\sum_{l=1}^q c_{j,k}^l X_l$, where $c_{j,k}^l:M\rightarrow \R$ and are $\ZygSpaceloc{s}$ maps.  
We wish to show $\forall x_0\in M$, $\exists V\subseteq M$ open with $x_0\in V$ and $c_{j,k}^l\big|_V\in \ZygXSpace{X}{s}[V]$.
Fix $x_0\in M$,
and let $W\subseteq M$ be a neighborhood of $x_0$ such that there is a $\ZygSpace{s+2}$ diffeomoprhism $\Phi:B^n(1)\rightarrow W$
with $\Phi(0)=x_0$.  Set $Y_j=\Phi^{*} X_j$, so that $Y_j\in \ZygSpace{s+1}[B^n(3/4)][\R^n]$ and $Y_1,\ldots, Y_q$ span the tangent
space at every point of $B^n(1)$.  Also we have $c_{j,k}^l\circ\Phi\in \ZygSpace{s}[B^n(3/4)]$.
\Cref{Cor::FuncSpaceRev::CompareEuclidean}
shows $c_{j,k}^l\circ\Phi\in \ZygSpace{s}[B^n(1/2)]=\ZygXSpace{Y}{s}[B^n(1/2)]$
and 
\cref{Eqn::Results::NormsDiffoInv}
shows $c_{j,k}^l\in \ZygXSpace{X}{s}[\Phi(B^n(1/2))]$.
This proves \cref{Item::QualRes::Global::Charts} with $V=\Phi(B^n(1/2))$.

\Cref{Item::QualRes::Global::Commutator}$\Rightarrow$\cref{Item::QualRes::Global::Charts}:  This is obvious.

Finally, as mentioned before, the uniqueness of the $\ZygSpace{s+2}$ manifold structure, as described in the theorem,
is an immediate consequence of \cref{Lemma:QualRes::Global::AuxUnique}.
\end{proof}

\section{H\"older Spaces}\label{Section::Holder}
Let $\Omega\subset \R^n$ be a bounded, Lipschitz domain.
It is easy to see that for $m\in \N$, $s\in [0,1]$, $m+s>0$ we have the containment $\HSpace{m}{s}[\Omega]\subseteq \ZygSpace{m+s}[\Omega]$.
For $m\in \N$, $s\in (0,1$), we also have the reverse containment $\ZygSpace{m+s}[\Omega]\subseteq \HSpace{m}{s}[\Omega]$;
this follows easily from \cite[Theorem 1.118 (i)]{TriebelTheoryOfFunctionSpacesIII}.

When we move to the corresponding spaces with respect to $C^1$ vector fields $X_1,\ldots, X_q$ on a $C^2$ manifold $M$, we have similar results.
For any $m\in \N$, $s\in [0,1]$ with $m+s>0$, we have $\HXSpace{X}{m}{s}[M]\subseteq \ZygXSpace{X}{m+s}[M]$; see \cite[\SSCompareFunctionSpaces]{StovallStreet}.
The reverse containment for $m\in \N$ and $s\in (0,1)$ is a bit more difficult and requires appropriate hypotheses on the vector fields.
We state a quantitative local version of this in the next proposition.

\begin{prop}\label{Prop::Holder::EquivToZyg}
We take all the same assumptions and notation as in \cref{Thm::QuantRes::MainThm}, and let $\Phi$ be as in that theorem (and $\Zygad{s}$-admissible constants as
in \cref{Defn::Results::sadmissible}).
Then, for $m\in \N$, $s\in (0,1)$, and for any function $f\in \CSpace{\Phi(B^n(1))}$,
\begin{equation}\label{Eqn::Holder::ToShow::Equiv}
	\HXNorm{f}{X}{m}{s}[\Phi(B^n(1))]\approx_{\Zygad{m+s-2}} \ZygXNorm{f}{X}{m+s}[\Phi(B^n(1))].
\end{equation}
\end{prop}
\begin{proof}
We use \cref{Lemma::ZygSpaces::EquivHolderZyg}; in particular, for $g\in \CSpace{B^n(1)}$, $m\in \N$, $s\in (0,1)$,
\begin{equation}\label{Eqn::Holder::ClassicalEquiv}
\HNorm{g}{m}{s}[B^n(1)]\approx \ZygNorm{g}{m+s}[B^n(1)],
\end{equation}
where the implicit constants depend only on $m+s$ and $n$.  
Let $Y_j=\Phi^{*} X_j$, and let $A$ be as in \cref{Thm::QuantRes::MainThm}.
Letting $Y_j=\Phi^{*} X_j$, \cref{Thm::QuantRes::MainThm} \cref{Item::QuantRes::YZyg}
shows $\ZygNorm{Y_j}{m+s-1}[B^n(1)][\R^n]\lesssim_{\Zygad{m+s-2}} 1$, and therefore by \cref{Eqn::Holder::ClassicalEquiv},
$\HNorm{Y_j}{m-1}{s}[B^n(1)][\R^n]\lesssim_{\Zygad{m+s-2}} 1$.  
Here, we are using the convention in \cref{Rmk::FuncSpaceRev::Convention} to define $\HSpace{-1}{\cdot}$ and $\ZygSpace{s-1}$ when $s-1\leq 0$. 
Similarly, we have $\ZygNorm{A}{m+s-1}[B^n(1)][\M^{n\times n}], \HNorm{A}{m-1}{s}[B^n(1)][\R^n]\lesssim_{\Zygad{m+s-2}} 1$.

Since $Y_{J_0}=K(I+A)\grad$ we have $\grad = K^{-1} (I+A)^{-1} Y_{J_0}$.  Thus, we may write $\grad$ as a linear combination of $Y_1,\ldots, Y_n$,
with coefficients whose $\HSpace{m-1}{s}[B^n(1)]$ and $\ZygSpace{m+s-1}[B^n(1)]$ norms are $\lesssim_{\Zygad{m+s-2}} 1$.

With all of the above remarks, 
\cref{Prop::FuncSpaceRev::CompareEucldiean}
shows for any $g\in B^n(1)$,
\begin{equation*}
	\HNorm{g}{m}{s}[B^n(1)]\approx_{\Zygad{m+s-2}} \HXNorm{g}{Y}{m}{s}[B^n(1)], \quad \ZygNorm{g}{m+s}[B^n(1)]\approx_{\Zygad{m+s-2}} \ZygXNorm{g}{Y}{m+s}[B^n(1)].
\end{equation*}
Combining this with \cref{Eqn::Holder::ClassicalEquiv}, we have
\begin{equation}\label{Eqn::Holder::TempequivA}
	\HXNorm{g}{Y}{m}{s}[B^n(1)] \approx_{\Zygad{m+s-2}}  \ZygXNorm{g}{Y}{m+s}[B^n(1)].
\end{equation}
\Cref{Eqn::Results::NormsDiffoInv} shows
\begin{equation*}
	\HXNorm{f\circ \Phi}{Y}{m}{s}[B^n(1)] = \HXNorm{f}{X}{m}{s}[\Phi(B^n(1))], \quad \ZygXNorm{f\circ \Phi}{Y}{m+s}[B^n(1)] = \ZygXNorm{f}{X}{m+s}[\Phi(B^n(1))].
\end{equation*}
Combining this and \cref{Eqn::Holder::TempequivA} with $g=f\circ \Phi$ yields \cref{Eqn::Holder::ToShow::Equiv} and completes the proof.
\end{proof}

Similarly, we may create H\"older versions of \cref{Thm::QualRes::LocalQual,Thm::QualRes::GlobalQual}.
We state these results here.  We take the same setup as in \cref{Thm::QualRes::LocalQual,Thm::QualRes::GlobalQual}.

\begin{cor}[The Local Result]\label{Cor::QualRes::LocalQual::Holder}
For $m\in \N$, $m\geq 1$ and $s\in (0,1)$ the following three conditions are equivalent:
\begin{enumerate}[(i)]
\item\label{Item:QualRes::Local::Diffeo::Holder} There is an open neighborhood $V\subseteq M$ of $x_0$ and a $C^2$ diffeomorphism $\Phi:U\rightarrow V$ where
$U\subseteq \R^n$ is open, such that $\Phi^{*}X_1,\ldots, \Phi^{*}X_q\in \HSpace{m+1}{s}[U][\R^n]$.

\item\label{Item::QualRes::LocalQual::Basis::Holder} 
Re-order the vector fields so that $X_1(x_0),\ldots, X_n(x_0)$ are linearly
independent.
There is an open neighborhood $V\subseteq M$ of $x_0$ such that:
    \begin{itemize}
    \item $[X_i,X_j]=\sum_{k=1}^n \ch_{i,j}^k X_k$, $1\leq i,j\leq n$, where $\ch_{i,j}^k\in \HXSpace{X}{m}{s}[V]$.
    \item For $n+1\leq j\leq q$, $X_j=\sum_{k=1}^n b_j^k X_k$, where $b_j^k\in \HXSpace{X}{m+1}{s}[V]$.
    \end{itemize}

\item\label{Item::QualRes::LocalQual::Spanning::Holder} There exists an open neighborhood $V\subseteq M$ of $x_0$ such that
$[X_i,X_j]=\sum_{k=1}^q c_{i,j}^k X_k$, $1\leq i,j\leq q$, where $c_{i,j}^k\in \HXSpace{X}{m}{s}[V]$.
\end{enumerate}

\end{cor}
\begin{proof}
\cref{Item:QualRes::Local::Diffeo::Holder}$\Rightarrow$\cref{Item::QualRes::LocalQual::Basis::Holder}$\Rightarrow$\cref{Item::QualRes::LocalQual::Spanning::Holder} has a nearly
identical proof to the corresponding results in \cref{Thm::QualRes::LocalQual}, and we leave the details to the reader.
Assume \cref{Item::QualRes::LocalQual::Spanning::Holder} holds.  Then, since $\HXSpace{X}{m}{s}[V]\subseteq \ZygXSpace{X}{m+s}[V]$ (by \cite[\SSCompareFunctionSpaces]{StovallStreet})
we have that \cref{Thm::QualRes::LocalQual} \cref{Item::QualRes::LocalQual::Spanning} holds (with $s$ replaced by $m+s$).  Therefore, \cref{Thm::QualRes::LocalQual} \cref{Item:QualRes::Local::Diffeo} holds (again, with
$s$ replaced by $m+s$);
we may shrink $U$ in \cref{Thm::QualRes::LocalQual} \cref{Item:QualRes::Local::Diffeo} so that it is a Euclidean ball.  Letting $\Phi$ be as in
\cref{Thm::QualRes::LocalQual} \cref{Item:QualRes::Local::Diffeo}, we have $\Phi^{*}X_1,\ldots, \Phi^{*}X_q\in \ZygSpace{m+s+1}[U][\R^n]$.
Since $U$ is a ball, \cref{Lemma::ZygSpaces::EquivHolderZyg} shows $\ZygSpace{m+s+1}[U][\R^n]=\HSpace{m+1}{s}[U][\R^n]$ (this is the point
where we use $s\ne 0,1$).  \Cref{Item:QualRes::Local::Diffeo::Holder}
follows, completing the proof.
\end{proof}

\begin{rmk}\label{Rmk::QualRes::Sne01::Holder}
The only place $m\geq 1$, $s\ne 0,1$ was used in \cref{Cor::QualRes::LocalQual::Holder} was \cref{Item::QualRes::LocalQual::Spanning::Holder}$\Rightarrow$\cref{Item:QualRes::Local::Diffeo::Holder}.
The implications \cref{Item:QualRes::Local::Diffeo::Holder}$\Rightarrow$\cref{Item::QualRes::LocalQual::Basis::Holder}$\Rightarrow$\cref{Item::QualRes::LocalQual::Spanning::Holder} hold
for $m\in \N$, $s\in [0,1]$ with the same proof.  We do not know whether \cref{Item::QualRes::LocalQual::Spanning::Holder}$\Rightarrow$\cref{Item:QualRes::Local::Diffeo::Holder} holds
for $m=0$ or $s=0,1$.
\end{rmk}

\begin{cor}[The Global Result]\label{Cor::Holder::GlobalResult}
For $m\in \N$, $m\geq 1$ and $s\in (0,1)$, the following three conditions are equivalent.
\begin{enumerate}[(i)]
\item\label{Item::QualRes::Global::Atlas::Holder} There exists a $\HSpace{m+2}{s}$ atlas on $M$, compatible with its $C^2$ structure,
such that $X_1,\ldots, X_q$ are $\HSpace{m+1}{s}$ with respect to this atlas.
\item\label{Item::QualRes::Global::Charts::Holder} For each $x_0\in M$, any of the three equivalent conditions from \cref{Cor::QualRes::LocalQual::Holder} holds for this choice of $x_0$.

\item\label{Item::QualRes::Global::Commutator::Holder} $[X_i,X_j]=\sum_{k=1}^q c_{i,j}^k X_k$, $1\leq i,j\leq q$, where $\forall x_0\in M$, $\exists V\subseteq M$ open with $x_0\in V$ such that $c_{i,j}^k\big|_V\in \HXSpace{X}{m}{s}[V]$, $1\leq i,j,k\leq q$.
\end{enumerate}
Furthermore, under these conditions, the $\HSpace{m+2}{s}$ manifold structure on $M$ induced by the atlas from \cref{Item::QualRes::Global::Atlas::Holder}
is unique, in the sense that if there is another $\HSpace{m+2}{s}$ atlas on $M$, compatible with its $C^2$ structure, and such that
$X_1,\ldots, X_q$ are $\HSpace{m+1}{s}$ with respect to this second atlas, then the identity
map $M\rightarrow M$ is a $\HSpace{m+2}{s}$ diffeomorphism between these two $\HSpace{m+2}{s}$ manifold structures
on $M$.
\end{cor}
\begin{proof}
With \cref{Cor::QualRes::LocalQual::Holder}
in hand, the proof is nearly identical to the proof of \cref{Thm::QualRes::GlobalQual} and we leave the details to the reader.
\end{proof}


\appendix
\section{Elliptic PDEs}
We require quantitative versions of some standard results from elliptic PDEs.  The proofs of these results are well-known,
and the quantitative versions follow by keeping track of constants in the proofs. We make no effort
to present the results or proofs in greatest generality, and only present what is needed for this paper.

	\subsection{Regularity of Linear Elliptic Equations}
Let $\sE$ be a constant coefficient partial differential operator of order $M$,
\begin{equation*}
\sE:\CjSpace{\infty}[\R^n][\C^{m_1}]\rightarrow \CjSpace{\infty}[\R^n][\C^{m_2}],
\end{equation*}
where $m_2\geq m_1$.  We may think of $\sE$ as a $m_2\times m_1$ matrix of
constant coefficient partial differential operators of order $\leq M$.

Fix $D\in (0,\infty)$.
Let $\sL=\sum_{|\alpha|\leq M} c_{\alpha}(x) \partial_x^{\alpha}$ where $c_{\alpha}:B^n(D)\rightarrow \M^{m_2\times m_1}(\C)$.
For $u:B^n(D)\rightarrow \C^{m_1}$ and $g:B^n(D)\rightarrow \C^{m_2}$ we consider the equation
\begin{equation}\label{Eqn::Elliptic::LinearReg::MainEqn}
(\sE+\sL)u=g.
\end{equation}

\begin{prop}\label{Prop::Elliptic::LinearReg}
Suppose $\sE$ is elliptic, and fix $\epsilon_0>0$.
There exists $\gamma=\gamma(\sE)>0$ such that if $u$ and $g$ satisfy \cref{Eqn::Elliptic::LinearReg::MainEqn}
and $\LpNorm{\infty}{c_\alpha}[B^n(D)][\M^{m_2\times m_1}]\leq \gamma$, $\forall \alpha$,
then the following holds for all $s>s_0>0$, $\eta\in (0,D)$,
\begin{equation}\label{Eqn::Elliptic::LinearReg::Qual}
\begin{split}
&u\in \ZygSpace{s_0+M}[B^n(D)][\C^{m_1}], g\in \ZygSpace{s}[B^n(D)][\C^{m_2}], c_{\alpha}\in \ZygSpace{s+\epsilon_0}[B^n(D)][\M^{m_2\times n_1}]
\\&\Rightarrow
u\in \ZygSpace{s+M}[B^n(\eta)][\C^{m_1}].
\end{split}
\end{equation}
Moreover we have
\begin{equation}\label{Eqn::Elliptic::LinearReg::Quant}
\ZygNorm{u}{s+M}[B^n(\eta)][\C^{m_1}]\leq C\left( \ZygNorm{g}{s}[B^n(D)][\C^{m_2}] + \ZygNorm{u}{s_0+M}[B^n(D)][\C^{m_1}] \right),
\end{equation}
where $C$ can be chosen to depend only on $s_0$, $s$, $\sE$, $D$, $\eta$, $\epsilon_0$, and upper bounds for
$\ZygNorm{c_{\alpha}}{s+\epsilon_0}[B^n(D)][\M^{m_2\times m_2}]$, $\ZygNorm{u}{s_0+M}[B^n(D)]$,
and $\ZygNorm{g}{s}[B^n(D)][\C^{m_2}]$.
\end{prop}
\begin{proof}[Proof Sketch]
We sketch a proof of \cref{Eqn::Elliptic::LinearReg::Qual} using theory from \cite{TaylorPDEIII}. There are many proofs of this result which are well-known to experts.  We use the theory
from \cite{TaylorPDEIII} because that reference uses Zygmund spaces, while many other references only state results for H\"older spaces with non-integer exponents (even though many of these
proofs can be generalized to Zygmund spaces). 
The quantitative estimate, \cref{Eqn::Elliptic::LinearReg::Quant}, follows by
keeping track of constants in this proof.  For Zygmund spaces, \cite{TaylorPDEIII} uses the notation $C_{*}^s$ instead of $\ZygSpace{s}[\R^n]$--for
this proof,
we use this notation to help the reader make the connection with the results in that book.

Note that if $\gamma=\gamma(\sE)>0$ is sufficiently small, $\sE+\sL$ is uniformly elliptic on $B^n(D)$.
Let $u\in \ZygSpace{s_0+M}[B^n(D)][\C^{m_1}]$, $g\in \ZygSpace{s}[B^n(D)][\C^{m_2}]$, and  $c_{\alpha}\in \ZygSpace{s+\epsilon_0}[B^n(D)][\M^{m_2\times n_1}]$ satisfying \cref{Eqn::Elliptic::LinearReg::MainEqn}.
Fix $\eta\in (0,D)$ and take $\phi_1,\phi_2,\phi_3\in C_0^\infty(B^n(D))$ such that
$\phi_j\equiv 1$ on a neighborhood of the support of $\phi_{j-1}$ and $\phi_1\equiv 1$ on a neighborhood of the closure
of $B^n(\eta)$.
Since $(\sE+\sL)u=g$, we have
\begin{equation}\label{Eqn::Elliptic::LinearReg::LocalizedEqn}
\phi_2(\sE+\sL)\phi_3 u=\phi_2 g.
\end{equation}
Using the notation of Chapter 13, Section 9 of \cite{TaylorPDEIII}, we have $\phi_2(\sE+\sL)=a(x,D)$ where
$a(x,\xi)\in C_{*}^{s+\epsilon_0} S^{M}_{1,0}$.

Set $\delta=\min\left\{\frac{\epsilon_0}{s+\epsilon_0}, \frac{s-s_0}{s+\epsilon_0}\right\}$ so that $\delta\in (0,1)$.
By Proposition 9.9 of Chapter 13 of \cite{TaylorPDEIII},
\begin{equation*}
a(x,\xi)=a^{\sharp}(x,\xi)+a^{\flat}(x,\xi), \quad a^{\sharp}\in S^M_{1,\delta}, \quad a^{\flat}\in C_*^{s+\epsilon_0} S_{1,\delta}^{M-(s+\epsilon_0)\delta}.
\end{equation*}
Note that since $\sE+\sL$ is elliptic on $B^n(D)$, $a$ is elliptic on a neighborhood of the support of $\phi_1$, and the same is therefore
true of $a^{\sharp}$.

Rewriting \cref{Eqn::Elliptic::LinearReg::LocalizedEqn} we have
\begin{equation}\label{Eqn::Elliptic::LinearReg::LocalizedEqn2}
a^{\sharp}(x,D) \phi_3 u = \phi_2 g - a^{\flat}(x,D) \phi_3 u.
\end{equation}
Since $\phi_3 u\in C_{*}^{s_0+M}$, by assumption, Proposition 9.10 of Chapter 13 of \cite{TaylorPDEIII}
implies $a^{\flat}(x,D)\phi_3 u\in C_{*}^{s_0+\min\{\epsilon_0, s-s_0\}}$.
Combining this with $\phi_2 g\in C_{*}^{s}$ we have $a^{\sharp}(x,D) \phi_3 u \in C_{*}^{s_0+\min\{ \epsilon_0, s-s_0 \}}$.

Since $a^{\sharp}$ is elliptic on a neighborhood of the support of $\phi_1$, we conclude
$\phi_1 u \in C_{*}^{s_0+M+\min\{\epsilon_0,s-s_0\}}$, and therefore
$u\in \ZygSpace{s_0+M+\min\{\epsilon_0,s-s_0\}}[B^n(\eta)][\C^{m_1}]$.
\Cref{Eqn::Elliptic::LinearReg::Qual} follows by iterating this result.
\end{proof}

\begin{rmk}
In \cite{TaylorPDEIII} a different (but equivalent) norm is used in the definition of $\ZygSpace{s}[B^n(\eta)]$ (see
\cref{Rmk::ZygSpace::WeirdDefn}).  The constants
in this eqivalence depend on $s$, $n$, and $\eta$.  This does not create a problem in \cref{Prop::Elliptic::LinearReg} since $C$ is allowed to depend on $\sE$ (and therefore on $n$), $s$, $s_0$,
$\eta$, and $D$.
\end{rmk} 
	
	\subsection{Regularity for a Nonlinear Elliptic Equation}
Let $\sE$ be a constant coefficient, first order, partial differential operator,
\begin{equation*}
\sE:\CjSpace{\infty}[\R^n][\C^{m_1}]\rightarrow \CjSpace{\infty}[\R^n][\C^{m_2}],
\end{equation*}
where $m_2\geq m_1$.  We may think of $\sE$ as a $m_2\times m_1$ matrix of
constant coefficient partial differential operators of order $\leq 1$.

Let $\Gamma:\C^{m_1}\times \C^{nm_1}\rightarrow \C^{m_2}$ be a bilinear map.
Fix $D>0$,
we consider the equation, for $b:B^n(D)\rightarrow \C^{m_1}$, $c:B^n(D)\rightarrow \C^{m_2}$,
\begin{equation}\label{Eqn::Ellipitic::Inhomog::MainEqn}
\sE b + \Gamma(b,\grad b)=c.
\end{equation}

\begin{prop}\label{Prop::Appendix::Regularity::MainProp}
Suppose $\sE$ is elliptic.  Then, there exists $\gamma=\gamma(\sE, \Gamma)>0$ such that
if $b$ and $c$ satisfy \cref{Eqn::Ellipitic::Inhomog::MainEqn},
and if for some $s_1,s_2>0$ we have $c\in \ZygSpace{s_2}[B^n(D);\C^{m_2}]$, $b\in \ZygSpace{s_1+1}[B^n(D)][\C^{m_1}]$,
with $\LpNorm{\infty}{b}[B^n(D)][\C^{m_1}]\leq \gamma$, then for all $\eta\in (0,D)$,
$b\in \ZygSpace{s_2+1}[B^N(\eta)][\C^{m_1}]$.  Moreover,
\begin{equation*}
\ZygNorm{b}{s_2+1}[B^n(\eta)][\C^{m_1}]\leq C\left(  \ZygNorm{b}{s_1+1}[B^n(D)][\C^{m_1}] + \ZygNorm{c}{s_2}[B^n(D)][\C^{m_2}]  \right),
\end{equation*}
where $C$ can be chosen to depend only on $s_1$, $s_2$, $D$, $\eta$, $\sE$, $\Gamma$, 
and upper bounds for $\ZygNorm{b}{s_1+1}[B^n(D)][\C^{m_1}]$ and $\ZygNorm{c}{s_2}[B^n(D)][\C^{m_2}]$.
\end{prop}
\begin{proof}
We will show, under the hypotheses of the proposition, that there exists $\gamma=\gamma(\sE,\Gamma)>0$ such that if 
$b$ and $c$ are as in the proposition, we have for $\eta\in (0,D)$,
\begin{equation}\label{Eqn::Elliptic::Inhomog::ToShow1}
b\in \ZygSpace{\min\{ s_1+3/2, s_2+1 \}}[B^n(\eta)][\C^{m_1}],
\end{equation}
and
\begin{equation}\label{Eqn::Elliptic::Inhomog::ToShow2}
\ZygNorm{b}{\min\{s_1+3/2,s_2+1\}}[B^n(\eta)][\C^{m_1}]\leq C\left(  \ZygNorm{b}{s_1+1}[B^n(D)][\C^{m_1}] + \ZygNorm{c}{s_2}[B^n(D)][\C^{m_2}]  \right),
\end{equation}
where $C$ is as in the statement of the proposition.  The result then follows by iteration.

We use \cref{Prop::Elliptic::LinearReg} with $M=1$, $\epsilon_0=\frac{1}{2}$, $s_0=s_1$, and $s=\min\{s_2, s_1+\frac{1}{2}\}$ applied to
\cref{Eqn::Ellipitic::Inhomog::MainEqn}.
With these choices, if $\gamma=\gamma(\sE, \Gamma)>0$ is sufficiently small,  \cref{Prop::Elliptic::LinearReg}
applies to prove \cref{Eqn::Elliptic::Inhomog::ToShow1,Eqn::Elliptic::Inhomog::ToShow2}, completing the proof.
\end{proof}
	
	\subsection{Existence for a Nonlinear Elliptic Equation}\label{Section::ExistNonlinear}
Fix $D>0$, $m_1,m_2\in \N$.  For functions $A:B^n(D)\rightarrow \R^{m_1}$ and $B:B^n(D)\rightarrow \R^{m_2}$
write
\begin{equation*}
\Deriv^1 A=(\partial_x^{\alpha} A)_{|\alpha|\leq 1}, \quad \Deriv^2 B=(\partial_x^{\alpha} B)_{|\alpha|\leq 2},\quad \Deriv_2 B=(\partial_x^{\alpha} B)_{|\alpha|=2}, 
\end{equation*}
so that, for example, $\Deriv^2 B$ is the vector of all partial derivatives of $B$ up to order $2$, and $\Deriv_2 B$ is the vector of all
partial derivatives of $B$ of order exactly $2$.

Fix a $C^\infty$ function $g$.  We wish to consider the equation
\begin{equation}\label{Eqn::Elliptic::Exist::MainEqn}
g(\Deriv^1 A(x), \Deriv^2 B(x))=0.
\end{equation}
Here $g$ is $C^\infty$ and defined on a neighborhood of the origin, takes values in $\R^{m_2}$,
and satisfies $g(0,0)=0$.
Our goal is to give conditions on $g$ so that given $A$ (sufficiently small), we can find $B=B(A)$ so that \cref{Eqn::Elliptic::Exist::MainEqn}
holds; we further wish to understand the regularity properties of $B$ in a quantitative way.

Though it is not necessary for the results that follow, we assume \cref{Eqn::Elliptic::Exist::MainEqn} is quasilinear
in $B$, which is sufficient for our purposes and simplifies the proof.  That is, we assume
\begin{equation}\label{Eqn::Elliptic::Exist::Quasilinear}
g(\Deriv^1 A(x), \Deriv^2 B(x)) = g_1(A(x), \Deriv^1 B(x)) \Deriv_2 B(x) + g_2(\Deriv^1 A(x), \Deriv^1B(x)),
\end{equation}
where $g_1$ and $g_2$ are smooth on a neighborhood of the origin, $g_1$ takes values in matrices of an appropriate size, and $g_2(0,0)=0$.

Finally, let $\sE_2$ denote the second order partial differential operator
\begin{equation}\label{Eqn::Elliptic::Exist::DefnsE2}
	\sE_2 B:=g_1(0,0)\Deriv_2 B,
\end{equation}
so that $\sE_2$ is an $m_2\times m_2$ matrix of constant, real coefficient partial differential operators of order $\leq 2$.

\begin{prop}\label{Prop::Elliptic::Exist::MainProp}
Suppose $\sE_2$ is elliptic.  Fix $s_0>0$ and a neighborhood $N\subseteq \ZygSpace{2+s_0}[B^n(D)][\R^{m_2}]$ of $0$.
Then, there exists a neighborhood $W\subseteq \ZygSpace{1+s_0}[B^n(D)][\R^{m_2}]$ of $0$
and a map
\begin{equation*}
\BofA:W\rightarrow N
\end{equation*}
such that
\begin{equation}\label{Eqn::Ellliptic::Exist::EEqn}
	g(\Deriv^1 A(x), \Deriv^2 \BofA(A)(x))=0, \quad x\in B^n(D), \quad A\in W.
\end{equation}
This map satisfies $\Deriv^1 \BofA(A)(0)=0$, $\forall A\in W$, and
\begin{equation}\label{Eqn::Ellliptic::Exist::BasicNorm}
\ZygNorm{\BofA(A)}{2+s_0}[B^n(D)][\R^{m_2}]\leq C \ZygNorm{A}{1+s_0}[B^n(D)][\R^{m_1}],
\end{equation}
where $C$ does not depend on $A\in W$.  Finally, for $\eta\in (0,D)$, let 
$R_{\eta}$ denote the restriction map $R_\eta: f\mapsto f\big|_{B^{n}(\eta)}$.
Then, for $s\geq s_0$, $\eta\in (0,D)$,
\begin{equation}\label{Eqn::Ellliptic::Exist::RegQual}
R_{\eta} \circ \BofA: \ZygSpace{1+s}[B^n(D)][\R^{m_1}]\cap W\rightarrow \ZygSpace{2+s}[B^n(\eta)][\R^{m_2}],
\end{equation}
and
\begin{equation}\label{Eqn::Ellliptic::Exist::RegQuant}
\ZygNorm{R_{\eta} \circ \BofA(A)}{2+s}[B^n(\eta)][\R^{m_2}] \leq C_{s,\eta},
\end{equation}
where $C_{s,\eta}$ can be chosen to depend on an upper bound for $\ZygNorm{A}{1+s}[B^n(D)][\R^{m_1}]$ and does not
depend on $A\in W$ in any other way.  It can depend on any of the other ingredients in the problem.
\end{prop}

The rest of this section is devoted to a sketch of a proof of \cref{Prop::Elliptic::Exist::MainProp}.
The proof is a standard application of the Inverse Function Theorem combined with \cref{Prop::Elliptic::LinearReg};
we include the proof as it gives the required quantitative estimates, which are essential for our purposes.

By expanding $g$ into a Taylor series, we have
\begin{equation*}
g(\Deriv^1 A, \Deriv^2 B)= \ADeriv A+ \sE B + q(\Deriv^1 A, \Deriv^2 B),
\end{equation*}
where $\ADeriv$ is a first order linear differential operator with constant coefficients,
$\sE$ is a second order linear differential operator with constant coefficients whose principal
symbol is $\sE_2$, and $q$ is smooth and vanishes to second order at $(0,0)$.

Since $\sE$ is elliptic (because $\sE_2$ is), it is a standard fact that $\sE$ has a continuous right inverse
\begin{equation*}
\sP: \ZygSpace{s_0}[B^n(D)][\R^{m_2}]\rightarrow \ZygSpace{2+s_0}[B^n(D)][\R^{m_2}],
\end{equation*}
where $\sE \sP = I$ and for all $|\alpha|\leq 1$, $\partial_x^{\alpha} \sP(B)(x)\big|_{x=0}=0$.

Set
\begin{equation*}
F(A,B)(x):=\left(A(x), g(\Deriv^1 A(x), \Deriv^2[-\sP \ADeriv A+ \sP B](x))\right).
\end{equation*}
Fix  (small) open neighborhoods $N_0, U_0\subseteq \ZygSpace{1+s_0}[B^n(D)][\R^{m_1}]\times \ZygSpace{s_0}[B^n(D)][\R^{m_2}]$ of $(0,0)$, to be chosen later.
We take $U_0=U_0(N_0)$ small enough that $F:U_0\rightarrow N_0$.

\begin{lemma}\label{Lemma::Elliptic::Exist::IFT}
There exists an open neighborhood $W_0\subseteq N_0$ of $(0,0)$ and a map
$G:W_0\rightarrow U_0$ such that $F(G(A,B))=(A,B)$ and
\begin{equation}\label{Eqn::Elliptic::Exist::IFT::Reg}
\Norm{G(A,B)}[\ZygSpace{1+s_0}[B^n(D)][\R^{m_1}]\times \ZygSpace{s_0}[B^n(D)][\R^{m_2}] ]\leq
C \Norm{(A,B)}[\ZygSpace{1+s_0}[B^n(D)][\R^{m_1}]\times \ZygSpace{s_0}[B^n(D)][\R^{m_2}] ],
\end{equation}
where $C$ does not depend on the choice of $(A,B)\in W_0$.
\end{lemma}
\begin{proof}
It is clear that $F$ is a $C^1$ map $F:U_0\subseteq \ZygSpace{1+s_0}\times \ZygSpace{s_0}\rightarrow N_0\subseteq \ZygSpace{1+s_0}\times \ZygSpace{s_0}$
with $F(0)=0$ and $dF(0)=I$.  The lemma now follows from the usual Inverse Function Theorem on Banach spaces.
\end{proof}

Let $W_0$ be as in \cref{Lemma::Elliptic::Exist::IFT} and set $W:=\{ A : (A,0)\in W_0\}$.  Note that $W\subseteq \ZygSpace{1+s_0}[B^n(D)][\R^{m_2}]$
is an open neighborhood of $0$.  Taking $G$ as in \cref{Lemma::Elliptic::Exist::IFT} it is easy to see that $G$ is of the form
$G(A,B)=(A, \Gt(A,B))$.  We set
\begin{equation*}
\BofA(A):=-\sP \ADeriv A + \sP \Gt(A,0).
\end{equation*}
It is clear that $\BofA$ satisfies \cref{Eqn::Ellliptic::Exist::EEqn}.
By taking $N_0$ small, we may take $U_0$ and $W$ as small as we like.
Thus, because the range of $G$ is contained in $U_0$, if $N_0$, $U_0$, and $W$ are chosen to be sufficiently small we have $\BofA:W\rightarrow N$.
Furthermore, by the choice of $\sP$ we have $\Deriv^1 \BofA(A)(0)=0$.  Also, \cref{Eqn::Ellliptic::Exist::BasicNorm}
follows from \cref{Eqn::Elliptic::Exist::IFT::Reg} and the continuity of $\sP$.

It remains to prove \cref{Eqn::Ellliptic::Exist::RegQual,Eqn::Ellliptic::Exist::RegQuant}.  For this, we use that we have the flexibility
to take $U_0$ and $N_0$ as small as we like (though they must be chosen independent of $s$).

Let $\gamma=\gamma(\sE_2)>0$ be as in \cref{Prop::Elliptic::LinearReg}.
By taking $N_0$ and $U_0$ sufficiently small, and using the fact that $g_1$ is smooth, we have
for $A\in W$, every coefficient  of the differential operator $\sL:=(g_1(A(x), \Deriv^1 \BofA(A)(x))-g_1(0,0))\Deriv_2$ has $L^\infty$
norm $\leq \gamma$; indeed, since $W_0\subseteq N_0$, taking $N_0$ small forces $W_0$, and therefore $W$, to be a small neighborhood of $0$.
Setting $B=\BofA(A)$, we will apply \cref{Prop::Elliptic::LinearReg} (with $u=B$) to the equation
\begin{equation}\label{Eqn::Elliptic::Exist::NewEqn}
(\sE_2+\sL) B = g_1(A(x), \Deriv^1B(x)) \Deriv_2 B(x) = -g_2(\Deriv^1A(x), \Deriv^1B(x)).
\end{equation}

Let $0<\eta_1<\eta_2\leq D$.  We will show for $s>s_2\geq s_0$, $A\in W$, $B=\BofA(A)$,
\begin{equation}\label{Eqn::Elliptic::Exist::BootstrapQual}
A\in \ZygSpace{1+s}[B^n(D)][\R^{m_1}], B\in \ZygSpace{2+s_2}[B^n(\eta_2)][\R^{m_2}]
\Rightarrow B\in \ZygSpace{2+s_2+\min\{ \frac{1}{2}, s-s_2 \}}[B^n(\eta_1)][\R^{m_2}],
\end{equation}
with
\begin{equation}\label{Eqn::Elliptic::Exist::BootstrapQuant}
\ZygNorm{B}{2+s_2+\min\{\frac{1}{2},s-s_2\}}[B^n(\eta_1)][\R^{m_2}]\leq C_{s,s_2,\eta_1,\eta_2},
\end{equation}
where $C_{s,s_2,\eta_1,\eta_2}$ can be chosen to depend on
$\ZygNorm{A}{1+s}[B^n(D)][\R^{m_1}]$ and $\ZygNorm{B}{2+s_2}[B^{n}(\eta_2)][\R^{m_2}]$, but not depend on
$A$ or $B$ in any other way.  It can depend on any other ingredient in the problem.
\cref{Eqn::Ellliptic::Exist::RegQual,Eqn::Ellliptic::Exist::RegQuant} follow from \cref{Eqn::Elliptic::Exist::BootstrapQual,Eqn::Elliptic::Exist::BootstrapQuant}
via a simple iteration.  Thus we prove \cref{Eqn::Elliptic::Exist::BootstrapQual,Eqn::Elliptic::Exist::BootstrapQuant} which will complete the proof.

Since $g_1$ and $g_2$ are smooth, if $A\in \ZygSpace{1+s}$ and $B\in \ZygSpace{2+s_2}$,
we have
$g_1(A, \Deriv^1B)\in \ZygSpace{s_2+1}$ and $g_2(\Deriv^1 A, \Deriv^1 B)\in \ZygSpace{\min\{s, s_2+1\}}\subseteq \ZygSpace{\min\{s, s_2+\frac{1}{2}\}}$ (see \cref{Lemma::ZygSpace::Comp}).
Furthermore, we have 
\begin{equation}\label{Eqn::Elliptic::Exist::BoundG1G2}
\ZygNorm{g_1(A, \Deriv^1B)}{s_2+1},\ZygNorm{g_2(\Deriv^1 A, \Deriv^1 B)}{\min\{s,s_2+\frac{1}{2}\}}\leq C_{s,s_2,\eta_1,\eta_2},
\end{equation}
where $C_{s,s_2,\eta_1,\eta_2}$ is as above; in particular, the estimate on $g_1(A, \Deriv^1B)$ in \cref{Eqn::Elliptic::Exist::BoundG1G2} shows that the coefficients of $\sL$ are in $\ZygSpace{s_2+1}$ with $\ZygSpace{s_2+1}$ norms
bounded by $C_{s,s_2,\eta_1,\eta_2}$, where $C_{s,s_2,\eta_1,\eta_2}$ is as above. 
Applying \cref{Prop::Elliptic::LinearReg} to \cref{Eqn::Elliptic::Exist::NewEqn} with $M=2$, $s_0=s_2$, $s=\min\{s, s_2+\frac{1}{2}\}$, and $\epsilon_0=\frac{1}{2}$, and using the estimate on $g_2(\Deriv^1 A, \Deriv^1 B)$ in \cref{Eqn::Elliptic::Exist::BoundG1G2},
 \cref{Eqn::Elliptic::Exist::BootstrapQual,Eqn::Elliptic::Exist::BootstrapQuant} follow, completing the proof.

	\subsection{An Elliptic Operator}
In this section, we discuss a particular first order, overdetermined, constant coefficient, linear, elliptic operator which is needed in this paper.
For a function $A=(A_1,\ldots, A_n)\in \CjSpace{\infty}[\R^n][\R^n]$ we define
\begin{equation*}
\sE A:= \left( \left(\diff{t_j} A_k - \diff{t_k} A_j \right)_{1\leq j<k\leq n}, \sum_{j=1}^n \diff{t_j} A_j  \right).
\end{equation*}

\begin{lemma}\label{Lemma::Elliptic::Ex}
$\sE$ is elliptic.
\end{lemma}
\begin{proof}
It is easy to compute $\sE^{*}\sE$ directly to see
\begin{equation*}
\sE^{*} \sE A = -\sum_{j=1}^n \frac{\partial^2}{\partial t_j^2} A,
\end{equation*}
and the result follows.
\end{proof}

A more abstract way to see \cref{Lemma::Elliptic::Ex} is as follows.  We identify $A$ with the $1$-form
$A=A_1 dt_1+A_2dt_2+\cdots +A_n dt_n$.  Then,
\begin{equation*}
dA = \sum_{1\leq j<k\leq n}  \left(\diff{t_j} A_k - \diff{t_k} A_j \right) dt_j\wedge dt_k, \quad \delta A = -\sum_{j=1}^n \diff{t_j} A_j,
\end{equation*}
where $\delta$ denotes the codifferential on $\R^n$.
Hence, $\sE$ can be written as $\sE A = (d A, -\delta A)$, and therefore
$\sE^{*} \sE = d\delta+\delta d$.  I.e., $\sE^{*}\sE$ is the Laplace--de Rham operator acting on $1$-forms,
and is therefore elliptic.



\bibliographystyle{amsalpha}

\bibliography{coords}

\center{\it{University of Wisconsin-Madison, Department of Mathematics, 480 Lincoln Dr., Madison, WI, 53706}}

\center{\it{street@math.wisc.edu}}

\center{MSC 2010:  58A30 (Primary), 57R55 and 53C17 (Secondary)}

\end{document}